\documentclass[a4paper, 11pt]{preprint}

\usepackage[T1]{fontenc} % Font encoding
\usepackage[utf8]{inputenc} % Input encoding

\usepackage[english]{babel} % Localization to English

\usepackage{amsmath, amssymb, mathtools} % AMS suite
\usepackage{amsthm}

\usepackage[colorinlistoftodos,prependcaption,color=yellow,textsize=tiny,textwidth=2cm]{todonotes}

\makeatletter
\renewcommand{\@todonotes@drawMarginNoteWithLine}{%
  \ifvmode
    \vspace*{-\parskip}%
    \vskip-\baselineskip
    \noindent
  \fi
  \begin{tikzpicture}[remember picture, overlay, baseline=-0.75ex]%
    \node [coordinate] (inText) {};%
  \end{tikzpicture}%
  \marginpar{%
    \@todonotes@drawMarginNote%
    \@todonotes@drawLineToRightMargin%
  }%
}
\makeatother

\usepackage[osf]{newtxtext} % Set text font to a Times close, osf = old style serifs

\usepackage{mathrsfs} % Mathscr font

\usepackage{microtype} % Microtypography

\usepackage{placeins}

\usepackage{tikz}
\usetikzlibrary{positioning}
\usetikzlibrary{calc}

\usepackage{pgfplots}
\usepackage{pgfplotstable}
\pgfplotsset{compat=1.18}
\usepackage{standalone}

\definecolor{adamc}{RGB}{213,94,0}
\definecolor{gnc}{RGB}{0,114,178}
\definecolor{ngc}{RGB}{230,159,0}
\definecolor{sgdc}{RGB}{0,158,115}

\usepgfplotslibrary{fillbetween}
\usepgfplotslibrary{groupplots}
\usetikzlibrary{pgfplots.groupplots}

\pgfplotsset{
    colormap={cividis}{
        rgb=(0.0000,0.1262,0.3015) rgb=(0.0000,0.1292,0.3077) rgb=(0.0000,0.1321,0.3142)
        rgb=(0.0000,0.1350,0.3205) rgb=(0.0000,0.1379,0.3269) rgb=(0.0000,0.1408,0.3334)
        rgb=(0.0000,0.1437,0.3400) rgb=(0.0000,0.1465,0.3467) rgb=(0.0000,0.1492,0.3537)
        rgb=(0.0000,0.1519,0.3606) rgb=(0.0000,0.1546,0.3676) rgb=(0.0000,0.1574,0.3746)
        rgb=(0.0000,0.1601,0.3817) rgb=(0.0000,0.1629,0.3888) rgb=(0.0000,0.1657,0.3960)
        rgb=(0.0000,0.1685,0.4031) rgb=(0.0000,0.1714,0.4102) rgb=(0.0000,0.1743,0.4172)
        rgb=(0.0000,0.1773,0.4241) rgb=(0.0000,0.1798,0.4307) rgb=(0.0000,0.1817,0.4347)
        rgb=(0.0000,0.1834,0.4363) rgb=(0.0000,0.1852,0.4368) rgb=(0.0000,0.1872,0.4368)
        rgb=(0.0000,0.1901,0.4365) rgb=(0.0000,0.1930,0.4361) rgb=(0.0000,0.1958,0.4356)
        rgb=(0.0000,0.1987,0.4349) rgb=(0.0000,0.2015,0.4343) rgb=(0.0000,0.2044,0.4336)
        rgb=(0.0000,0.2073,0.4329) rgb=(0.0055,0.2101,0.4322) rgb=(0.0236,0.2130,0.4314)
        rgb=(0.0416,0.2158,0.4308) rgb=(0.0576,0.2187,0.4301) rgb=(0.0710,0.2215,0.4293)
        rgb=(0.0827,0.2244,0.4287) rgb=(0.0932,0.2272,0.4280) rgb=(0.1030,0.2300,0.4274)
        rgb=(0.1120,0.2329,0.4268) rgb=(0.1204,0.2357,0.4262) rgb=(0.1283,0.2385,0.4256)
        rgb=(0.1359,0.2414,0.4251) rgb=(0.1431,0.2442,0.4245) rgb=(0.1500,0.2470,0.4241)
        rgb=(0.1566,0.2498,0.4236) rgb=(0.1630,0.2526,0.4232) rgb=(0.1692,0.2555,0.4228)
        rgb=(0.1752,0.2583,0.4224) rgb=(0.1811,0.2611,0.4220) rgb=(0.1868,0.2639,0.4217)
        rgb=(0.1923,0.2667,0.4214) rgb=(0.1977,0.2695,0.4212) rgb=(0.2030,0.2723,0.4209)
        rgb=(0.2082,0.2751,0.4207) rgb=(0.2133,0.2780,0.4205) rgb=(0.2183,0.2808,0.4204)
        rgb=(0.2232,0.2836,0.4203) rgb=(0.2281,0.2864,0.4202) rgb=(0.2328,0.2892,0.4201)
        rgb=(0.2375,0.2920,0.4200) rgb=(0.2421,0.2948,0.4200) rgb=(0.2466,0.2976,0.4200)
        rgb=(0.2511,0.3004,0.4201) rgb=(0.2556,0.3032,0.4201) rgb=(0.2599,0.3060,0.4202)
        rgb=(0.2643,0.3088,0.4203) rgb=(0.2686,0.3116,0.4205) rgb=(0.2728,0.3144,0.4206)
        rgb=(0.2770,0.3172,0.4208) rgb=(0.2811,0.3200,0.4210) rgb=(0.2853,0.3228,0.4212)
        rgb=(0.2894,0.3256,0.4215) rgb=(0.2934,0.3284,0.4218) rgb=(0.2974,0.3312,0.4221)
        rgb=(0.3014,0.3340,0.4224) rgb=(0.3054,0.3368,0.4227) rgb=(0.3093,0.3396,0.4231)
        rgb=(0.3132,0.3424,0.4236) rgb=(0.3170,0.3453,0.4240) rgb=(0.3209,0.3481,0.4244)
        rgb=(0.3247,0.3509,0.4249) rgb=(0.3285,0.3537,0.4254) rgb=(0.3323,0.3565,0.4259)
        rgb=(0.3361,0.3593,0.4264) rgb=(0.3398,0.3622,0.4270) rgb=(0.3435,0.3650,0.4276)
        rgb=(0.3472,0.3678,0.4282) rgb=(0.3509,0.3706,0.4288) rgb=(0.3546,0.3734,0.4294)
        rgb=(0.3582,0.3763,0.4302) rgb=(0.3619,0.3791,0.4308) rgb=(0.3655,0.3819,0.4316)
        rgb=(0.3691,0.3848,0.4322) rgb=(0.3727,0.3876,0.4331) rgb=(0.3763,0.3904,0.4338)
        rgb=(0.3798,0.3933,0.4346) rgb=(0.3834,0.3961,0.4355) rgb=(0.3869,0.3990,0.4364)
        rgb=(0.3905,0.4018,0.4372) rgb=(0.3940,0.4047,0.4381) rgb=(0.3975,0.4075,0.4390)
        rgb=(0.4010,0.4104,0.4400) rgb=(0.4045,0.4132,0.4409) rgb=(0.4080,0.4161,0.4419)
        rgb=(0.4114,0.4189,0.4430) rgb=(0.4149,0.4218,0.4440) rgb=(0.4183,0.4247,0.4450)
        rgb=(0.4218,0.4275,0.4462) rgb=(0.4252,0.4304,0.4473) rgb=(0.4286,0.4333,0.4485)
        rgb=(0.4320,0.4362,0.4496) rgb=(0.4354,0.4390,0.4508) rgb=(0.4388,0.4419,0.4521)
        rgb=(0.4422,0.4448,0.4534) rgb=(0.4456,0.4477,0.4547) rgb=(0.4489,0.4506,0.4561)
        rgb=(0.4523,0.4535,0.4575) rgb=(0.4556,0.4564,0.4589) rgb=(0.4589,0.4593,0.4604)
        rgb=(0.4622,0.4622,0.4620) rgb=(0.4656,0.4651,0.4635) rgb=(0.4689,0.4680,0.4650)
        rgb=(0.4722,0.4709,0.4665) rgb=(0.4756,0.4738,0.4679) rgb=(0.4790,0.4767,0.4691)
        rgb=(0.4825,0.4797,0.4701) rgb=(0.4861,0.4826,0.4707) rgb=(0.4897,0.4856,0.4714)
        rgb=(0.4934,0.4886,0.4719) rgb=(0.4971,0.4915,0.4723) rgb=(0.5008,0.4945,0.4727)
        rgb=(0.5045,0.4975,0.4730) rgb=(0.5083,0.5005,0.4732) rgb=(0.5121,0.5035,0.4734)
        rgb=(0.5158,0.5065,0.4736) rgb=(0.5196,0.5095,0.4737) rgb=(0.5234,0.5125,0.4738)
        rgb=(0.5272,0.5155,0.4739) rgb=(0.5310,0.5186,0.4739) rgb=(0.5349,0.5216,0.4738)
        rgb=(0.5387,0.5246,0.4739) rgb=(0.5425,0.5277,0.4738) rgb=(0.5464,0.5307,0.4736)
        rgb=(0.5502,0.5338,0.4735) rgb=(0.5541,0.5368,0.4733) rgb=(0.5579,0.5399,0.4732)
        rgb=(0.5618,0.5430,0.4729) rgb=(0.5657,0.5461,0.4727) rgb=(0.5696,0.5491,0.4723)
        rgb=(0.5735,0.5522,0.4720) rgb=(0.5774,0.5553,0.4717) rgb=(0.5813,0.5584,0.4714)
        rgb=(0.5852,0.5615,0.4709) rgb=(0.5892,0.5646,0.4705) rgb=(0.5931,0.5678,0.4701)
        rgb=(0.5970,0.5709,0.4696) rgb=(0.6010,0.5740,0.4691) rgb=(0.6050,0.5772,0.4685)
        rgb=(0.6089,0.5803,0.4680) rgb=(0.6129,0.5835,0.4673) rgb=(0.6168,0.5866,0.4668)
        rgb=(0.6208,0.5898,0.4662) rgb=(0.6248,0.5929,0.4655) rgb=(0.6288,0.5961,0.4649)
        rgb=(0.6328,0.5993,0.4641) rgb=(0.6368,0.6025,0.4632) rgb=(0.6408,0.6057,0.4625)
        rgb=(0.6449,0.6089,0.4617) rgb=(0.6489,0.6121,0.4609) rgb=(0.6529,0.6153,0.4600)
        rgb=(0.6570,0.6185,0.4591) rgb=(0.6610,0.6217,0.4583) rgb=(0.6651,0.6250,0.4573)
        rgb=(0.6691,0.6282,0.4562) rgb=(0.6732,0.6315,0.4553) rgb=(0.6773,0.6347,0.4543)
        rgb=(0.6813,0.6380,0.4532) rgb=(0.6854,0.6412,0.4521) rgb=(0.6895,0.6445,0.4511)
        rgb=(0.6936,0.6478,0.4499) rgb=(0.6977,0.6511,0.4487) rgb=(0.7018,0.6544,0.4475)
        rgb=(0.7060,0.6577,0.4463) rgb=(0.7101,0.6610,0.4450) rgb=(0.7142,0.6643,0.4437)
        rgb=(0.7184,0.6676,0.4424) rgb=(0.7225,0.6710,0.4409) rgb=(0.7267,0.6743,0.4396)
        rgb=(0.7308,0.6776,0.4382) rgb=(0.7350,0.6810,0.4368) rgb=(0.7392,0.6844,0.4352)
        rgb=(0.7434,0.6877,0.4338) rgb=(0.7476,0.6911,0.4322) rgb=(0.7518,0.6945,0.4307)
        rgb=(0.7560,0.6979,0.4290) rgb=(0.7602,0.7013,0.4273) rgb=(0.7644,0.7047,0.4258)
        rgb=(0.7686,0.7081,0.4241) rgb=(0.7729,0.7115,0.4223) rgb=(0.7771,0.7150,0.4205)
        rgb=(0.7814,0.7184,0.4188) rgb=(0.7856,0.7218,0.4168) rgb=(0.7899,0.7253,0.4150)
        rgb=(0.7942,0.7288,0.4129) rgb=(0.7985,0.7322,0.4111) rgb=(0.8027,0.7357,0.4090)
        rgb=(0.8070,0.7392,0.4070) rgb=(0.8114,0.7427,0.4049) rgb=(0.8157,0.7462,0.4028)
        rgb=(0.8200,0.7497,0.4007) rgb=(0.8243,0.7532,0.3984) rgb=(0.8287,0.7568,0.3961)
        rgb=(0.8330,0.7603,0.3938) rgb=(0.8374,0.7639,0.3915) rgb=(0.8417,0.7674,0.3892)
        rgb=(0.8461,0.7710,0.3869) rgb=(0.8505,0.7745,0.3843) rgb=(0.8548,0.7781,0.3818)
        rgb=(0.8592,0.7817,0.3793) rgb=(0.8636,0.7853,0.3766) rgb=(0.8681,0.7889,0.3739)
        rgb=(0.8725,0.7926,0.3712) rgb=(0.8769,0.7962,0.3684) rgb=(0.8813,0.7998,0.3657)
        rgb=(0.8858,0.8035,0.3627) rgb=(0.8902,0.8071,0.3599) rgb=(0.8947,0.8108,0.3569)
        rgb=(0.8992,0.8145,0.3538) rgb=(0.9037,0.8182,0.3507) rgb=(0.9082,0.8219,0.3474)
        rgb=(0.9127,0.8256,0.3442) rgb=(0.9172,0.8293,0.3409) rgb=(0.9217,0.8330,0.3374)
        rgb=(0.9262,0.8367,0.3340) rgb=(0.9308,0.8405,0.3306) rgb=(0.9353,0.8442,0.3268)
        rgb=(0.9399,0.8480,0.3232) rgb=(0.9444,0.8518,0.3195) rgb=(0.9490,0.8556,0.3155)
        rgb=(0.9536,0.8593,0.3116) rgb=(0.9582,0.8632,0.3076) rgb=(0.9628,0.8670,0.3034)
        rgb=(0.9674,0.8708,0.2990) rgb=(0.9721,0.8746,0.2947) rgb=(0.9767,0.8785,0.2901)
        rgb=(0.9814,0.8823,0.2856) rgb=(0.9860,0.8862,0.2807) rgb=(0.9907,0.8901,0.2759)
        rgb=(0.9954,0.8940,0.2708) rgb=(1.0000,0.8979,0.2655) rgb=(1.0000,0.9018,0.2600)
        rgb=(1.0000,0.9057,0.2593) rgb=(1.0000,0.9094,0.2634) rgb=(1.0000,0.9131,0.2680)
        rgb=(1.0000,0.9169,0.2731)
    }
}

\usepackage{mhequ} % Martin's custom equations, automatic labelling
\usepackage{comment} % Martin's comments package, also works inside equations

\usepackage[colorlinks=true, linkcolor=colorLink, citecolor=colorCite, urlcolor=colorLink, linktocpage=true]{hyperref}
\usepackage{cleveref}

\crefname{lemma}{lemma}{lemmas}
\Crefname{lemma}{Lemma}{Lemmas}

\usepackage{orcidlink} % ORCID symbol
\usepackage{multirow} % Multi-row table cells
\usepackage{booktabs} % Professional table rules

\definecolor{colorLink}{RGB}{0,100,162} % Colour of hyperref urls
\definecolor{colorCite}{RGB}{8,124,100} % Colour of citations
\theoremstyle{plain}
\newtheorem{theorem}{Theorem}[section]

\newtheorem{lemma}[theorem]{Lemma}

\newtheorem*{theorem*}{Theorem}
\newtheorem*{definition*}{Definition}

\theoremstyle{definition}
\newtheorem{definition}[theorem]{Definition}

\newtheorem{remark}[theorem]{Remark}

\numberwithin{equation}{section}
\renewcommand{\theequation}{\thesection.\arabic{equation}}

\global\long\def\D{\mathrm{d}}%
\global\long\def\dif{\,\mathrm{d}}%
\global\long\def\R{\mathbb{R}}%
\global\long\def\P{\mathbb{P}}%
\global\long\def\Q{\mathbb{Q}}%
\global\long\def\E{\mathbb{E}}%
\global\long\def\KL{\textup{KL}}%
\global\long\def\FR{\textup{FR}}%
\global\long\def\AFR{\textup{AFR}}%
\global\long\def\ddh{\left.\frac{\D}{\D h}\right|_{h=0}}%
\global\long\def\V{L^{1}_{T}C^{1}_{\textup{Lip}}}%

\colorlet{darkblue}{blue!90!black}
\colorlet{darkgreen}{green!50!black}
\let\comm\comment

\begin{document}
\title{The Advective Fisher--Rao Geometry
of Deterministic Measure Transport
}
\author{
	Benjamin Gess\textsuperscript{$1$},
	Johannes M\"{u}ller\textsuperscript{$2,$\orcidlink{0000-0001-8729-0466}}
	}
\institute{
	Institut f\"{u}r Mathematik, Technische Universit\"{a}t Berlin \& Max-Planck-Institut f\"{u}r Mathematik in den Naturwissenschaften,
	Email: \href{mailto:benjamin.gess@tu-berlin.de}{\color{black}\texttt{benjamin.gess@tu-berlin.de}}
	\and
	Institut f\"{u}r Mathematik, Technische Universit\"{a}t Berlin, Email: \href{mailto:johannes.christoph.mueller@tu-berlin.de}{\color{black} \texttt{johannes.christoph.mueller@tu-berlin.de}}
	}
\maketitle
\begin{abstract}
    A novel \emph{advective Fisher--Rao metric} is introduced for optimization tasks on paths of probability measures governed by the continuity equation. This metric is shown to lead to optimal descent directions.
    
    It is then shown that this metric arises naturally from three different perspectives: As the rescaled zero-noise limit of the Fisher--Rao metric on path measures, as the expected value of the second variation of the Freidlin--Wentzell large deviation rate functional, and as the Hessian of the Benamou--Brenier action functional from dynamic optimal transport.
    
    We supplement this geometric construction with computational experiments. Here, we demonstrate empirically that the advective Fisher--Rao metric yields the desired optimal fitting of probability densities, whereas the Gauss--Newton method yields optimal fitting of velocity fields.
    
	\vspace{.5em}
\end{abstract}

\setcounter{tocdepth}{2}
\tableofcontents

\section{Introduction}
At the heart of a variety of problems lies the task of matching a reference path of probability measures $\rho_\bullet^\star = (\rho^\star_t)_{t\in[0, 1]}$, which means that one seeks to solve the problem
\begin{equ}\label{eq:matching-paths}
    \min_{\rho_\bullet\in \textup{AC}_T(\mathcal P_2(\R^d))} D(\rho_\bullet^\star, \rho_\bullet),
\end{equ}
where
$\textup{AC}_T(\mathcal P_2(\R^d))$ denotes the set of absolutely continuous paths of probability measures with respect to the Wasserstein-2 distance and
$D$ is some measure of distance.
For example, this problem arises in optimal control of densities \cite{annunziato2010optimal} including dynamic optimal transport and its entropic regularization in the form of Schr\"odinger bridges \cite{benamou2000computational, leonard2012schrodinger}. 
Recently, such control problems on the space of probability measures have attracted a great deal of attention in generative machine learning in the form of flow-matching and diffusion models \cite{Sohl-Dickstein2015, Ho2020, rombach2022high, lipman2023flow} as well as for approximate sampling methods~\cite{berner2024an, domingo2024stochastic, blessing2025trust}, see also \Cref{subsubsec:examples} for a more detailed discussion of problems of this kind. 
To approximately solve this problem, it is common to employ gradient-based optimization methods to minimize the discrepancy $D$.
This requires the choice of a Riemannian metric on the search space $\textup{AC}_T(\mathcal P_2(\R^d))$, where the choice of metric has a significant influence on the optimization dynamics.
A particularly favorable scenario arises if the gradient flow of $D$ with respect to $g$ drives $\rho_\bullet$ along a geodesic towards the reference path $\rho_\bullet^\star$, since it enables a global description of the optimization dynamics as well as fast global exponential convergence rates.
In this case, we call the metric $g$ \emph{compatible} with $D$.
In practice, the problem \eqref{eq:matching-paths} is typically not solved directly, but reformulated as an optimization problem on the space of velocity fields $v_\bullet$ that control the evolution of the density path $\rho_\bullet$ via \eqref{eq:introduction-CE} below.
This is justified by the fact that any absolutely continuous curve  $\rho_\bullet^v\in\textup{AC}_T(\mathcal P_2(\R^d))$ is induced by a time-dependent velocity field $v_\bullet = (v_t)_{t\in[0,1]}$ via the continuity equation
\begin{equ}\label{eq:introduction-CE}
    \partial_t \rho_t^v + \nabla\cdot(\rho_t^v v_t)=0, \qquad \rho_0^v=\rho_0.
\end{equ}
Assuming that the reference path $\rho^\star_\bullet$ is induced by  $v_\bullet^\star$, it is a common approach to consider the least-squares functional
\begin{equ}\label{eq:least-squares-objective}
    E(v_\bullet) = \frac12\int_0^1\int_{\mathbb{R}^d}\|v_t(x)-v_t^\star(x)\|^2 \rho_t^\star(\D x)\dif t
\end{equ}
as a proxy for the original problem~\eqref{eq:matching-paths}.
Importantly, the solution map
\begin{equ}
    \Phi\colon L^1_T C^1_{\textup{Lip}}(\R^d),\quad v_\bullet \mapsto \rho_\bullet^v
\end{equ}
of the continuity equation is neither linear nor injective, nor does it have an injective derivative.
As a consequence, optimization in the search space of velocity fields is not equivalent to the space of paths of probability densities.
The fact that the ultimate goal is to optimize the path of densities $p_\bullet^v$ rather than the velocity field $v_\bullet$ raises a natural question:
\begin{quote}
    \emph{
    Is there a Riemannian metric $g$ with respect to which the gradient flow of $E$ drives $\rho_\bullet^v$ along a geodesic towards the reference path $\rho_\bullet^\star$?
    }
\end{quote}
This question is answered positively in this work, based on the introduction of a novel geometry $g$, the \textit{advective Fisher--Rao metric}, which is shown to induce optimal gradient descent directions on the level of paths of probability measures.
Notably, this metric differs from the $L^2$ geometry induced by \eqref{eq:least-squares-objective} on velocities, and the latter does not lead to such optimal descent directions.
\subsection{Contributions and main results}
In this section, we first introduce the advective Fisher--Rao metric, which leads to optimal gradient descent directions on the level of paths of probability measures.
We discuss how this metric appears as a canonical object from three different perspectives, as well as practical implications for the training of generative models.

\subsubsection{The advective Fisher--Rao metric and its optimality}
As laid out in the introduction, we are looking for a metric on the paths of probability measures induced by deterministic advection \eqref{eq:introduction-CE}, and its corresponding flow
\begin{equ}\label{eq:intro-SDE}
    \dot X_t^v = v_t(X_t^v), \qquad X_0 \sim \rho_0.
\end{equ}
In the setting of diffusive convection with diffusivity $\varepsilon >0$, that is,
\begin{equ}\label{eq:intro-SDE-2}
 \mathrm{d}X_t^{v,\varepsilon}=v_t(X_t^{v,\varepsilon})\,\mathrm{d}t+\sqrt{\varepsilon}\,\mathrm{d}W_t,\qquad X_0^{\varepsilon}\sim \rho_0,
\end{equ}
with the corresponding Fokker-Planck equation
\begin{equ}
    \partial_t \rho_t^{v,\varepsilon} + \nabla\cdot(\rho_t^{v,\varepsilon} v_t) = \frac{\varepsilon}{2} \Delta \rho_t^{v,\varepsilon}, \qquad \rho_0^{\varepsilon}=\rho_0,
\end{equ}
the Fisher--Rao metric on the space of path measures offers a canonical choice. Indeed, since for $\varepsilon >0$, the  path measures $\P^{v,\varepsilon}=\mathcal{L}(X^{v,\varepsilon}_\cdot)$ for different choices of $v$ are equivalent, the Fisher--Rao metric is well-defined, and induces geodesic descent directions \cite{amari2008information}. 
However, for $\varepsilon =0$ the path measures are typically singular, and the Fisher--Rao metric is not defined.
An important insight of this section is that, in the noisy regime $\varepsilon >0$, the pullback of the Fisher--Rao metric on path measures along the solution map $v\mapsto \P^{v,\varepsilon}$ induces a natural geometry on the space of velocity fields. 
When expressed in terms of velocity fields and rescaled appropriately, this pulled-back Fisher--Rao metric admits a nontrivial singular limit as $\varepsilon \to 0$. 
This limiting structure motivates the definition of the advective Fisher--Rao metric for deterministic measure-valued flows. We then show that the resulting metric yields gradient descent directions that are optimal at the level of the induced density paths.
As a first step towards the definition of a compatible metric in the absence of noise, we give an explicit expression for the pullback of the Fisher--Rao metric along the solution map of the SDE \eqref{eq:intro-SDE-2}.

\begin{theorem*}[\Cref{thm:FR-SDE}]
Let $\varepsilon>0$ be the noise level and let $\P^{v,\varepsilon}$ denote the path measure of the SDE \eqref{eq:intro-SDE}. 
Then, under a Novikov condition on $v_\bullet$, $u_\bullet$, and $w_\bullet$, the pullback of the Fisher--Rao metric along the map $v_\bullet\mapsto \P^{v,\varepsilon}$ is given by
\begin{equ}\label{eq:FR-SDE-intro}
    g_{v_\bullet}^{\textup{FR},\varepsilon}(u_\bullet,w_\bullet)=\frac1\varepsilon\int_{0}^{1}\int_{\R^d}u_{t}(x)\cdot w_{t}(x) \rho_{t}^{v,\varepsilon}(\D x)\dif t. 
\end{equ}
\end{theorem*}
Based on this, we propose an extension of the Fisher--Rao metric to the setting of deterministic dynamics.
\begin{definition*}[Advective Fisher--Rao metric]
Consider a velocity field $v_\bullet\in\V$. 
For $u_\bullet, w_\bullet\in L^2_T C^1_{\textup{Lip}}$ we call 
\begin{equ}\label{eq:intro-AFR}
    g^{\textup{AFR}}_{v_{\bullet}}(u_{\bullet},w_{\bullet})=\int_{0}^{T}\int_{\R^d} u_{t}(x)\cdot w_{t}(x) \rho_{t}^{v}(\D x) \dif t, 
\end{equ}
the \emph{advective Fisher--Rao} metric, where $\rho_{\bullet}^{v}$ solves the continuity equation \eqref{eq:introduction-CE}.
\end{definition*}
We next show that the advective Fisher--Rao metric is indeed compatible with the least-squares energy $E$ in the sense that the gradient flow of $E$ with respect to the metric follows a geodesic in the space of paths of probability densities from any initial condition to the global optimizer.

\begin{theorem*}[\Cref{thm:FR-gradient-densities}, optimality]
Consider a velocity field $v_\bullet\in\V$ such that $\rho^\star_\bullet\ll\rho^v_\bullet$ and assume that the gradient with respect to the advective Fisher--Rao metric $u_{\bullet} = \nabla^{\AFR} E(v_{\bullet})$ exists in $L^1_TC^1_{\textup{Lip}}$. 
Then, it holds that
\begin{equ}
    \ddh \rho^{v+hu}_\bullet = \rho^v_\bullet - \rho^\star_\bullet. 
\end{equ}
\end{theorem*}

Informally speaking, this result suggests that, given the existence of the gradient flow, the advective Fisher--Rao gradient flow
\begin{equ}
    \partial_\tau v^\tau_\bullet = -\nabla^{\AFR} E(v^\tau_{\bullet})
\end{equ}
leads to the following dynamics
\begin{equ}
    p^{v^\tau}_\bullet = p^{\star}_\bullet + e^{-\tau} (p^{v^0}_\bullet-p^{\star}_\bullet).
\end{equ}
This means that the density path $p^{\tau}_\bullet$ converges to $p^\star_\bullet$ at a rate of $O(e^{-\tau})$ along the linear interpolation, which is known as a mixture geodesic in information geometry \cite{amari2016information, ay2017information}.
In particular, this shows that the advective Fisher--Rao metric is compatible with $E$, thereby providing an affirmative answer to the original question.

\subsubsection{The Advective Fisher--Rao geometry from three perspectives}
We next show that the  advective Fisher--Rao metric introduced above is a canonical object from three different perspectives:
(1) as a zero-noise limit of the Fisher--Rao metric on path measures, (2) as the expectation of the second variation of the Freidlin--Wentzell rate functional over the initial distribution, and (3) as the Hessian metric of the Benamou--Brenier action functional.
\paragraph{(1) Information geometry.}
The advective Fisher--Rao metric arises as a natural extension of the Fisher--Rao metric to the setting of deterministic dynamics as a rescaled zero-noise limit in which the Fisher--Rao metric becomes singular.
\begin{theorem*}[\Cref{thm:zero-noise-limit}]
Consider $v_{\bullet}\in\V$ and assume that $\rho_0\in\mathcal{P}_2(\R^d)$.
Then we obtain the advective Fisher--Rao metric as the rescaled zero-noise limit of the Fisher--Rao metrics on path measures, that is, for any $u_{\bullet},w_{\bullet}\in L^{2}_{T}C_{\textup{Lip}}^{1}$ we have 
\begin{equation*}
    g_{v_{\bullet}}^{\textup{AFR}}(u_{\bullet},w_{\bullet})=\lim_{\varepsilon\to0}\varepsilon g_{v_\bullet}^{\textup{FR},\varepsilon}(u_\bullet,w_\bullet).
\end{equation*}
\end{theorem*}
\paragraph{(2) Large deviations.}
As the metric is obtained as a zero-noise limit, we can also relate it to the Freidlin--Wentzell large deviation principle for small-noise diffusions, which describes the concentration of the path measures $\P^{v,\varepsilon}$ in the limit $\varepsilon\to0$.
The speed of the concentration is quantified by the Freidlin--Wentzell functional
\begin{equ}
    I_{v}(x_\bullet) = \frac{1}{2}\int_0^1 \lVert \dot{x}_t - v_t(x_t) \rVert_2^2 \dif t.
\end{equ}

\begin{theorem*}[\Cref{thm:Freidlin--Wentzell}]
Consider $\rho_0\in\mathcal{P}_2(\R^d)$, velocity fields $v_{\bullet}\in L^{2}_T C_{\textup{b}}^{2}$ and $u_{\bullet}\in L^{2}_T C_{\textup{Lip}}^{1}$, and denote the flow induced by $v_\bullet$ starting at $x_0$ by $\varphi^v(x_0)$.
Then it holds that 
\begin{equ}
    g_{v_{\bullet}}^{\textup{AFR}}(u_{\bullet},u_{\bullet})=
    \int_{\R^d} \left.\frac{\D^2}{\D h^2}\right|_{h=0}I_{v}(\varphi^{v+hu})(x_0) \rho_0(\D x_0) ,
\end{equ}
\end{theorem*}
\paragraph{(3) Optimal transport.}
The advective Fisher--Rao metric operates on velocity fields and is tightly connected to the solution operator of the continuity equation.
As such, it is a natural question how it is connected to the optimal transport geometry of the space of probability densities.
Recall the Benamou--Brenier formulation of the Wasserstein distance, which states
\begin{equ}
    W_2^2(\mu,\nu) = 2\inf_{p_\bullet, j_\bullet} \left\{\mathcal{A}(p_\bullet, j_\bullet) : p_0 = \mu, p_1 = \nu, \partial_t p_t + \nabla\cdot j_t = 0\right\},
\end{equ}
where the action functional is given by
\begin{equ}
    \mathcal{A}(p_\bullet, j_\bullet) =
        \frac12\int_0^1 \int_{\R^d} \frac{\lVert j_t(x) \rVert_2^2}{p_t(x)} \dif x \dif t
\end{equ}
Here, the current is given by $j_t = v_t p_t$.

\begin{theorem*}[\Cref{thm:hessian-action}]
Assume that the initial distribution $p_0\in\mathcal{P}_2(\R^d)$ has a positive density with respect to the Lebesgue measure.
For any velocity field $v_\bullet\in L^2_T C^1_{\textup{Lip}}$ and $u_\bullet \in C_{\textup{c}}^1([0,1]\times\R^d; \R^d)$, it holds that
\begin{equ}\label{eq:intro-difference-quotient}
    g^{\AFR}_{v_\bullet}(u_\bullet, u_\bullet) = 2\lim_{h\to0}\frac{\mathcal{A}(p_\bullet^{h}, j_\bullet^{h}) - \mathcal{A}(p_\bullet, j_\bullet) - \delta^+\mathcal{A}(p_\bullet, j_\bullet)[p_\bullet^{h}-p_\bullet, j_\bullet^{h} -j_\bullet]}{h^2},
\end{equ}
where $p_\bullet = p^{v}_\bullet$, $p^h_\bullet = p^{v+h u}_\bullet$, $j_\bullet = v_\bullet p^v_\bullet$, and $j^h_\bullet = (v_\bullet+h u_\bullet) p^{v+h u}_\bullet$.
\end{theorem*}
The result relates the advective Fisher--Rao metric to the Benamou--Brenier action functional.
Note that the difference quotient in \eqref{eq:intro-difference-quotient} is the difference quotient, which, given the existence of the second variation, differentiability of the map $\Psi(v_\bullet) = (p^v_\bullet, p_\bullet^v v_\bullet)$, and differentiability of the curve $(p^h_\bullet, j^h_\bullet)$ with respect to $h$, would converge to
\begin{equ}
    \delta^2 \mathcal{A}(\Psi(v_\bullet))[\delta \Psi(v_\bullet)[u_\bullet], \delta \Psi(v_\bullet)[u_\bullet]].
\end{equ}
In this case, the advective Fisher--Rao metric can be interpreted as the pullback of the Hessian metric of the Benamou--Brenier action functional along $\Psi$.

\subsubsection{The advective Fisher--Rao metric on paths of probability measures}
Having defined the advective Fisher--Rao metric on velocity fields, we now lift this construction to a Riemannian metric on the space of paths of probability measures. 
We work on the space of regular curves
\begin{equation}
    \mathrm{AC}_T^2(\mathcal{P}_2(\mathbb{R}^d))
    = \Bigl\{ \rho_\bullet\in\mathrm{AC}_T(\mathcal{P}_2(\mathbb{R}^d)) : \int_0^1 \lvert\dot\rho_t\rvert_{W_2}^2\,\mathrm{d}t < \infty \Bigr\},
\end{equation}
where $\lvert\dot\rho_t\rvert_{W_2}= \lim_{h\to0}W_2(\rho_{t+h},\rho_t)h^{-1}$ is the metric derivative \cite{ambrosio2005gradient}.
For any curve $\rho_\bullet\in\mathrm{AC}_T^2(\mathcal{P}_2(\mathbb{R}^d))$ there exists a unique velocity field $v_\bullet = v^\rho_\bullet\in L^2_TL^2(\rho_\bullet)$ such that the continuity equation
\begin{equation}
    \partial_t\rho_t + \nabla\cdot(\rho_t v_t) = 0,
\end{equation}
holds and $v_t\in\mathcal{T}_{\rho_t}\mathcal{P}_2(\mathbb{R}^d)$ for almost every $t$, where
\begin{equ}
    \mathcal{T}_\rho\mathcal{P}_2(\mathbb{R}^d)=\overline{\{\nabla\phi:\phi\in C^\infty_c(\mathbb{R}^d)\}}^{L^2(\rho)}
\end{equ}
denotes the Otto--Wasserstein tangent space.
To describe variations of a curve that remain compatible with the advective structure, we consider tangent vectors $\xi_\bullet=\ddh\rho^{v+hu}_\bullet$ at $\rho_\bullet=\rho^v_\bullet$, which arise from a variation of the velocity field. 
Such a tangent curve $\xi_\bullet$ satisfies the linearized continuity equation
\begin{equation}\label{eq:linearised-CE}
    \partial_t\xi_t + \nabla\cdot(\xi_t v_t) = -\nabla\cdot(\rho_t u_t)
\end{equation}
in a distributional sense. 
This leads to the definition of the tangent space 
\begin{equs}
    T_{\rho_{\bullet}}\mathrm{AC}^2_T(\mathcal P_2(\mathbb R^d))
    \coloneqq
    \left\{
    \xi_\bullet \in 
    T_{\rho_\bullet}L^2_T(\mathcal{P}_2(\R^d))
    :
    \begin{gathered} 
    \xi_t v_t \in \mathcal{D}'(\R^d) \textup{ for all } t\in[0,1]\\ 
    D_t \xi_\bullet \in T_{\rho_\bullet}L^2_T(\mathcal{P}_2(\R^d))
    \\
    \end{gathered}
    \right\},
\end{equs}
where we refer to 
\begin{equ}\label{eq:continuity-operator}
    D_t \xi_t \coloneqq \partial_t\xi_t + \nabla\cdot(\xi_t v_t)
\end{equ}
as the \emph{transport derivative} of $\xi_\bullet$ along $\rho_\bullet$, which is to be understood in the distributional sense. 
With these ingredients, we introduce the advective Fisher--Rao metric on paths of measures.

\begin{definition*}[Advective Fisher--Rao metric on paths of measures, Definition \ref{def:AFR-measures}]
Consider a path $\rho_\bullet \in \textup{AC}_T^2(\mathcal{P}_{2}(\R^{d}))$ and its velocity  $v_\bullet  \in T_{\rho_\bullet} L_T^2(\mathcal{P}_{2}(\R^{d}))$.
For two tangent vectors $\xi_\bullet, \zeta_\bullet\in T_{\rho_\bullet} \textup{AC}_T^2(\mathcal{P}_2(\R^d))$, we refer to
\begin{equ}
    g_{\rho_{\bullet}}^{\AFR}(\xi_{\bullet},\xi_{\bullet})=\int_{0}^{1}g_{\rho_{t}}^{\textup{WO}}(D_{t}\xi_{t}, D_{t}\xi_{t})\dif t 
\end{equ}
as the \emph{advective Fisher--Rao metric}. 
\end{definition*}

A natural point of reference is the \emph{integrated Wasserstein--Otto metric}
\begin{equation}
    g_{\rho_\bullet}^{\mathrm{WO}}(\xi_\bullet,\xi_\bullet)
    = \int_0^1 g_{\rho_t}^{\mathrm{WO}}(\xi_t,\xi_t)\,\mathrm{d}t
    = \int_0^1\!\!\int_{\mathbb{R}^d} \lVert v^\xi_t(x)\rVert^2\,\rho_t(\mathrm{d}x)\,\mathrm{d}t,
\end{equation}
where $v^\xi_t\in\mathcal{T}_{\rho_t}\mathcal{P}_2(\mathbb{R}^d)$ is the unique solution of $\xi_t = -\nabla\cdot(\rho_t v^\xi_t)$.
In contrast, the advective Fisher--Rao metric involves the {transport derivative} of $\xi_\bullet$ along the base velocity field and measures the size of the source term that appears when the variation $\xi_\bullet$ is transported along the underlying flow. 
Geometrically, $D_t$ coincides with the covariant derivative of a flat connection that differs from the Levi-Civita connection of the Wasserstein--Otto metric~\cite{ay2025information}.
The metric defined above is the natural counterpart of the advective Fisher--Rao metric on velocity fields and inherits several fundamental properties:
\begin{itemize}
    \item \emph{Compatibility.} Lemma~\ref{thm:FR-isometrc} shows that the solution map $v_\bullet\mapsto\rho_\bullet^v$ of the continuity equation is a Riemannian isometry with respect to the advective Fisher--Rao metrics.
    \item \emph{Relation to the Benamou--Brenier action.} Theorem~\ref{thm:convex-geometry-BB} identifies the AFR metric on paths as the second variation of the Benamou--Brenier action functional, establishing a direct connection to optimal transport.
    \item \emph{Optimality.} Theorem~\ref{thm:measures-optimality} proves that the gradient flow of the least-squares energy on the velocities with respect to the advective Fisher--Rao metric follows a mixture geodesic toward the reference path, guaranteeing exponential convergence. 
\end{itemize}

\subsubsection{Areas of application}\label{subsubsec:examples}

Here, we discuss different problems that are covered by our general framework of control of probability measures with quadratic cost on the velocity fields. 

\paragraph{Flow matching}
In generative modeling, it is the goal to produce approximate samples from an unknown data distribution $\rho_{\textup{data}}$. 
Flow matching was introduced in~\cite{lipman2023flow} and quickly became one of the most powerful paradigms in generative modeling. 
It aims to generate approximate samples by transforming sampling from an easy source distribution $\rho_{\textup{source}}$, for example, a standard Gaussian, along a flow of a velocity field. 
In flow matching, a neural network is used to parametrize a time-dependent velocity field $v^\theta_\bullet$ that generates the flow. 
Then, samples from a source distribution $\rho_{\textup{source}}$ are transformed according to 
\begin{equ}
    \dot{X}_t = v^\theta_t(X_t), \quad X_0\sim \rho_{\textup{source}}.
\end{equ}
Denoting the distribution of $X_t$ by $\rho_t$, it is the goal to achieve $\rho_1\approx \rho_{\textup{data}}$. 
To this end, one constructs a velocity field $v^\star_\bullet$ such that 
\begin{equ}
    \dot{X}^\star_t = v^\theta_t(X_t^\star), \quad X_0^\star\sim \rho_{\textup{source}} 
\end{equ}
achieves exact interpolation, in the sense that $\rho^\star_1= \rho_{\textup{data}}$, where $\rho^\star_t$ denotes the distribution of $X^\star_t$. 
Such a reference field $v^\star_\bullet$ is neither available in closed form nor unique, and various strategies for the construction and estimation of such velocity fields based on conditional interpolation or ideas from optimal transport have been developed. 
We refer to \cite{lipman2024flow, albergo2025stochastic, wald2025flow, pierret2026flow} for extensive discussions on this aspect of flow-based generative models. 
Having chosen a reference velocity field $v^\star_\bullet$, the network's parameters are optimized to minimize the flow matching loss  
\begin{equ}
    L_{\textup{FM}}(\theta) = \frac12\int_0^1\int_{\mathbb{R}^d}\|v_t^\theta(x)-v_t^\star(x)\|^2 \rho_t^\star(\D x) \dif t.
\end{equ}
This objective function can be written as $L(\theta) = E(v^\theta_\bullet)$ for a quadratic energy $E$ on the velocity fields and therefore of the form \eqref{eq:least-squares-objective}. 

\paragraph{Score matching}
Diffusion models are arguably one of the most successful generative models~\cite{Sohl-Dickstein2015, Ho2020, rombach2022high}. 
Here, we focus on score matching, where one considers forward dynamics 
\begin{equ}
    \D X_t^\star = f_t(X_t^\star) + g_t \D W_t, \quad X_0^\star \sim \rho_{\textup{data}}
\end{equ}
and we denote the density of $X_t^\star$ by $\rho^\star_t$, where $f_t$ and $g_t>0$ are specified by the user in a way that $\rho^\star_1\approx \rho_{\textup{source}}$, see \cite{lai2025principles}. 
The backwards ODE 
\begin{equ}
    \dot{X}_t = f_t(X_t) - \frac{g_t^2}2 s_t^\star(X_t), \quad X_1 \sim \rho^\star_1
\end{equ}
induces exactly the same time marginals as $X_t^\star$, meaning that $X_t \sim \rho^\star_t$, where the true score function is given by $s_t^\star = \nabla \ln p^\star_t$ and $p^\star_t$ denotes the Lebesgue density of $\rho^\star_t$. 
As $s^\star_\bullet$ is generally not accessible, it is approximated by a neural network $s^\theta_\bullet \approx s^\star_\bullet$, and the network's parameters $\theta$ are optimized to minimize the least-squares loss 
\begin{equ}
    L_{\textup{SM}}(\theta) = \frac12 \int_0^1 \alpha_t  \int_{\R^d}  \lVert s^\theta_t(x) - s^\star_t(x) \rVert_2^2 \rho^\star_t(\D x) \dif t,
\end{equ}
where $\alpha_t>0$ is a weighting factor specified by the user. 
The backward ODE with the approximate score function is given by 
\begin{equ}
    \dot{X}^\theta_t = f_t(X_t) - \frac{g_t^2}2 s_t^\theta(X_t), \quad X_1^\theta\sim \rho_{\textup{source}}. 
\end{equ}
Denoting the corresponding velocity fields by 
\begin{equ}
    v_t^\theta(x)\coloneqq f_t(x) - \frac{g_t^2}{2} s_t^\theta(x), \quad \textup{and } 
    v_t^\star(x)\coloneqq f_t(x) - \frac{g_t^2}{2} s_t^\star(x), 
\end{equ}
we can rewrite the loss function according to 
\begin{equ}
    L(\theta) = 2\int_0^1 \frac{\alpha_t}{g_t^4} \int_{\R^d} \lVert v^\theta_t(x) - v^\star_t(x) \rVert_2^2 \rho_t^\star(\D x) \dif t,
\end{equ}
which is a quadratic function of the velocity $v_\bullet^\theta$. 

\paragraph{Imitation learning in linear environments}
The reasoning above can readily be extended to imitation learning under general linear dynamics 
\begin{equ}
    \dot{x}_t = A_t x_t + B_t u_t
\end{equ}
where $x\in\R^d$ is the state and $u\in\R^n$ the control. 
For various real-world problems, incorporating expert demonstrations was necessary to achieve high performance~\cite{pomerleau1988alvinn,ross2010efficient,bojarski2016end,teng2023hierarchical,zhao2023learning}.
The integration of such prior knowledge is implemented in the framework known as \emph{imitation learning} or \emph{behavioral cloning}. 
Here, one assumes access to an expert policy $\pi^\star_\bullet\colon[0,1]\times\R^d\to\R^n$, uses a neural network $\pi^\theta_\bullet\colon[0,1]\times\R^d\to\R^n$, and trains the network's parameters to minimize the loss 
\begin{equs}
    L(\theta) & = \frac12 \int_{\R^d}\int_0^1 \lVert \pi^\theta_t(x) - \pi^\star_t(x) \rVert^2  \rho^\star_t(\D x) \dif t
\end{equs}
where $\rho^\star_t$ denotes the distribution of $x_t$ under the policy $\pi^\star_\bullet$. 
Despite the linear nature of the dynamics, the problem is not solvable via the Riccati equations as it involves a possibly nonlinear expert policy $\pi^\star_t$. 
For example, even in linear model predictive control, the expert policy becomes piecewise affine and intractable to compute for high-dimensional problems, and we refer to~\cite[Section 7.2]{reiter2026synthesis} for a survey of imitation learning approaches in this case. 

In this setting of linear dynamics, the state dynamics under a policy can be written as 
\begin{equ}
    \dot{x}_t = v^\theta_t(x_t), \quad \textup{where } v^\theta_t(x) = A_t x + B_t \pi^\theta_t(x). 
\end{equ}
If $B_t$ has full column rank for every $t\in[0,1]$, we obtain 
\begin{equs}
    L(\theta) & = \frac12 \int_{\R^d}\int_0^1 (v^\theta_t - v^\star_t)^\top (B_t^\top B_t)^{-1}(v^\theta_t - v^\star_t)  \rho^\star_t(\D x) \dif t, 
\end{equs}
which is again a quadratic objective in the velocity field and hence of the form \eqref{eq:least-squares-objective}. 

\paragraph{Computational experiments} 
To use the developed Riemannian metric for optimization of the velocity fields, we consider the application to the flow matching approach. 
Here, we discretize the advective Fisher--Rao gradient flow and obtain a natural gradient method, which updates the parameters $\theta$ of the neural network according to
\begin{equ}\label{eq:intro-natural-gradient}
    \theta_{k+1} = \theta_k - \eta \cdot F(\theta_k)^{-1} \nabla L_{\textup{FM}}(\theta_k),
\end{equ}
where the Fisher-information matrix $F(\theta)$ is a discretization of the advective Fisher--Rao metric with entries
\begin{equ}
    F(\theta)_{ij} = g^{\AFR}_{v^\theta_\bullet}(\partial_{\theta_i} v^\theta_\bullet, \partial_{\theta_j} v^\theta_\bullet) =  \int_0^1\int_{\R^d} \partial_{\theta_i}v^\theta(x)\cdot \partial_{\theta_j}v^\theta(x) p_t^{\theta}(x) \dif x \dif t.
\end{equ}
In contrast, the Gauss--Newton method replaces the Fisher-information matrix in the preconditioned gradient scheme \eqref{eq:intro-natural-gradient} by
\begin{equ}
    G(\theta)_{ij} = (\partial_{\theta_i} v^\theta_\bullet, \partial_{\theta_j} v^\theta_\bullet)_{L^2_TL^2(p^\star_\bullet)} =  \int_0^1\int_{\R^d} \partial_{\theta_i}v^\theta(x)\cdot \partial_{\theta_j}v^\theta(x) p_t^{\star}(x) \dif x \dif t.
\end{equ}
Consistent with the theoretical results, we demonstrate empirically that the natural gradient method leads to the desired optimal fitting of the paths of probability densities, while the Gauss--Newton method leads to an optimal fitting of the velocity fields, see \Cref{fig:velocity-GN}.
Further, we demonstrate that both of these second-order optimizers have the ability to increase the accuracy of flow-based generative models over first-order optimizers.
All experiments are presented in \Cref{subsec:NG-method}.

\subsection{Related works}
\paragraph{Schr\"odinger problems and entropy-regularized optimal transport.}
The objective \eqref{eq:least-squares-objective} is well known to coincide with the Kullback--Leibler divergence between path measures induced by the underlying SDEs, a correspondence that dates back at least to the classical work of \cite{follmer1985entropy}. Minimizing this path-space relative entropy gives rise to $L^2$ stochastic optimal control problems, whose optimal solutions characterize dynamic Schr\"odinger bridges and thereby the entropic regularization of dynamic optimal transport \cite{benamou2000computational}.
A large body of work has been devoted to understanding the convergence of these entropy-regularized problems in the small-noise limit $\varepsilon\to0$.
Foundational results qualitatively establish the convergence in the vanishing-noise limit \cite{mikami2002optimal,mikami2004monge,leonard2012schrodinger,carlier2017convergence}.
More recently, quantitative convergence rates have been developed through asymptotic expansions of the regularization error, including first-order expansions in \cite{carlier2023convergence,eckstein2024convergence,malamut2025convergence} and second-order expansions in \cite{conforti2021formula,chizat2020faster,pal2024difference, nutz2026entropic}.
Closely related works further investigate the convergence of Schr\"odinger potentials and transport maps in the small-noise regime~\cite{gigli2021second,pooladian2021entropic,nutz2022entropic,bernton2022entropic,chiarini2023gradient,carlier2023convergence,malamut2025convergence,lopez2026uniform}.
Our perspective is complementary to this line of research.
Rather than studying the asymptotic behavior of the optimal values or minimizers of entropy-regularized optimal transport problems, we investigate the differential geometry induced by the regularization functional itself.
More precisely, we study the Riemannian metric arising from its Hessian on path space and identify its singular limit as the noise level tends to zero, leading to the advective Fisher--Rao geometry for deterministic flows of probability measures.
\paragraph{Stochastic optimal control and sampling}
The objective \eqref{eq:least-squares-objective} arises in $L^2$ stochastic optimal control problems.
A similar least-squares objective was used in \cite{richter2021solving} and subsequent works \cite{berner2024an} to phrase the training of score-based diffusion samplers as a path-space stochastic optimal control problem.
A similar paradigm is pursued in \cite{zhang2022path}, which frames sampling from unnormalized densities as a path-space stochastic control problem optimized via path-integral methods, and in \cite{domingo2025adjoint}, where control and reward-based fine-tuning are addressed using an adjoint matching regression target.
Further, a geometric optimization framework for this ansatz utilizing a trust-region method based on the relative entropy on the space of path measures has been developed in \cite{blessing2025trust}.
Although closely related to our approach in spirit, these methods use a reverse KL divergence as their objective function rather than the forward KL divergence that is commonly used in flow and score matching models.
This choice of objective results in exponential geodesics, as opposed to the mixture geodesics that arise in our framework.
Further, this approach heavily relies on the presence of noise, as otherwise the KL divergence between path measures of different velocity fields will never be finite, necessitating our zero-noise construction of a corresponding geometry for deterministic forward dynamics.
\paragraph{Geometry of flow-based generative models}
Various works have studied the geometry of certain aspects of flow-based generative models.
\cite{ghimire2023geometry,vuong2025we} argue that the forward and backward dynamics in diffusion models should be interpreted through the lens of optimal transport, as they can be interpreted as Wasserstein gradient flows of the KL divergence with respect to the data distribution and the prior distribution, respectively.
This observation covers the data generation dynamics, but fails to address the objective function or the training process.
To study geometric structures embedded within these models, other recent works have introduced localized or pull-back Riemannian structures. For instance, \cite{karczewski2026the} analyzes the geometry induced by the Fisher information matrix along the reverse probability flow ODE to examine trajectory properties, while \cite{azeglio2025metric} leverages learned score functions to construct ambient-space metrics for navigating and exploring the underlying data manifold.
Though insightful, these approaches pursue a different goal and characterize the static space of data points or sample paths rather than the functional landscape of the objectives themselves.
On the side of optimization of flow-based generative models, several second-order heuristics have been proposed to improve their trainability.
For example, \cite{liu2021second} proposes a low-rank approximation of the Hessian for neural ODEs, which, however, is not derived from a functional perspective and does not provide a geometric interpretation, thereby leading to suboptimal update directions~\cite{muller2024position}.
Further, \cite{hu2026sing} proposes a natural gradient method for SDE drifts, which uses the Fisher information matrix as a preconditioner and thereby inherently requires noise in the forward dynamics.

\paragraph{Reinforcement learning and natural policy gradients}
The Riemannian metric constructed in this work is conceptually related, though with important differences, to the metric underlying many important successes of reinforcement learning.
Here, Kakade proposed an adaptive mixture of the Fisher--Rao metrics on the individual components as a metric on the space of Markov kernels~\cite{kakade2001natural}.
This geometry has been studied in terms of invariance axiomatics~\cite{Lebanon2005AxiomaticGO, montufar2014fisher}, has been characterized as isometric to the conditional Fisher--Rao metric on occupancy measures \cite{muller2024geometry}, and has been shown to arise from the Fisher--Rao metric on path measures~\cite{bagnell2003covariant, peters2003reinforcement, nagaoka2017exponential}.
The advective Fisher--Rao metric constructed by us is closely related to the pullback of the Fisher--Rao metric on path measures.
However, the advective Fisher--Rao metric introduced here operates in a regime of continuous time and deterministic forward dynamics, requiring a rescaled zero-noise limit, thereby offering close links to large deviation theory and optimal transport.
Furthermore, most reinforcement learning focuses on a different objective function from the one considered here, which prevents a characterization as a metric that yields the optimal update directions obtained here.
\paragraph{Fisher--Rao gradient flows and sampling}
A substantial body of work uses Fisher--Rao gradient flows, or their discretisations, for sampling and achieves $O(e^{-t})$ convergence, independently of the target, see~\cite{lu2019accelerating,lu2023birth,carrillo2026fisher} and subsequent work.
These works are related in philosophy to our approach in that they propose a Fisher--Rao geometry for the optimization of an entropic objective.
However, they consider the Fisher--Rao metric for a fixed-time marginal, whereas our $g^{\mathrm{AFR}}$ metric is a global, pathwise tensor that accounts for the temporal coupling induced by the continuity equation.
\paragraph{Wasserstein--Fisher--Rao geometry}
Our work studies the zero-noise limit of the Fisher--Rao metric on path measures, which we show to converge to the Hessian geometry of the Benamou--Brenier action functional.
As such, it can be seen as recovering a geometry induced by a functional appearing in optimal transport as the limiting object of the central metric in information geometry.
Thus, we want to contrast our work with the construction of a geometry unifying the Fisher--Rao and Wasserstein geometries, which is known as the Wasserstein--Fisher--Rao geometry or Hellinger--Kantorovich geometry.
This construction was motivated by unbalanced optimal transport \cite{piccoli2014generalized} and allows for mass creation and destruction in addition to the transport of mass, thereby offering an interpolation between the Wasserstein and Fisher--Rao geometries~\cite{liero2018optimal,chizat2018interpolating}.
The resulting object is a geometry on static probability measures, whereas our construction works on the space of paths of probability measures, thereby capturing the temporal coupling induced by the continuity equation.
\section{Preliminaries and Setup}
This section serves the purpose of introducing the notation and technical assumptions we will use throughout the paper.

\subsection{Notation and conventions}
We denote the $d$-dimensional Euclidean space by $\R^{d}$, the Euclidean gradient by $\nabla$, and the divergence of a vector field $v$ by $\nabla\cdot v$.
We denote the space of distributions by $\mathcal{D}'(\R^d)$, which is the topological dual space to the locally convex vector space $C^\infty_{\textup{c}}(\R^d)$ of compactly supported smooth functions on $\R^d$. 
For the sake of simplified notation, we set $L^p_T\coloneqq L^p([0,1])$ for $p\in[1,\infty)\cup\{+\infty\}$ and for a Banach space $X$, we denote the Bochner--Lebesgue and Sobolev--Bochner--Lebesgue spaces by $L^p_T(X)\coloneqq L^p([0,1];X)$ and $H^1_T(X)\coloneqq H^1([0,1];X)$, respectivly. 
The space of continuous paths is $C_{T}(\R^{d})\coloneqq C([0,T];\R^{d})$.
With $\mathcal{P}_{2}(\R^{d})$ we denote the space of probability measures on $\R^{d}$ with finite second moment 
\begin{equ}
    \mathcal{P}_{2}(\R^{d})\coloneqq\left\{\mu\in\mathcal{P}(\R^{d}): m_{2}(\mu)<\infty\right\},
    \qquad
    m_{2}(\mu)\coloneqq\int_{\R^{d}}\|x\|^{2}\,\dif\mu(x).
\end{equ}
The Wasserstein-$2$ distance between $\mu,\nu\in\mathcal{P}_{2}(\R^{d})$ is given by
\begin{equ}
    W_{2}^{2}(\mu,\nu)\coloneqq\inf_{\pi\in\Pi(\mu,\nu)}\int_{\R^{d}\times\R^{d}}\|x-y\|^{2}\,\dif\pi(x,y), \end{equ}
where $\Pi(\mu,\nu)$ denotes the set of couplings of $\mu$ and $\nu$.
For a measurable map $\Phi$, the pushforward of a measure $\mu$ is $(\Phi)_{\sharp}\mu$.
The dual pairing between a measure $\mu$ and a function $f$ is $\langle\mu,f\rangle\coloneqq\int f\dif\mu$.
For $\alpha\in(0,1]$, the  $\alpha$-H\"older semi-norm of a function $\phi\colon \R^{d}\to\R$ is given by 
\begin{equation*}
    [\phi]_{\alpha} \coloneqq \sup_{x\neq y} \frac{|\phi(x)-\phi(y)|}{\|x-y\|^{\alpha}}. 
\end{equation*}
In the special case $\alpha=1$, we denote the Lipschitz semi-norm by 
\begin{equation*}
    [\phi]_{\mathrm{Lip}} \coloneqq \sup_{x\neq y} \frac{|\phi(x)-\phi(y)|}{\|x-y\|}. 
\end{equation*}
For a convex functional $F\colon C\to\R\cup\{+\infty\}$ we denote the Bregman divergence by 
\begin{equ}
    D_F(x,y)\coloneqq F(x) - F(y) - \delta^+ F(y)[x-y] \quad \text{for all } x,y\in\operatorname{dom}(F), 
\end{equ}
where for any direction $d=y-x$ for $x,y\in\operatorname{dom}(F)$ the right-sided first variation is defined as
\begin{equ}
    \delta^+ F(x)[d] \coloneqq \lim_{h\searrow0} \frac{F(x+hd)-F(x)}{h}.
\end{equ}

\subsection{Flows, the continuity equation, and path measures}
We consider the SDE
\begin{equ}
    \D X_{t}=v_{t}(X_{t})\D t+\sqrt{\varepsilon}\D W_{t},\label{eq:SDE}
\end{equ}
where $X_{0}\sim p_{0}$, where $v_{\bullet}$ is a time-dependent velocity field.
The introduction of the noise term $\sqrt{\varepsilon}\D W_{t}$ allows us to construct a non-degenerate geometry Fisher--Rao metric in the noiseless case, via a rescaled zero-noise limit.
\paragraph{Regularity of velocity fields}
We work in the classic setting of Carath\'{e}odory solutions of ODEs, which is a convenient setting for the existence of solutions to the continuity equation that also covers many neural network architectures.
We denote the space of all continuously differentiable Lipschitz vector fields by
\begin{equation*}
    C_{\textup{Lip}}^{1}(\R^{d};\R^{d})\coloneqq\left\{ v\in C^{1}(\R^{d};\R^{d}):v\textup{ is globally Lipschitz}\right\} .
\end{equation*}
This is a Banach space with the norm
\begin{equation*}
	\lVert v\rVert_{C_{\textup{Lip}}^{1}(\R^{d};\R^{d})}\coloneqq\lVert v(0)\rVert+\lVert Dv\rVert_{\infty}.
\end{equation*}
Further, we consider the Bochner-Lebesgue space\footnote{Note that as a time horizon we choose $T=1$ throughout this paper; nevertheless, the notation $L^1_T$ indicates that $L^1$ is an integrability assumption in time.}
\begin{equation}
	\V\coloneqq L^{1}([0,1];C_{\textup{Lip}}^{1}(\R^{d};\R^{d}))
\end{equation}
of time-dependent velocity fields that are continuously differentiable and Lipschitz in space, such that
\begin{equation*}
	\lVert v_{t}(0)\rVert\le\alpha_{t}\quad\text{and }\lVert v_{t}(x)-v_{t}(y)\rVert\le\beta_{t}\lVert x-y\rVert
\end{equation*}
for all $x,y\in\R^{d}$, $t\in[0,1]$ and some $\alpha,\beta\in L^1_T$.
Here, the norm on $\V$ is given by
\begin{equation*}
    \lVert v_\bullet\rVert_{\V}=\int_{0}^{1}\lVert v_{t}\rVert_{C_{\textup{Lip}}^{1}(\R^{d};\R^{d})}\dif t = \lVert \alpha \rVert_{L^1_T} + \lVert \beta \rVert_{L^1_T}.
\end{equation*}
Note that by the Carath\'{e}odory existence theorem, $\V$ is a convenient space of velocity fields for the existence of solutions of the ODE and consequently for the existence of solutions to the continuity equation.
Indeed, a velocity field $v_\bullet\in \V$ induces a unique flow $\varphi_{\bullet}^{v}$ satisfying
\begin{equation*}
    \partial_{t}\varphi_{t}^{v}(x)=v_{t}(\varphi_{t}^{v}(x)),\quad\text{and }\varphi_{0}^{v}(x)=x
\end{equation*}
for all $t\in[0,1]$, $x\in\R^{d}$, see \cite[Lemma 8.1.4]{ambrosio2008gradient}.
Sometimes, we require stronger square-integrability in time and write
\begin{equ}
    L^2_TC^{1}_{\textup{Lip}} \coloneqq L^{2}([0,1];C_{\textup{Lip}}^{1}(\R^{d};\R^{d})).
\end{equ}
In this case, the functions $\alpha, \beta$ are in $L^2_T$ and equivalent norms are given by
\begin{equ}
    \lVert \alpha \rVert_{L^2_T} + \lVert \beta \rVert_{L^2_T} \quad\text{or}\quad \sqrt{\lVert \alpha \rVert_{L^2_T}^2 + \lVert \beta \rVert_{L^2_T}^2}.
\end{equ}
Note that for $v\in\V$, not only the ODE but also the SDE \eqref{eq:SDE} has a unique strong solution for any initial condition $X_{0}\sim p_{0}$, where $p_{0}$ is a probability measure with finite second moments, see \cite{karatzas2014brownian}.
In flow-matching and diffusion models, one typically uses neural networks $v^\theta_\bullet$ that take time as an explicit input.
If the activation function is continuous and differentiable with a bounded derivative, so is the resulting velocity if a joint function of space and time, and hence contained in $\V$.
The regularity of $\V$ is enough to ensure well-posedness of the forward dynamics.
However, sometimes, we require square integrability in time and denote the corresponding space by $L^2_T C^1_{\textup{Lip}}$.
For some results, we require stronger spatial regularity and hence introduce the function space
This is a Banach space with the norm
\begin{equation*}
	\lVert v\rVert_{C_{\textup{b}}^{2}(\R^{d};\R^{d})}\coloneqq\lVert v(0)\rVert+\lVert Dv\rVert_{\infty}+\lVert D^2v\rVert_{\infty}.
\end{equation*}
Similarly for $\alpha\in(0,1]$, we define 
\begin{equation*}
    C_{\textup{Lip}}^{1, \alpha}(\R^{d};\R^{d})\coloneqq\left\{ v\in C^{1}(\R^{d};\R^{d}): Dv\textup{ is globally bounded and $\alpha$-H\"older}\right\} .
\end{equation*}
with the corresponding norm 
\begin{equation*}
	\lVert v\rVert_{C_{\textup{Lip}}^{1, \alpha}(\R^{d};\R^{d})}\coloneqq\lVert v(0)\rVert+\lVert Dv\rVert_{\infty} + [Dv]_\alpha.
\end{equation*}

\paragraph{The continuity equation}
It is well-known that for any initial probability measure $p_{0}\in\mathcal{P}_{2}(\R^{d})$ with finite second moments the transported measures $\rho_{t}\coloneqq(\varphi_{t}^{v})_{\sharp}p_{0}\in\mathcal{P}_{2}(\R^{d})$ are a distributional solution to the continuity equation
\begin{equation*}
	\partial_{t}\rho_{t}+\nabla\cdot(v_{t}\rho_{t})=0
\end{equation*}
meaning that for any test function $\phi\in C_{\textup{c}}^{\infty}((0,1)\times\R^{d})$ it holds that
\begin{equation*}
	\int_{0}^{1}\int_{\R^{d}}\left(\partial_{t}\phi_{t}(x)+v_{t}(x)\cdot\nabla\phi_{t}(x)\right) \rho_{t}(\dif x). 
\end{equation*}
For a proof, we refer to \cite[Lemma 8.1.6.]{ambrosio2008gradient} or \cite[Lemma 4.1.1.]{figalli2021invitation}.
In particular, this shows that if $p_0$ is a measure with positive Lebesgue density, then so is $p_t$ for any $t\in[0,1]$.
The method of characteristics can easily be generalized to arbitrary finite signed measures $\rho_{0}$.
Further, any solution of the continuity equations is given by the method of characteristics, and in particular, solutions of the continuity equation are unique under the regularity conditions on $v_\bullet$.
Finally, we remind the reader of the continuity equation with source term $f_{\bullet}$, which is given by
\begin{equ}
    \partial_{t}\xi_{t}+\nabla\cdot(v_{t}\xi_{t})=f_{t}
\end{equ}
admits a unique solution for $v_\bullet\in\V$ given by the methods of characteristics
\begin{equ}
    \xi_t = (\varphi_t^v)_\sharp \xi_0 + \int_0^t (\varphi_{t,s}^v)_\sharp f_s \dif s,
\end{equ}
see for example \cite{ambrosio2014continuity}.
\paragraph{Girsanov's theorem
}
We call the space $C_T(\R^{d})$ of continuous paths the $d$-dimensional \emph{path space} and refer to probability measures on this space as path measures.
It is well known that the $L^{2}$-costs on the level of velocity fields can be related to the Kullback--Leibler divergence between the path measures $\P^{v,\varepsilon}$ and $\P^{\star,\varepsilon}$ via Girsanov's theorem.
To state a version of Girsanov's theorem, we consider two path measures $\mathbb{P}^{v, \varepsilon}$ and $\P^{w, \varepsilon}$ induced by the SDEs
\begin{equs}
	\D X_{t} & =v_{t}(X_{t})\D t+ \sqrt{\varepsilon}\D W_{t},\quad X_{0}\sim p_0\label{eq:SDE-X}\\
	\D Y_{t} & =w_{t}(Y_{t})\D t+\sqrt{\varepsilon}\D W_{t},\quad Y_{0}\sim p_0,
\end{equs}
where we assume weak existence and uniqueness in law.
We set $u_{t}(x)\coloneqq v_{t}(x)-w_{t}(x)$ and assume that Novikov's condition
\begin{equ}
    \E_{\mathbb{P}^{w, \varepsilon}}\left[\exp\left(\frac{1}{2\varepsilon}\int\|u_{t}(X_{t})\|_{2}^{2}\D t\right)\right]<+\infty
\end{equ}
holds.
Then we have
\begin{equ}\label{eq:girsanov}
	\frac{\textup{d}\mathbb{P}^{v, \varepsilon}}{\textup{d}\mathbb{P}^{w, \varepsilon}}(x_{\bullet}) =
	\exp\left(\frac{1}{{\varepsilon}}\int_{0}^{1}u_{t}(x_{t})\cdot(\mathrm{d}x_{t}-w_{t}(x_{t})\D t)-\frac{1}{2\varepsilon}\int_{0}^{1}\|u_{t}(x_{t})\|_{2}^{2}\,\mathrm{d}t\right). \label{eq:girsanov}
\end{equ}
Note that as under $\P^{w,\varepsilon}$ we have $\D x_t - w_t(x_t)\D t = \sqrt{\varepsilon} \D W_t$, the Radon-Nikodym derivative is often written as
\begin{equ}
	\frac{\textup{d}\mathbb{P}^{v, \varepsilon}}{\textup{d}\mathbb{P}^{w, \varepsilon}}(x_{\bullet}) =
	\exp\left(\frac{1}{{\sqrt{\varepsilon}}}\int_{0}^{1}u_{t}(x_{t})\cdot\D W_t-\frac{1}{2\varepsilon}\int_{0}^{1}\|u_{t}(x_{t})\|_{2}^{2}\,\mathrm{d}t\right).
\end{equ}
For a proof, we refer to \cite{ustunel2013transformation} for additive and to \cite{nusken2021solving} for multiplicative noise, respectively.
Girsanov's theorem yields a direct expression of the relative entropy between the path measures $\P^{v}$ and $\P^{w}$ as a least-squares cost of the difference of the drifts $v_{\bullet}$ and $w_{\bullet}$, which is well-known in the field of stochastic control and diffusion processes, see e.g. \cite{follmer1985entropy,leonard2012girsanov}.
Recall that for two probability measures $\P$
and $\Q$ the \emph{relative entropy}, \emph{Kullback--Leibler divergence} or \emph{shortly
KL-divergence} is given by
\begin{equation}
D_{\textup{KL}}(\P,\Q)=\begin{cases}
\E_{\P}\left[\ln\left(\frac{\D\P}{\D\Q}\right)\right], & \text{if }\P\ll\Q,\\
+\infty, & \text{otherwise}.
\end{cases}
\end{equation}
Consider two path measures $\P^{v}$ and $\P^{w}$
induced by the SDE \eqref{eq:SDE}. Then it holds that
\begin{equation*}
D_{\KL}(\P^{w},\P^{v})=
    \frac{1}{2\varepsilon}\int_{0}^{1}\int_{\R^{d}}\|w_{t}(x)-v_{t}(x)\|_{2}^{2}\;\rho_{t}^{w}(\D x)\dif t,
\end{equation*}
For a proof, we refer to \cite{lai2025principles}.
\subsection{Information geometry and natural gradients}
We have seen that a least-squares objective on velocities can be viewed as a KL-divergence on the level of path measures.
The KL-divergence is inherently connected to the negative entropy, which induces a natural Riemannian structure given by the Fisher--Rao metric, which is a central object in information geometry.
For a more comprehensive introduction to information geometry, we refer to the extensive monographs \cite{amari2016information,ay2017information}.
Consider a probability measure $\P$ and two variations $\P^h$ and $\Q^h$ of $\P$ meaning that $\P^{0}=\Q^{0}=\P$, where we assume that $\P^h, \Q^h\ll\P$ and that
the derivatives $\Xi = \ddh \P^h$ and $Z = \ddh \Q^h$ exist in a suitable sense.
Then, the \emph{Fisher--Rao metric} is, given the existence of the expectation, given by
\begin{equs}\label{eq:log-expression-FR}
    g_{\P}^{\FR}(\Xi,Z) & =
    \E_{\P}\left[\ddh
    \ln\left(\frac{\D\P^{h}}{\D\P}\right)\ddh
    \ln\left(\frac{\D\Q^{h}}{\D\P}\right)\right].
\end{equs}
A precise meaning to the tangent space of the measures equivalent to $\P$ and in particular the derivatives $\ddh \P^h$ and $\ddh \Q^h$ was given by Pistone and Sempi in their seminal work~\cite{pistone1995infinite}.
They showed that one can endow the space of equivalent measures as a smooth Banach manifold over the Orlicz space $L^\Phi(\P)$, where $\Phi(t)=e^{\lvert t \rvert}-1$.
In particular, this gives the following characterization of the tangent space
\begin{equ}\label{eq:definition-PS-tangent}
    T_\P \mathcal{P}(\P) = \left\{ \ddh \P^h : \ddh \ln \frac{\D\P^h}{\D\P} \textup{ exists in } L^\Phi(\P) \right\},
\end{equ}
where $\mathcal{P}(\P) = \{ \Q : \Q\ll\P \textup{ and } \P\ll\Q \}$ denotes the set of probability measures equivalent to $\P$.
In particular, for such variations $\P^h$, the derivative $\ddh\P^h$ exists with respect to the total variation distance and is given by
\begin{equ}
    \ddh \P^h = \left( \ddh \ln \frac{\D\P^h}{\D\P} \right) \P\in T_\P \mathcal{P}(\P).
\end{equ}
The topology on the tangent space is inherited from the Orlicz space $L^\Phi(\P)$, and the corresponding convergence is known as e-convergence.
For a detailed discussion, we refer to the original construction \cite{pistone1995infinite} and the more recent textbook \cite{ay2017information}.
It is well known that the Fisher--Rao metric is inherently connected to the negative entropy, as it can be regarded as its Hessian.
Much of information geometry is devoted to the study of the duality induced by this convex function. 
In particular, it induces a \emph{dually flat} structure on the space of probability measures that leads to the following notions of geodesics, which are dual to each other.
The $m$-geodesic and $e$-geodesic between two probability measures are given by linear interpolation and log-linear interpolation of the densities, respectively, meaning that
\begin{equ}
    \P_{t}^{(m)}=(1-t)\P_{0}+t\P_{1},\quad\text{and }\frac{\D\P_{t}^{(e)}}{\D\Q}\propto\left(\frac{\D\P_{0}}{\D\Q}\right)^{1-t}\left(\frac{\D\P_{1}}{\D\Q}\right)^{t},
\end{equ}
where $\Q$ is a reference measure dominating both $\P_{0}$ and $\P_{1}$.
Note that the $e$-geodesic is independent of the reference measure $\Q$.
The $e$-geodesic is also known as the \emph{exponential} geodesic, while the $m$-geodesic is also known as the \emph{mixture} geodesic.
They are not geodesics with respect to the Fisher--Rao metric, but with respect to the dually flat connections, which are not metric connections.
\begin{theorem}
[Fisher--Rao gradients]\label{thm:FR-gradients}
If $D_{\KL}(\P^\star, \P)<+\infty$ and $D_{\KL}(\P, \P^\star)<+\infty$, it holds that 
\begin{equs}
	\nabla_{\P}^{\FR}D_{\KL}(\P^{\star},\P)& =\P-\P^{\star},\quad\text{and }
	\\
	\nabla_{\P}^{\FR}D_{\KL}(\P,\P^{\star})&=\left(\ln\left(\frac{\D\P}{\D\P^{\star}}\right)-D_{\KL}(\P,\P^{\star})\right)\P.
\end{equs}
This shows that the negative Fisher--Rao gradients of forward and backward KL point into the directions of the $m$- and $e$-geodesics connecting $\P$ to $\P^{\star}$, respectively.
\end{theorem}
Despite both identities being known in the literature, we provide a short, self-contained proof in the appendix. These identities were used in various works obtaining $O(e^{-t})$ convergence guarantees for Fisher--Rao gradient flows~\cite{lu2019accelerating, liu2021second,lu2023birth, chen2026sampling}.
In contrast to Wasserstein gradient flows, the convergence rate is independent of the target measure $\P^{\star}$, which makes it an attractive choice for optimization and sampling.
As the space of probability measures is infinite-dimensional, it is typically not feasible to compute the Fisher--Rao gradient directly on the level of probability measures.
Hence, one resorts to discretizations of the space of probability measures, such as parametric families, which leads to the notion of natural gradients.
Given a Riemannian manifold $(\mathcal{M}, g)$, a parametric family $\mathcal{M}_{\Theta}=\{p_{\theta}:\theta\in\R^{d_{\textup{p}}}\}\subseteq\mathcal{M}$, and a loss function $L\colon\R^{d_{\textup{p}}}\to\R$ given by $L(\theta)=E(p_{\theta})$, we call a solution $\delta\theta$ to the linear system
\begin{equ}
	G(\theta)\delta\theta=\nabla L(\theta),
\end{equ}
\emph{natural gradient}, where $G(\theta)_{ij}=g_{p_{\theta}}(\partial_{\theta_{i}}p_{\theta},\partial_{\theta_{j}}p_{\theta})$. Here, $G(\theta)$ is
called the \emph{Gramian matrix}.
Natural gradients can also be shown to be the best approximation of the gradient descent dynamics in the manifold $\mathcal{M}$ while staying on the parametric model $\mathcal{M}_{\Theta}$.
To make this precise, we denote the parametrization map by $P=(\theta\mapsto p_{\theta})$.
Then, for any natural gradient direction $\delta\theta$ it holds that
\begin{equ}\label{eq:projection}
	DP(\theta)\delta\theta=\Pi_{T_{\theta}\mathcal{M}_{\Theta}}\nabla^{g}E(p_{\theta}),
\end{equ}
where $\Pi_{T_{\theta}\mathcal{M}_{\Theta}}$ denotes the $g$-orthogonal projection onto the generalized tangent space
\begin{equ}
	T_{\theta}\mathcal{M}_{\Theta}=\operatorname{span}\{\partial_{\theta_{i}}p_{\theta}:i=1,\ldots,d_{\textup{p}}\}.
\end{equ}
For a proof, we refer to \cite{amari2016information} for regular Gramians and to \cite{vanOostrum2023invariance} and \cite{muller2023achieving} for the degenerate and infinite-dimensional case, respectively.
In the specific case of the Fisher--Rao metric, the Gramian is given by the \emph{Fisher-information} matrix
\begin{equ}
	F(\theta)_{ij}\coloneqq g_{\P_{\theta}}^{\FR}(\partial_{\theta_{i}}\P_{\theta},\partial_{\theta_{j}}\P_{\theta})=\E_{\P_{\theta}}\left[\partial_{\theta_{i}}\ln\left(\frac{\D\P_{\theta}}{\D\Q}\right)\cdot\partial_{\theta_{j}}\ln\left(\frac{\D\P_{\theta}}{\D\Q}\right)\right],
\end{equ}
where $\Q$ is a reference measure dominating all $\P_{\theta}$.
Note that the Fisher-information matrix is independent of the choice of $\Q$.
\section{The Advective Fisher--Rao Metric}
In this section, we introduce the advective Fisher--Rao metric, which is the main object of this paper.
Our definition is motivated by an explicit expression of the pullback of the Fisher--Rao metric on path measures of SDEs to the space of velocity fields.
Further, we show that the advective Fisher--Rao metric maintains an important property of the Fisher--Rao metric, as it is compatible with the least-squares objective $E$.

\subsection{The Fisher--Rao metric of path measures of SDEs}
In the presence of noise
\begin{equ}
	\D X_{t}=v_{t}(X_{t})\D t+\sqrt{\varepsilon}\D W_{t}, \quad X_0\sim \rho_0, \label{eq:SDE-eps}
\end{equ}
the least-squares objective on the velocity fields agrees with the Kullback--Leibler divergence of the path measures.
The Kullback--Leibler divergence naturally gives rise to the Fisher--Rao metric, as the Fisher--Rao metric is the local Hessian of the KL-divergence and its gradients of the KL-divergence point into the directions of the $m$- and $e$-geodesics, see \Cref{thm:FR-gradients}.
This motivates us to study the Fisher--Rao metric on the space of path measures of SDEs, and we obtain an explicit expression.

\begin{theorem}[Fisher--Rao metric on path measures]\label{thm:FR-SDE}
Consider a velocity field $v_{\bullet}$ such that the SDE \eqref{eq:SDE-eps} admits a unique weak solution.
Further, consider two velocity fields $u_{\bullet}$ and $w_{\bullet}$ and assume that there is $\delta>0$ such that Novikov's condition
\begin{equs}
	\E_{\P^{v,\varepsilon}}\left[\exp\left(\frac{h^{2}}{2\varepsilon}\int_{0}^{1}\|u_{t}(X_{t})\|_{2}^{2}\dif t\right)\right] & <+\infty\quad\text{and}\label{eq:novikov-fisher-rao}\\
	\E_{\P^{v,\varepsilon}}\left[\exp\left(\frac{h^{2}}{2\varepsilon}\int_{0}^{1}\|w_{t}(X_{t})\|_{2}^{2}\dif t\right)\right] & <+\infty\label{eq:novikov-fisher-rao-2}
\end{equs}
holds for all $\lvert h\rvert<\delta$ and assume that the SDE \eqref{eq:SDE-eps}
admits a unique weak solution for $v_{\bullet}+hu_{\bullet}$ and $v_{\bullet}+hw_{\bullet}$
for all $\lvert h\rvert<\delta$.
Then $\Xi^{u,\varepsilon}=\ddh \P^{v+hu,\varepsilon}$ and $Z^{w,\varepsilon}=\ddh\P^{v+hw,\varepsilon}$ exist as elements in $T_{\P^{v, \varepsilon}} \mathcal{P}(\P^{v, \varepsilon})$ and the Fisher--Rao metric is given by 
\begin{equ}\label{eq:FR-SDE}
	g_{\P^{v,\varepsilon}}^{\FR}(\Xi^{u,\varepsilon},Z^{w,\varepsilon})=\frac{1}{\varepsilon}\int_{0}^{1}\int_{\R^d}u_{t}(x)\cdot w_{t}(x) \rho_{t}^{v,\varepsilon}(\D x)\dif t,
\end{equ}
where $\rho^{v,\varepsilon}_t$ denotes the law of $X_t$ under $\P^{v,\varepsilon}$. 
\end{theorem}
\begin{proof}
From the definition \eqref{eq:definition-PS-tangent} of the tangent space, the statement that $\Xi^{u, \varepsilon}$ and $Z^{w,\varepsilon}$ exist as elements in the tangent space $T_{\P^{v,\varepsilon}}\mathcal{P}(\P^{v,\varepsilon})$ is equivalent to the statement that the logarithmic densities 
\begin{equs}
    \ln\left(\frac{\D\mathbb{P}^{v+hu, \varepsilon}}{\D\mathbb{P}^{v, \varepsilon}}\right) \quad \textup{and } 
    \ln\left(\frac{\D\mathbb{P}^{v+hw, \varepsilon}}{\D\mathbb{P}^{v, \varepsilon}}\right)
\end{equs}
are differentiable with respect to $h$ in the Orlicz space $L^\Phi(\P^{v,\varepsilon})$, where $\Phi(t) = e^{\lvert t \rvert }-1$. 
We use Girsanov's theorem and, in particular, \eqref{eq:girsanov} to express the Radon--Nikodym derivative of the path measures and obtain
\begin{equation*}
    \ln\left(\frac{\D\mathbb{P}^{v+hu, \varepsilon}}{\D\mathbb{P}^{v, \varepsilon}}(x_{\bullet})\right)
    =
    h f(x_\bullet) - h^{2}g(x_\bullet),
\end{equation*}
where
\begin{equs}
    f(x_\bullet) & = \frac1\varepsilon \int_{0}^{1}u_{t}(x_{t})\cdot (\D x_{t}-v_{t}(x_{t})\D t) \quad \textup{and }  
    g(x_\bullet) = \frac1{2\varepsilon} \int_0^1 \lVert u_t(x_t) \rVert_2^2 \dif t
\end{equs}
do not depend on $h$. 
As $hf+h^2g$ is quadratic in $h$, it suffices to show $f, g\in L^\Phi(\P^{v, \varepsilon})$, meaning that 
\begin{equ}
    \E_{\P^{v,\varepsilon}}\left[e^{\alpha \lvert f \rvert}\right], \E_{\P^{v,\varepsilon}}\left[e^{\beta\lvert  g \rvert}\right] < +\infty 
\end{equ}
for some $\alpha, \beta>0$. 
From \eqref{eq:novikov-fisher-rao} we immediately obtain 
\begin{equ}
    \E_{\P^{v,\varepsilon}}\left[e^{\beta \lvert g \rvert}\right] 
    = 
    \E_{\P^{v,\varepsilon}}\left[\exp\left(\frac{\beta}{2\varepsilon}\int_{0}^{1}\|u_{t}(X_{t})\|_{2}^{2}\dif t\right)\right]
    < +\infty 
\end{equ}
for $\beta<\delta^2$. 
To show $\E_{\P^{v,\varepsilon}}[e^{\alpha \lvert f \rvert}]<+\infty$ we first bound $\E_{\P^{v,\varepsilon}}[e^{ \alpha f}]$. 
It holds that 
\begin{equs}
    \E_{\P^{v,\varepsilon}} \left[e^{ \alpha f}\right] 
    = 
    \E_{\P^{v,\varepsilon}}\left[ \exp\left( \frac\alpha{\sqrt{\varepsilon}} \int_0^1 u_t(X_t) \cdot \D W_t \right) \right]
\end{equs}
and therefore, we consider 
\begin{equ}
    M_t \coloneqq \frac{1}{\sqrt{\varepsilon}}\int_0^t u_s(X_s) \cdot \D W_s, 
\end{equ}
for $t\in[0,1]$, which is a square-integrable martingale as Novikov's condition \eqref{eq:novikov-fisher-rao} implies 
\begin{equs}
    \E_{\P^{v, \varepsilon}}\left[\int_0^1 \lVert u_t(X_t) \rVert^2_2 \dif t\right]<+\infty.
\end{equs}
Consider now the stochastic exponential of $2\alpha M_t$, which is given by 
\begin{equ}
    Z_t \coloneqq \exp\left( 2\alpha M_t - 2 \alpha^2 \langle M \rangle_t \right). 
\end{equ}
If $\alpha<\frac{\delta}{2}$ Novikov's condition \eqref{eq:novikov-fisher-rao} ensures 
\begin{equs}
    \E_{\P^{v, \varepsilon}}\left[\exp\left( \frac12 \langle 2 \alpha M \rangle_1 \right)\right]
    & = 
    \E_{\P^{v, \varepsilon}}\left[\exp\left( 2 \alpha^2 \langle M \rangle_1 \right)\right] 
    \\ & 
    = 
    \E_{\P^{v, \varepsilon}}\left[\exp\left( \frac{2 \alpha^2}{\varepsilon} \int_0^1 \lVert u_t(X_t) \rVert_2^2 \dif t \right)\right] < +\infty.
\end{equs}
Hence, for $\alpha<\frac{\delta}{2}$, the stochastic exponential is a martingale and it holds that 
\begin{equs}
    \E_{\P^{v,\varepsilon}} \left[e^{ \alpha f}\right]  & = \E_{\P^{v, \varepsilon}}\left[e^{\alpha M_1}\right] 
    \\ & = 
    \E_{\P^{v, \varepsilon}}\left[Z_1^\frac12 e^{\alpha^2 \langle M\rangle _1}\right] 
    \\ & 
    \le \E_{\P^{v, \varepsilon}}\left[Z_1\right]^\frac12 \E_{\P^{v, \varepsilon}}\left[e^{2\alpha^2 \langle M\rangle _1}\right]^\frac12 < +\infty. 
\end{equs}
To bound $\E_{\P^{v, \varepsilon}}[e^{-\alpha f}]$, note that $\tilde u = -u$ also satisfies the same Novikov condition \eqref{eq:novikov-fisher-rao} and therefore the same argument as above applies to $\tilde f = -f$. 
Hence, $\E_{\P^{v, \varepsilon}}[e^{-\alpha f}] = \E_{\P^{v, \varepsilon}}[e^{\alpha \tilde f}]<+\infty$ for $\alpha<\frac{\delta}{2}$. 
Finally, for $\alpha<\frac{\delta}{2}$ we have 
\begin{equ}
    \E_{\P^{v, \varepsilon}}\left[e^{\alpha \lvert f \rvert}\right] \le \E_{\P^{v, \varepsilon}}\left[e^{\alpha f}\right]+ \E_{\P^{v, \varepsilon}}\left[e^{-\alpha f}\right] < +\infty.
\end{equ}

Having established that $f,g\in L^\Phi(\P^{v, \varepsilon})$, we can differentiate with respect to $h$ and obtain 
\begin{equation*}
    \ddh\ln\left(\frac{\D\mathbb{P}^{v+hu, \varepsilon}}{\D\mathbb{P}^{v, \varepsilon}}(x_{\bullet})\right)
    = 
    \frac{1}{\varepsilon}\int_{0}^{1}u_{t}(x_{t})\cdot (\D x_{t}-v_{t}(x_{t})\D t)
\end{equation*}
and analogously 
\begin{equation*}
\ddh\ln\left(\frac{\D\mathbb{P}^{v+hw, \varepsilon}}{\D\mathbb{P}^{v, \varepsilon}}(x_{\bullet})\right)
=
\frac1\varepsilon\int_{0}^{1}w_{t}(x_{t})\cdot (\D x_{t}-v_{t}(x_{t})\D t).
\end{equation*}
Recall that $\varepsilon^{-\frac12}(\D X_{t}-v_{t}(X_{t})\D t)=\D W_{t}$ under $\mathbb{P}^{v, \varepsilon}$.
Using the It\^{o} isometry, we obtain from the definition \eqref{eq:log-expression-FR} of the Fisher--Rao metric that
\begin{equation*}
	\begin{split}
		g_{\mathbb{P}^{v, \varepsilon}}^{\textup{FR}}(\Xi^{u, \varepsilon},Z^{w, \varepsilon}) & =\mathbb{E}_{\mathbb{P}^{v, \varepsilon}}\bigg[\frac{1}{\varepsilon^2} \int_0^1 u_{t}(X_{t})\cdot (\D X_{t}-v_{t}(X_{t})\D t)\times
		\\
		& \qquad\qquad\int_{0}^{1}w_{t}(X_{t})\cdot (\D X_{t}-v_{t}(X_{t})\D t)\bigg]
		\\
		& =\frac1{\varepsilon}\cdot\mathbb{E}_{\mathbb{P}^{v, \varepsilon}}\left[\int_{0}^{1}u_{t}(X_{t})\cdot w_{t}(X_{t})\dif t\right]\\
		& = \frac1{\varepsilon}\int_{0}^{1}\int_{\R^d}u_{t}(x)\cdot w_{t}(x) \rho_t^{v,\varepsilon}(\D x) \dif t.
	\end{split}
\end{equation*}
\end{proof}

Novikov's condition can readily be replaced by weaker conditions ensuring that Girsanov's theorem holds. However, if both $u$ and $w$ have at most linear growth, then Novikov's condition is always satisfied.
To show this, we use basic notions of sub-Gaussian random variables and refer to \cite{vershynin2018high} for an extensive introduction to this concept.
Recall that a random variable $X$ is called \emph{sub-Gaussian} if either of the following two equivalent conditions holds
\begin{enumerate}
	\item There is $\lambda_{0}>0$ such that $\E[e^{\lambda\lVert X\rVert_{2}^{2}}]<+\infty$ for all $\lambda<\lambda_{0}$.
	\item There is $c>0$ such that $\P(\lVert X\rVert_{2}>t)\le2e^{-\frac{t^{2}}{c}}$.
\end{enumerate}
For a Gaussian $\mathcal{N}(0,\sigma^{2}I)$ one can choose $\lambda_{0}=(2\sigma^{2})^{-1}$. Finally, recall that the sum of two sub-Gaussians $X, Y$ with constants $c_{X},c_{Y}>0$ is again sub-Gaussian.
\begin{lemma}[Linear growth implies Novikov] \label{lem:novikov}
Let $v_{\bullet}$, $u_{\bullet}$ be velocity fields that satisfy the linear growth condition
\begin{equs}
    \|v_{t}(x)\|_{2}\le a_t(1+\|x\|_{2}) \text{ and }
    \|u_{t}(x)\|_{2}\le b_t(1+\|x\|_{2})\text{ for }t\in[0,1],x\in\R^{d},
\end{equs}
for some $a\in L^1_T$ and $b\in L^2_T$, assume that the SDE \eqref{eq:SDE-eps} has a unique weak solution and assume that the initial condition is sub-Gaussian.
Then, Novikov's condition \eqref{eq:novikov-fisher-rao} holds for a suitably small $\delta>0$.
\end{lemma}
\begin{proof}
We need to show that
\begin{equ}
   \left( \int_0^1 \lVert u_t(X_t) \rVert^2 \dif t\right)^{\frac12}
\end{equ}
is a sub-Gaussian random variable, where $X_t$ is the solution of the SDE \eqref{eq:SDE-eps}.
It holds that
\begin{equ}
    X_t = X_0 + \int_0^t v_s(X_s) \dif s + \sqrt{\varepsilon} W_t.
\end{equ}
The triangle inequality yields
\begin{equs}
    \lVert X_t \rVert & \le \lVert X_0 \rVert + \int_0^t \lVert v_s(X_s) \rVert \dif s + \sqrt{\varepsilon} \lVert W_t \rVert
    \\ &
    \le \lVert X_0 \rVert + \int_0^t a_s (1+\lVert X_s \rVert) \dif s + \sqrt{\varepsilon} \sup_{s\in[0,1]} \lVert W_s \rVert.
\end{equs}
Gronwall's inequality yields
\begin{equs}
    \lVert X_t \rVert & \le \left(\lVert X_0 \rVert + \lVert a \rVert_{L^1_T} + \sqrt{\varepsilon} \sup_{s\in[0,1]} \lVert W_s \rVert\right) e^{\lVert a \rVert_{L^1_T}}
    \\ & \le c Z,
\end{equs}
where $Z = 1+\lVert X_0 \rVert + \sqrt{\varepsilon} \sup_{s\in[0,1]} \lVert W_s \rVert$.
Note that $Z$ is sub-Gaussian as $\lVert X_0 \rVert$ is sub-Gaussian by assumption and $\sup_{s\in[0,1]} \lVert W_s \rVert$ is sub-Gaussian as the supremum of a Gaussian process.
We estimate
\begin{equ}
    \lVert u_t(X_t) \rVert \le b_t (1+\lVert X_t \rVert) \le b_t (1+c Z)
\end{equ}
and consequently
\begin{equ}
    \int_0^1 \lVert u_t(X_t) \rVert^2 \dif t \le \int_0^1 b_t^2 (1+c Z)^2 \dif t \le 2\lVert b \rVert_{L^2_T}^2 (1+c^2 Z^2).
\end{equ}
As $Z$ is sub-Gaussian, this finishes the proof.
\end{proof}

\Cref{thm:FR-SDE} provides the explicit expression
\begin{equs}
    g_{v_{\bullet}}^{\textup{FR}, \varepsilon}(u_{\bullet},w_{\bullet})
    & = g_{\P^{v, \varepsilon}}^{\textup{FR}}\left(\ddh \P^{v+hu, \varepsilon}, \ddh \P^{v+hw, \varepsilon}\right) 
    \\ & = \frac1\varepsilon \int_{0}^{1}\int_{\R^d}u_{t}(x)\cdot w_{t}(x)\rho_{t}^{v, \varepsilon}(\D x)\dif t\label{eq:def-FR-velocities}
\end{equs}
of the pullback of the Fisher--Rao metric on path measures of SDEs to the space of velocity fields under strong regularity conditions.  
This definition, however, is meaningful on a much larger space of velocity fields. 

\begin{definition}[Fisher--Rao metric on velocity fields]\label{def:FR-velocities}
Let $v_{\bullet}\in \V$ and let $\rho_{\bullet}^{v, \varepsilon}$  denote the densities of the solution of the SDE \eqref{eq:SDE-eps}. 
For two velocity fields $u_{\bullet},w_{\bullet}$ we refer to
\begin{equation}
	g_{v_{\bullet}}^{\textup{FR}, \varepsilon}(u_{\bullet},w_{\bullet})= \frac1\varepsilon \int_{0}^{1}\int_{\R^d}u_{t}(x)\cdot w_{t}(x)\rho_{t}^{v, \varepsilon}(\D x)\dif t\label{eq:def-FR-velocities}
\end{equation}
as the \emph{Fisher--Rao metric} on the space of velocity fields whenever this integral is well-defined. 
\end{definition}

\begin{remark}[Extensions]
The expression of the Fisher--Rao metric can readily be extended to the case of multiplicative noise as well as to cases where the initial condition is not fixed.
More precisely, in the case of multiplicative noise
\begin{equ}
    \D X_{t}=v_{t}(X_{t})\D t+\sigma_{t}(X_{t})\D W_{t},
\end{equ}
the Fisher--Rao metric is given by
\begin{equ}
    g_{\P^{v,\sigma}}^{\FR}(\Xi^{u,\sigma},Z^{w,\sigma})=\int_{0}^{1}\int_{\R^d}u_{t}(x)\cdot \left(\sigma_{t}(x)\sigma_{t}(x)^{\top}\right)^{-1}w_{t}(x)\rho_{t}^{v,\sigma}(\D x)\dif t.
\end{equ}
Here, we assume that $\sigma_{t}(x)^{\top}\sigma_{t}(x)$ is invertible for all $t\in[0,1]$ and $x\in \R^d$ and denote the path measure of the process by $\P^{v,\sigma}$ as well as the time martingals by $\rho^{v,\sigma}$.
In the case of a non-fixed initial condition, the Fisher--Rao metric is given by the sum of the Fisher--Rao metric on path measures with a fixed initial condition and the Fisher--Rao metric on the initial distribution.
Both statements follow from a general version of Girsanov's theorem incorporating multiplicative noise and non-fixed initial conditions, see \cite{richter2021solving}.
\end{remark}

\subsection{The advective Fisher--Rao metric and its optimality}
Equation \eqref{eq:FR-SDE} gives an explicit expression for the Fisher--Rao metric on the space of path measures induced by the SDE \eqref{eq:SDE-eps} that blows up as $\varepsilon\to0$.
It is natural to extend the Fisher--Rao metric by simply removing the factor of $\varepsilon^{-1}$, which
yields the following Riemannian metric.
\begin{definition}[Advective Fisher--Rao metric on velocity fields]\label{def:AFR-velocities}
Consider a velocity field $v_{\bullet}\in \V$ and denote the solution $\rho_{\bullet}^{v}$ of the continuity equation
\begin{equ}\label{eq:CE}
    \partial_t \rho_t^v + \nabla\cdot(\rho_t^v v_t)=0, \qquad \rho_0^v=\rho_0.
\end{equ}
For two velocity fields $u_{\bullet},w_{\bullet}$ we refer to
\begin{equation}
	g_{v_{\bullet}}^{\textup{AFR}}(u_{\bullet},w_{\bullet})=\int_{0}^{1}\int_{\R^d}u_{t}(x)\cdot w_{t}(x)\rho_{t}^{v}(\D x)\dif t\label{eq:def-FR-velocities}
\end{equation}
as the \emph{advective Fisher--Rao (AFR) metric} on the space of velocity fields whenever this integral is well-defined.
\end{definition}
\begin{remark}
    [Existence]
    Note that if $\rho_\bullet^v$ has uniformly bounded second moments, the metric is well defined for $u_\bullet, w_\bullet\in L^{2}_T C_{\textup{Lip}}^{1}$.
    This is for example the case if $\rho_0\in \mathcal{P}_2(\R^d)$ and $v_\bullet\in \V$, see \cite{ambrosio2012user}.
\end{remark}
Similar to the construction of the Wasserstein--Otto metric in optimal transport theory, we define the tangent space as the completion of the space of continuously differentiable Lipschitz vector fields with respect to the Fisher--Rao metric.
This yields
\begin{equ}
	T_{v_{\bullet}}^{\textup{AFR}} \V =\overline{L^{2}_T C_{\textup{Lip}}^{1}}^{g_{v_{\bullet}}^{\AFR}}=L^2_T L^2(\rho_{\bullet}^{v})
\end{equ}
where
\begin{equ}
    L^2_T L^2(\rho_{\bullet}^{v}) \coloneqq \left\{ v\colon[0,1]\times\R^d\to\R^d : \int_0^1 \int_{\R^d} \lVert v_t(x) \rVert^2 \rho_t^v(\D x)  \dif t < +\infty \right\}.
\end{equ}
Note that the Fisher--Rao metric depends on $v_{\bullet}$, unlike the $L^{2}$-metric, where the reference density $\rho_{\bullet}^{\star}$ is used.
Recall that we are investigating the energy function
\begin{equ}
    E\colon L^2_T C^1_{\textup{Lip}} \to \R, \quad v_\bullet \mapsto \frac12 \int_0^1 \int_{\R^d} \lVert v_t(x) - v^\star_t(x) \rVert_2^2 \rho^\star_t(\D x) \dif t.
\end{equ}
To ensure finiteness of the energy we assume $\rho_0\in\mathcal{P}_2(\R^d)$, in which case $\rho_t^\star$ has finite second moments $m_2(\rho_t^\star)$, which are bounded uniformly in time, see for example \cite{ambrosio2012user} or Lemma \ref{lem:evolution-W2}.
Note that
\begin{equ}
    \lVert v_t(x) - v^\star_t(x) \rVert \le \lVert v_t - v_t^\star \rVert_{C^1_{\textup{Lip}}}(1+\lVert x \rVert)
\end{equ}
and we can estimate
\begin{equs}
    \lVert v_t - v^\star_t \rVert_{L^2(\rho^\star_t)}
    & \le
    \lVert v_t - v_t^\star \rVert_{C^1_{\textup{Lip}}} \lVert 1+\lVert x \rVert \rVert_{L^2(\rho^\star_t)}
    \\ & \le
    \lVert v_t - v_t^\star \rVert_{C^1_{\textup{Lip}}} (1+\lVert \lVert x \rVert \rVert_{L^2(\rho^\star_t)})
    \\ & =
    \lVert v_t - v_t^\star \rVert_{C^1_{\textup{Lip}}} (1+\sqrt{m_2(\rho_t^\star)}).
    \\ & =
    c \lVert v_t - v_t^\star \rVert_{C^1_{\textup{Lip}}}.
\end{equs}
This shows the finiteness of the energy $E$ if $v_\bullet, v^\star_\bullet \in L^2_T C^1_{\textup{Lip}}$ and $p_0\in\mathcal{P}_2(\R^d)$.
To understand the dynamics of the advective Fisher--Rao gradient flow, we first derive an explicit expression in the space of velocity fields.
\begin{lemma}\label{thm:FR-gradient-velocity}
Assume that the initial distribution $p_0\in\mathcal{P}_2(\R^d)$ has a positive density and consider $v_\bullet, v^\star_\bullet \in L^2_TC^1_{\textup{Lip}}$.
Then for any $u_\bullet \in L^2_TC^1_{\textup{Lip}}$ it holds that
\begin{equ}
    g^{\AFR}_{v_\bullet}\left( (v_{\bullet}-v_{\bullet}^{\star})\cdot\frac{p_{\bullet}^{\star}}{p_{\bullet}^{v}}, u_\bullet\right) = \delta E(v_\bullet)[u_\bullet].
\end{equ}
\end{lemma}
\begin{proof}
First, note that positivity is preserved by the continuity equation.
A direct computation yields that for any $u_{\bullet}\in L^2_T C^1_{\textup{Lip}}$ it holds that
\begin{equation*}
\begin{split}
	\delta E(v_{\bullet})[u_{\bullet}]
    & =\lim_{h\to0}\frac{E(v_\bullet+hu_\bullet)-E(v_\bullet)}{h} \\
    & =\lim_{h\to0}\int_{0}^{1}\int_{\R^d}\left(u_t(x)\cdot(v_t(x)-v^\star_t(x)) + h \lVert u_t(x) \rVert^2\right) p_{t}^{\star}(x)\dif x\dif t \\
    & =\int_{0}^{1}\int_{\R^d}(v_{t}(x)-v_{t}^{\star}(x))\cdot u_{t}(x)p_{t}^{\star}(x)\dif x\dif t
	\\&
	=g_{v_{\bullet}}^{\textup{AFR}}\left((v_{\bullet}-v_{\bullet}^{\star})\cdot\frac{p_{\bullet}^{\star}}{p^v_{\bullet}},u_{\bullet}\right),
\end{split}
\end{equation*}
which yields the claim.
\end{proof}
Note that it is not clear whether
\begin{equ}
    (v_\bullet -v_\bullet^\star)\cdot\frac{p_\bullet^\star}{p_\bullet^v} \in L^2_TL^2(p^v_\bullet) = T_{v_\bullet}^{\AFR} \V.
\end{equ}
As such, it is not guaranteed that the gradient of the energy functional exists in the tangent space, hence in the classical sense.
However, it always exists in the larger space $L^2_T L^1(p_\bullet^v)$, as can be checked readily.
Nevertheless, we write $\nabla^{\AFR}E(v_{\bullet})=(v_{\bullet}-v_{\bullet}^{\star})\cdot\frac{p_{\bullet}^{\star}}{p_{\bullet}^{v}}$ in slightly informal fashion.
To understand how the advective Fisher--Rao gradient flow affects the densities $p_{\bullet}^{v}$, we use the following general result.
\begin{lemma}[Derivative of the solution operator]\label{thm:derivative-solution}
Let $v_\bullet,u_\bullet\in \V$ and assume that $\rho_0$ has a finite second moment. 
Then, the limit
\begin{equ}\label{eq:weak-star-convergence}
    \xi_{t}^{u} \coloneqq\lim_{h\to0}\frac{\rho_{t}^{v+hu}-\rho_{t}^{v}}{h}
\end{equ}
exists for every $t\in[0,1]$ as a weak-$\ast$ limit in $(C^1_{\textup{Lip}})^\ast$ as well as a weak-$\ast$ limit in $(L^1_T C^1_{\textup{Lip}})^\ast$.
For any $\phi\in C^1_{\textup{Lip}}$ it holds that
\begin{equ}
    \langle \xi_t^{u}, \phi\rangle = \int_{\R^d} \nabla \phi(x) \cdot \psi_t^{u}((\varphi^{v}_t)^{-1}(x)) \rho_t(\D x),
\end{equ}
where $\varphi^v_\bullet$ denotes the flow induced by $v_\bullet$ and $\psi^{u}_\bullet$ is the unique solution to
\begin{equ}
    \partial_{t}\psi_{t}^{u}(x)=Dv_{t}(\varphi_{t}^{v}(x))\cdot\psi_{t}^{u}(x)+u_{t}(\varphi_{t}^{v}(x)).
\end{equ}
Further, $\xi_{\bullet}^{u}$ is continuous as a curve in $(C^1_{\textup{Lip}})^\ast$ with respect to the weak-$\ast$ topology and is a distributional solution of the continuity equation with source term given by
\begin{equ}
	\partial_{t}\xi_{t}^{u}+\nabla\cdot(\xi_{t}^{u}v_{t})=-\nabla\cdot(\rho_{t}^{v}u_{t}) \quad \text{and } \xi_0=0.
    \label{eq:linear-transport}
\end{equ}
\end{lemma}
\begin{proof}
First, we show convergence of the difference quotients $\Delta_{t}^{h}\coloneqq\frac{\rho_{t}^{v+hu}-\rho_{t}^{v}}{h}$ for a fixed $t\in[0,1]$.
For this, we take a function $\phi\in  C_{\textup{Lip}}^{1}(\R^{d})$, and compute
\begin{equs}\label{eq:proof-derivative}
    \langle \Delta_t^h, \phi \rangle & =h^{-1}\left(\int_{\R^{d}}\phi(x)\rho_{t}^{v+hu}(\D x)-\int_{\R^{d}}\phi_t(x) \rho_{t}^{v}(\D x) \right)\\
    & =\int_{\R^{d}}\frac{\phi(\varphi_{t}^{v+hu}(x))-\phi(\varphi_{t}^{v}(x))}{h}\rho_{0}(\D x),
\end{equs}
where $\varphi^{v+hu}_\bullet$ denotes the flow induced by the velocity field $v_{\bullet}+hu_\bullet$.
By Lemma \ref{app:lem:sensitivity}, the pointwise limit is given by
\begin{equ}
    \lim_{h\to0} \frac{\phi(\varphi_{t}^{v+hu}(x))-\phi(\varphi_{t}^{v}(x))}{h} = \nabla \phi(\varphi_{t}^{v}(x))\cdot\psi_{t}^{u}(x)
\end{equ}
for all $t\in[0,1]$, $x\in\R^d$.
By Lemma \ref{app:lem:stability}, there is a constant $c>0$ such that
\begin{equ}
    \lVert \varphi^{v+hu}_t(x) - \varphi_t^{u}(x) \rVert \le c(1+\lVert x \rVert) \quad \textup{for all } x\in\R^d, \lvert h \rvert \le1.
\end{equ}
This yields
\begin{equs}
    \left \lvert \frac{\phi(\varphi_{t}^{v+hu}(x))-\phi(\varphi_{t}^{v}(x))}{h} \right\rvert
    & \le \lVert \phi \rVert_{C^1_{\textup{Lip}}} \cdot \frac{\lVert \varphi^{v+hu}_t(x) - \varphi_t^{u}(x) \rVert }{h}
    \\ & \le \lVert \phi \rVert_{C^1_{\textup{Lip}}} \cdot c(1+\lVert x \rVert).
\end{equs}
This provides a dominating function as
\begin{equs}
    \int_{\R^{d}}
    \left \lvert \frac{\phi(\varphi_{t}^{v+hu}(x))-\phi(\varphi_{t}^{v}(x))}{h} \right\rvert \rho_{0}(\D x) \le \lVert \phi \rVert_{ C^1_{\textup{Lip}}} c(1+m_1(\rho_0))
\end{equs}
and hence, we obtain
\begin{equs}
    \lim_{h\to0} \langle \Delta_t^h, \phi \rangle
    & =  \int_{\R^d} \nabla \phi(\varphi^{v}_t(x)) \cdot \psi_t^{u}(x) \rho_0(\D x)
    \\ & = \int_{\R^d} \nabla \phi(x) \cdot \psi_t^{u}((\varphi^{v}_t)^{-1}(x)) \rho_t(\D x)
    \\ & \eqqcolon \langle \xi_t^{u}, \phi \rangle.
\end{equs}
Note that this defines a linear bounded functional $\xi_t^{u}$ with respect to $\phi$ as $\psi_t^{u}$ has linear growth with a constant bounded uniformly in $t$, see for example Lemma \ref{app:lem:flow-grow}, which ensures
\begin{equs}
    \left\lvert \int_{\R^d} \nabla \phi(\varphi^{v}_t(x)) \cdot \psi_t^{u}(x) \rho_0(\D x) \right\rvert
    & \le
    \int_{\R^d} \left\lvert \nabla \phi(\varphi^{v}_t(x)) \cdot \psi_t^{u}(x) \right\rvert \rho_0(\D x)
    \\ & \le
    \lVert \phi \rVert_{C^1_{\textup{Lip}}}  \int_{\R^d}  \lVert \psi_t^{u}(x) \rVert \rho_0(\D x)
    \\ & \le
     \lVert \phi \rVert_{C^1_{\textup{Lip}}} \int_{\R^d} c(1+\lVert x \rVert) \rho_0(\D x)
    \\ & \le c(1+m_1(\rho_0)) \lVert \phi \rVert_{C^1_{\textup{Lip}}}.
\end{equs}
Note that the argument above also provides
\begin{equs}\label{proof:eq:weak-stat-time}
    \lim_{h\to0} \int_0^1 \langle \Delta_t^h, \phi_t \rangle \dif t
    =
    \int_0^1 \langle \xi_t^{u}, \phi_t \rangle \dif t
\end{equs}
for any $\phi_\bullet\in L^1_TC^1_{\textup{Lip}}$ as again
\begin{equs}
    \lvert \langle \xi_t^{u}, \phi_t \rangle \rvert
    & \le \lVert \phi_t \rVert_{C^1_{\textup{Lip}}} \cdot c(1+ m_1(\rho_0)),
\end{equs}
which provides a dominating function in time.
To see that the limit $\xi_{\bullet}^{u}$ is a distributional solution of \eqref{eq:linear-transport}, we first note that for a test function $\phi_\bullet\in C^\infty_{\textup{c}}((0,1)\times\R^d;\R^d)$ the continuity equation ensures
\begin{equation*}
    \int_{0}^{1}\left\langle \rho_{t}^{v+hu}, \partial_{t}\phi_{t}+v_{t}\cdot \nabla\phi_{t}+hu_{t}\cdot \nabla\phi_{t}\right\rangle\dif t=0
\end{equation*}
for any $h\in\R$.
Taking the difference quotient, we obtain
\begin{equation*}
    \int_{0}^{1}\left( \left\langle \rho_{t}^{v+hu}, u_{t}\cdot \nabla\phi_{t}\right\rangle  + \left\langle \Delta_{t}^{h}, \partial_{t}\phi_{t}+v_{t}\cdot \nabla\phi_{t}\right\rangle \right)\dif t=0.
\end{equation*}
From \eqref{proof:eq:weak-stat-time} we obtain
\begin{equ}
    \lim_{h\to0} \int_{0}^{1} \left\langle \Delta_{t}^{h}, \partial_{t}\phi_{t}+v_{t}\cdot \nabla\phi_{t}\right\rangle \dif t
    =
    \int_{0}^{1} \left\langle \xi_t^{u}, \partial_{t}\phi_{t}+v_{t}\cdot \nabla\phi_{t}\right\rangle \dif t
\end{equ}
as $\partial_t \phi_\bullet + v_\bullet\cdot\nabla\phi_\bullet \in L^1_TC^1_{\textup{Lip}}$.
To treat the second term, we use that $p^{v+hu}_t \to p_t^v$ in $W_2$ as $h \to 0$, see Lemma \ref{lem:W2-stability} below. 
Since $W_2$ convergence implies convergence when testing with functions with at most quadratic growth, see \cite[Proposition 3.4]{ambrosio2012user}, and $u_\bullet\cdot \nabla \phi_\bullet\in L^1_TC^1_{\textup{Lip}}(\R^d)$, we can use dominated convergence to obtain
\begin{equ}
    \lim_{h\to0} \left\langle \rho_{t}^{v+hu}, u_{t}\cdot \nabla\phi_{t}\right\rangle  = \lim_{h\to0} \left\langle \rho_{t}^{v}, u_{t}\cdot \nabla\phi_{t}\right\rangle
\end{equ}
for a fixed time $t\in[0,1]$.
Finally, we obtain a majorant in time, as for any $\lvert h \rvert\le1$ we can estimate
\begin{equs}
    \left\lvert \left\langle \rho_t^{v+hu}, u_t\cdot \nabla \phi_t \right\rangle \right\rvert
    & \le \int_{\R^d} \lvert u_t(x)\cdot \nabla \phi_t(x) \rvert \rho_t^{v+hu}(\D x)
    \\ & \le \int_{\R^d} \lVert u_t(x) \rVert \cdot \lVert \nabla \phi_t(x) \rVert \rho_t^{v+hu}(\D x)
    \\ & \le \lVert \nabla \phi_t \rVert_{C^1_{\textup{Lip}}} \int_{\R^d} c(1+ \lVert \varphi^{v+hu}_t(x)\rVert)  \rho_0(\D x)
    \\ & \le \lVert \nabla \phi_t \rVert_{C^1_{\textup{Lip}}} \int_{\R^d} C(1+ \lVert x\rVert)  \rho_0(\D x)
    \\ & \le \lVert \nabla \phi_t \rVert_{C^1_{\textup{Lip}}} C (1+m_1(\rho_0)) ,
\end{equs}
where we used that the flow $\varphi_t^{v+hu}$ has at most linear growth with a constant uniform in time and $\lvert h \rvert \le 1$, see Lemma \ref{app:lem:flow-grow}.
Overall, we obtain
\begin{equs}
    \int_0^1 \langle \xi_t^{u}, \partial_t \phi_t + v_t\cdot\nabla\phi_t \rangle \dif t = -\int_0^1 \langle \rho_t^{v}, u_t\cdot\nabla \phi_t \rangle \dif t,
\end{equs}
which shows that $\xi_\bullet^{u}$ is a distributional solution to
\begin{equ}
    \partial_t \xi_t + \nabla \cdot(\xi_t v_t) = - \nabla \cdot (\rho_t u_t).
\end{equ}
\end{proof}
\begin{theorem}[Fisher--Rao gradient and $m$-geodesics]\label{thm:FR-gradient-densities}
Consider $v_{\bullet},v_{\bullet}^{\star}\in \V$ as well as an initial distribution $\rho_0\in\mathcal{P}_2(\R^d)$. 
Assume that $\rho^\star_\bullet\ll\rho^v_\bullet$ and that $u_{\bullet}\coloneqq\nabla^{\FR}E(v_{\bullet})=(v_{\bullet}-v_{\bullet}^{\star})\cdot\frac{\D \rho_{\bullet}^{\star}}{\D \rho_{\bullet}^{v}}\in \V$.
Then, it holds that
\begin{equation}
    \ddh \rho_{\bullet}^{v+hu}=\rho_{\bullet}^{v}-\rho_{\bullet}^{\star}.
\end{equation}
\end{theorem}
\begin{proof}
We set $\xi_{\bullet}\coloneqq\ddh \rho_{\bullet}^{v+hu}$. 
By Lemma \ref{thm:derivative-solution} we have
\begin{equs}
\partial_{t}\xi_{t}+ & \nabla\cdot\left(\xi_{t}v_{t}\right)=-\nabla\cdot\left(\rho_{t}u_{t}\right).
\end{equs}
Compare this to the tangent vector $\zeta_{\bullet}=p_{\bullet}^{v}-p_{\bullet}^{\star}$, for which the continuity equation yields
\begin{equs}
	\partial_{t}\zeta_{t} & =\partial_{t}\rho_{t}^{v}-\partial_{t}\rho_{t}^{\star}\\
 	& =-\nabla\cdot(\rho_{t}^{v}v_{t}-\rho_{t}^{\star}v_{t}^{\star})\\
 	& =-\nabla\cdot(\zeta_{t}v_{t}+\rho_{t}^{\star}v_{t}-\rho_{t}^{\star}v_{t}^{\star})\\
 	& =-\nabla\cdot(\zeta_{t}v_{t}+\rho_{t}u_{t}).
\end{equs}
Note that both $\xi_\bullet$ and $\zeta_\bullet$ are weak-$\ast$ continuous in $(C^1_{\textup{Lip}})^\ast$ with respect to $t$ and satisfy $\xi_{0}=\zeta_{0}=0$. Subtracting the two equations, we obtain
\begin{equation*}
    \partial_{t}(\xi_{t}-\zeta_{t})+\nabla\cdot(v_{t}(\xi_{t}-\zeta_{t}))=0
\end{equation*}
as well as $\xi_{0}-\zeta_{0}=0$.
By the method of characteristics, $0$ is the unique solution of this equation, which yields $\xi_{\bullet}=\zeta_{\bullet}$.
\end{proof}

Informally, this result suggests that, given its existence, the advective Fisher--Rao gradient flow 
\begin{equ}
    \partial_\tau v^\tau_\bullet = -\nabla^{\AFR} E(v^\tau_{\bullet})
\end{equ}
leads to dynamics of the probability densities $\rho^\tau_\bullet=\rho^{v^\tau}_\bullet$, which are given by
\begin{equ}
    \rho^{\tau}_\bullet = \rho^{\star}_\bullet + e^{-\tau} (\rho^{0}_\bullet-\rho^{\star}_\bullet).
\end{equ}
This means that the density path $p^{\tau}_\bullet$ converges to $p^\star_\bullet$ at a rate of $O(e^{-\tau})$ along the linear interpolation, which is known as a mixture geodesic in information geometry \cite{amari2016information, ay2017information}.
In particular, this shows that the advective Fisher--Rao metric is compatible with $E$, thereby providing an affirmative answer to the original question.

\begin{remark}[Extension]\label{rem:non-euclidean-loss}
Consider the case that the least-squares objective is not with respect to the Euclidean norm, but given by 
\begin{equ}
    E(v_\bullet) \coloneqq \frac12 \int_0^1 \int_{\R^d} \lVert \sigma_t(x)^{-1}(v_t(x) - v_t^\star(x))\rVert_2^2 \rho^\star_t(\D x) \dif t, 
\end{equ}
where, we assume that $\sigma_{t}(x)$ is invertible for all $t\in[0,1]$ and $x\in \R^d$. 
Recall that for the noisy case with multiplicative noise
\begin{equ}
    \D X_{t}=v_{t}(X_{t})\D t+\sigma_{t}(X_{t})\D W_{t},
\end{equ}
the Fisher--Rao metric is given by
\begin{equ}
    g_{\P^{v,\sigma}}^{\FR}(\Xi^{u,\sigma},Z^{w,\sigma})=\int_{0}^{1}\int_{\R^d}u_{t}(x)\cdot \left(\sigma_{t}(x)\sigma_{t}(x)^{\top}\right)^{-1}w_{t}(x)\rho_{t}^{v,\sigma}(\D x)\dif t,
\end{equ}
where we denote the path measure of the process by $\P^{v,\sigma}$ as well as the time marginals by $\rho^{v,\sigma}$. 
A natural extension of the advective Fisher--Rao metric to a non-Euclidean geometry is now given by 
\begin{equ}
    g_{v_\bullet}^{\AFR, \sigma}(u_\bullet, w_\bullet)=\int_{0}^{1}\int_{\R^d}u_{t}(x)\cdot \left(\sigma_{t}(x)\sigma_{t}(x)^{\top}\right)^{-1}w_{t}(x)\rho_{t}^{v}(\D x)\dif t.
\end{equ}
\end{remark}

\section{The Advective Fisher--Rao Geometry from Three Perspectives}
We have motivated the advective Fisher--Rao metric via an explicit expression of the Fisher--Rao metric of path measures of SDEs.
Here, we show that the rescaled zero-noise limit of these Fisher--Rao metrics does indeed converge to the advective Fisher--Rao metric.
Motivated by this characterization as a zero-noise limit, we connect the advective Fisher--Rao metric to the Freidlin--Wentzell rate functional, which describes the concentration of the path measures of the SDE around the ODE.
Finally, we study the relation of the advective Fisher--Rao metric to optimal transport and show that it agrees with the Hessian of the Benamou--Brenier action functional.
\subsection{Information geometry}
A central object in information geometry is the Fisher--Rao metric, which arises naturally from the Hessian of the Kullback-- Leibler divergence.
We now show that the advective Fisher--Rao metric on velocity fields can be obtained as the rescaled zero-noise limit of the Fisher--Rao metric on path measures.
\begin{theorem}\label{thm:zero-noise-limit}
Consider $v_{\bullet}\in\V$ and assume that $\rho_0\in\mathcal{P}_2(\R^d)$.
Then we obtain the advective Fisher--Rao metric as the rescaled zero-noise limit of the Fisher--Rao metrics on path measures, that is, for any $u_{\bullet},w_{\bullet}\in L^{2}_{T}C_{\textup{Lip}}^{1}$ we have that
\begin{equation*}
    g_{v_{\bullet}}^{\textup{AFR}}(u_{\bullet},w_{\bullet})=\lim_{\varepsilon\to0}\varepsilon g_{v_\bullet}^{\textup{FR},\varepsilon}(u_\bullet,w_\bullet).
\end{equation*}
\end{theorem}
We require some auxiliary results.
\begin{lemma}[Wasserstein convergence of time marginals] \label{lem:wasserstein-convergence-marginals}
Consider a velocity field $v_{\bullet}$, which satisfies the one-sided Lipschitz condition
\begin{equ}
    (x-y)\cdot(v_t(x) - v_t(y)) \leq \beta_t \|x-y\|^2 \quad\text{for all }t\in[0,1],x,y\in\R^d
\end{equ}
for some $\beta\in L^1_T$ and assume that for every $\varepsilon\ge0$ the SDE \eqref{eq:SDE-eps} admits a unique weak solution with path measure $\P^{v,\varepsilon}$ and assume that the corresponding time marginals
$\rho_t^{v,\varepsilon}$ have finite second moments, which are uniformly bounded in $\varepsilon\in[0,\delta]$ and $t\in[0,1]$ for some $\delta>0$.
Then for all $\varepsilon\in[0,\delta]$ and $t\in[0,1]$ it holds that
\begin{equ}
    W_2^2(\rho^{v,\varepsilon}_t, \rho^v_t) \leq\varepsilon d t e^{2\int_0^t \beta_s \dif s}.
\end{equ}
\end{lemma}
\begin{proof}
Let us denote the solutions of the SDE \eqref{eq:SDE-eps} and the corresponding ODE by $X_t^\varepsilon$ for $\varepsilon>0$ and $X_t$ for $\varepsilon=0$, respectively.
Then, $(X_t^\varepsilon, X_t)$ is a coupling of $\rho_t^{v,\varepsilon}$ and $\rho_t^v$ and hence
\begin{equation*}
    W_2^2(\rho^{v,\varepsilon}_t, \rho^v_t) \leq \mathbb{E}\left[\|X_t^\varepsilon - X_t\|^2\right].
\end{equation*}
We set $Z_t = X_t^\varepsilon - X_t$, which satisfies the SDE
\begin{equation*}
    \D Z_t = (v_t(X_t^\varepsilon) - v_t(X_t)) \D t + \sqrt{\varepsilon} \D W_t \quad Z_0=0.
\end{equation*}
Using It\^{o}'s formula, we obtain
\begin{equation*}
    \|Z_t\|^2 = 2\int_0^t Z_s\cdot (v_s(X_s^\varepsilon) - v_s(X_s)) \D s + 2\sqrt{\varepsilon} \int_0^t Z_s\cdot \D W_s + d\varepsilon t.
\end{equation*}
By assumption, $X_t^\varepsilon$ and $X_t$ have uniformly bounded second moments and hence
$\E[\int_0^1 \lVert Z_t\rVert_2^2 \dif t] < +\infty$.
This shows that the stochastic integral $\int_0^t Z_s\cdot \D W_s$ is a martingale with expectation zero and hence its expectation vanishes.
Using the one-sided Lipschitz continuity of $v_{\bullet}$, we obtain
\begin{equation*}
    \mathbb{E}\left[\|Z_t\|^2\right] \leq 2\int_0^t \beta_s \mathbb{E}\left[\|Z_s\|^2\right] \dif s + d\varepsilon t.
\end{equation*}
Now, by Gr\"{o}nwall's inequality we obtain
\begin{equation*}
    \mathbb{E}\left[\|Z_t\|^2\right] \leq d\varepsilon t e^{2\int_0^t \beta_s \dif s},
\end{equation*}
which finishes the proof.
\end{proof}
\begin{lemma}[Bounding second moments by Wasserstein-$2$]\label{lem:second-moment-via-W2}
    Consider two probability measures $\mu,\nu\in\mathcal{P}_2(\R^d)$ with finite second moments $m_2(\mu)$ and $m_2(\nu)$.
    Then it holds that
    \begin{equ}
        \left\lvert \sqrt{m_2(\mu)} - \sqrt{m_2(\nu)} \right\rvert \leq W_2(\mu, \nu)
    \end{equ}
    as well as
    \begin{equ}
        m_2(\mu) \le m_2(\nu) + 2\sqrt{m_2(\nu)} W_2(\mu, \nu) + W_2^2(\mu, \nu).
    \end{equ}
\end{lemma}
\begin{proof}
It holds that
\begin{equ}
    \lVert x \rVert \le \lVert y \rVert + \lVert x-y \rVert\quad\text{for all }x,y\in\R^d.
\end{equ}
Hence, if we set $f(x,y)\coloneqq \rVert x \rVert$, $g(x,y)\coloneqq \rVert y \rVert$, and $h(x,y)\coloneqq \rVert x-y \rVert$ we obtain for any coupling $\pi$ of $\mu$ and $\nu$ that
\begin{equ}
    \lVert f \rVert_{L^2(\pi)}
    \le \lVert g+h\rVert_{L^2(\pi)}
    \leq \lVert g \rVert_{L^2(\pi)} + \lVert h \rVert_{L^2(\pi)}.
\end{equ}
Note that $\lVert f \rVert_{L^2(\pi)} = \sqrt{m_2(\mu)}$, $\lVert g \rVert_{L^2(\pi)} = \sqrt{m_2(\nu)}$ for any coupling $\pi$.
Taking the infimum over all couplings $\pi$ yields
\begin{equ}
    \sqrt{m_2(\mu)}
    \leq \sqrt{m_2(\nu)} + W_2(\mu, \nu).
\end{equ}
By symmetry, we also obtain $\sqrt{m_2(\nu)} \leq \sqrt{m_2(\mu)} + W_2(\mu, \nu)$ and hence the desired result.
Finally, the second estimate follows from the first one as
\begin{equs}
    m_2(\mu)-m_2(\nu) & = \left(\sqrt{m_2(\mu)}-\sqrt{m_2(\nu)}\right)\left(\sqrt{m_2(\mu)}+\sqrt{m_2(\nu)}\right) \\
    & \leq W_2(\mu, \nu)\left(\sqrt{m_2(\mu)}+\sqrt{m_2(\nu)}\right) \\
    & \leq W_2(\mu, \nu)\left(\sqrt{m_2(\nu)}+\sqrt{m_2(\nu)}+W_2(\mu, \nu)\right) \\
    & = 2\sqrt{m_2(\nu)} W_2(\mu, \nu) + W_2^2(\mu, \nu).
\end{equs}
\end{proof}

\begin{lemma}[Growth of Wasserstein-$2$]\label{lem:evolution-W2}
Consider a velocity field $v_\bullet$ such that there are functions
$\alpha,\beta\in L^1_T$ satisfying
\begin{equ}
    \lVert v_t(0) \rVert \le \alpha_t, \quad \lVert v_t(x)-v_t(y) \rVert \le \beta_t \lVert x -y\rVert \quad \text{for all } t\in[0,1], x, y\in\R^d,
\end{equ}
and consider $\rho_0\in\mathcal{P}_2(\R^d)$, and denote the solution of the continuity equation $\rho_{\bullet}^{v}$.
Then we have
\begin{equ}
    W_2(\rho_t^v, \rho_0) \le \lVert \alpha \rVert_{L^1([0,t])} e^{\lVert \beta \rVert_{L^1([0,t])}} + \sqrt{m_2(\rho_0)} \left(e^{\lVert \beta \rVert_{L^1([0,t])}} -1  \right)
\end{equ}
and
\begin{equ}
    \sqrt{m_2(\rho_t^v)} \le \lVert \alpha \rVert_{L^1([0,t])} e^{\lVert \beta \rVert_{L^1([0,t])}} + \sqrt{m_2(\rho_0)} e^{\lVert \beta \rVert_{L^1([0,t])}}.
\end{equ}
\end{lemma}
\begin{proof}
By the definition of the Wasserstein distance, we have
\begin{equ}
    W_2^2(\rho_t^v, \rho_0) \le \int_{\R^d} \lVert \varphi_t^v(x) - x \rVert^2 \rho_0(\D x) =
    \lVert \varphi_t^v(x) - x \rVert_{L^2(\rho_0)}^2.
\end{equ}
Note that pointwise $\lVert \varphi_t^v(x) - x \rVert \le A + B \lVert x\rVert$ for
\begin{equ}
    A = \lVert \alpha \rVert_{L^1([0,t])}e^{\lVert \beta \rVert_{L^1([0,t])}} \quad  \text{and } B=(e^{\lVert \beta\rVert_{L^1([0,t])}}-1)
\end{equ}
follows from \Cref{app:lem:flow-grow}.
Hence, using the triangle inequality, we obtain
\begin{equ}
    \lVert \varphi_t^v(x) - x \rVert_{L^2(\rho_0)} \le
    \lVert A + B \lVert x \rVert \rVert_{L^2(\rho_0)} \le
    A + B \lVert \lVert x \rVert \rVert_{L^2(\rho_0)},
\end{equ}
where $\lVert \lVert x \rVert \rVert_{L^2(\rho_0)}=\sqrt{m_2(\rho_0)}$.
The proof for the bound of the second moment follows analogously.
\end{proof}

\begin{proof}[Proof of \Cref{thm:zero-noise-limit}]
From Definition \ref{def:FR-velocities} we obtain 
\begin{equ}
    \varepsilon g_{\P^{v,\varepsilon}}^{\textup{FR}}(u, w)
    =
    \int_{0}^{1}\int_{\R^d}u_{t}(x)\cdot w_{t}(x)\rho_{t}^{v,\varepsilon}(\D x)\dif t.
\end{equ}
First note that for fixed $t\in[0,1]$ we have $\rho_{t}^{v,\varepsilon}\to \rho_{t}^{v}$ in $\mathcal{P}_2(\R^d)$ by Lemma \ref{lem:wasserstein-convergence-marginals}.
As the integrand $\lvert u_{t}(x)\cdot w_{t}(x)\rvert$ is continuous and has at most quadratic growth in $x$ it holds that
\begin{equ}
    \lim_{\varepsilon\to0}\int_{\R^d}  u_{t}(x)\cdot w_{t}(x) \rho_{t}^{v,\varepsilon}(\D x) = \int_{\R^d} u_{t}(x)\cdot w_{t}(x) \rho_{t}^{v}(\D x) ,
\end{equ}
see for example \cite[Proposition 3.4]{ambrosio2012user}.
Hence, it remains to find an integrable dominating function in time.
Note that as $u_\bullet, w_\bullet \in L^2_T(C^1_{\textup{Lip}}(\R^d))$ there are functions $a, b \in L^2_T$ such that
\begin{equs}
    \lVert u_{t}(x) \rVert & \leq a_t (1+\lVert x \rVert) \quad\text{and }
    \lVert w_{t}(x) \rVert \leq b_t (1+\lVert x \rVert)
\end{equs}
for all $t\in[0,1]$ and $x\in\R^d$.
Hence, we have
\begin{equs}
    \int_{0}^{1}\int_{\R^d}\lvert u_{t}(x)\cdot w_{t}(x)\rvert \rho_{t}^{v,\varepsilon}(\D x)\dif t
    & \leq \int_{0}^{1}a_t b_t\int_{\R^d}\left(1+\lVert x\rVert\right)^{2}\rho_{t}^{v,\varepsilon}(\D x) \dif t
    \\
    & \leq 2\int_{0}^{1}a_t b_t\int_{\R^d}\left(1+\lVert x\rVert^2\right)\rho_{t}^{v,\varepsilon}(\D x) \dif t
    \\
    & = 2\int_{0}^{1}a_t b_t\left(1+m_2(\rho_t^{v,\varepsilon})\right)\dif t.
\end{equs}
Hence, it remains to bound the second moment of $\rho_{t}^{v,\varepsilon}$.
Note that by Lemma \ref{lem:second-moment-via-W2}
\begin{equ}
    m_2(\rho_{t}^{v,\varepsilon}) \le m_2(\rho_{t}^{v}) + 2\sqrt{m_2(\rho_{t}^{v})} W_2(\rho_{t}^{v,\varepsilon}, \rho_{t}^{v}) + W_2^2(\rho_{t}^{v,\varepsilon}, \rho_{t}^{v}).
\end{equ}
Note that by Lemma \ref{lem:wasserstein-convergence-marginals} we have
    $W_2^2(\rho_{t}^{v,\varepsilon}, \rho_{t}^{v}) \le c \varepsilon$
and by Lemma \ref{lem:evolution-W2} the second moment $m_2(\rho_t^v)$ can be bounded uniformly in $t\in[0,1]$.
Hence, we obtain the existence of a constant $C>0$ such that
\begin{equ}
    \int_{0}^{1}\int_{\R^d}\lvert u_{t}(x)\cdot w_{t}(x)\rvert \rho_{t}^{v,\varepsilon}(\D x) \dif t
    \leq C\int_{0}^{1}a_t b_t \dif t<+\infty.
\end{equ}
\end{proof}

\subsection{Large deviation theory}
We have shown that the advective Fisher--Rao metric can be obtained as the rescaled zero-noise limit of the Fisher--Rao metric on path measures.
The concentration of the path measures in the zero-noise limit admits a large deviation principle with rate functional given by the Freidlin--Wentzell functional \cite{freidlin1998random}.
The following result shows that the advective Fisher--Rao metric can be obtained by evaluating the second variation of the Freidlin--Wentzell rate functional, which is given by
\begin{equ}
    I_v(x_\bullet) = \frac{1}{2} \int_0^1 \|\dot{x}_t - v_t(x_t)\|^2 \dif t
\end{equ}
for a path $x_\bullet$ and a fixed velocity field $v_\bullet$.

\begin{theorem}[Connection to Freidlin--Wentzell functional]\label{thm:Freidlin--Wentzell}
Consider a velocity field $v_{\bullet}\in L^{2}_T C_{\textup{b}}^{2}$ and $u_{\bullet}\in L^{2}_T C_{\textup{Lip}}^{1}$ and assume that $\rho_0\in\mathcal{P}_2(\R^d)$.
Then it holds that
\begin{equation}\label{eq:fw_identity}
    g^{\AFR}_{v_\bullet}(u_\bullet,u_\bullet)
    = \int_{\R^d} \delta^2 I_v(\varphi^{v}_\bullet(x_0))\left[\psi_\bullet^u(x_0), \psi_\bullet^u(x_0)\right] \rho_0(\D x_0),
\end{equation}
where $\delta^2 I_v$ denotes the second variation of $I_v$ and
\begin{equ}
    \psi_t^u(x_0) = \ddh \varphi^{v+hu}_t(x_0),
\end{equ}
where $\varphi^{v+hu}_t(x_0)$ denotes the flow generated by $v_{\bullet} + hu_{\bullet}$ starting at $x_0$.
\end{theorem}
We begin by computing the first and second variation of the rate functional.
\begin{lemma}[First and second variation]\label{lem:variations-freidlin-wentzell}
Consider a velocity field $v_\bullet\in L^2_TC^1_{\textup{Lip}}$.
Then the Freidlin--Wentzell functional $I_v\colon H^1_T(\R^d)\to\R$
is G\^{a}teaux differentiable and the first variation is given by
\begin{equ}
    \delta I_v(x_\bullet)[y_\bullet] = \int_0^1 (\dot{x}_t - v_t(x_t))\cdot (\dot{y}_t - Dv_t(x_t) y_t) \dif t
\end{equ}
for all $x_\bullet, y_\bullet\in H^1_T(\R^d)$.
If further $v_\bullet\in L^2_T C^2_{\textup{b}}$, then $I_v$ is twice G\^{a}teaux differentiable and the second variation is given by
\begin{equs}
    \delta^2 I_v(x_\bullet)[y_\bullet, z_\bullet] & = \int_0^1 (\dot{y}_t - Dv_t(x_t) y_t) \cdot (\dot{z}_t - Dv_t(x_t) z_t) \dif t
    \\ & \quad - \int_0^1 (\dot{x}_t - v_t(x_t))\cdot D^2 v_t(x_t)[y_t, z_t] \dif t
\end{equs}
for all $x_\bullet, y_\bullet, z_\bullet\in H^1_T(\R^d)$.
\end{lemma}
\begin{proof}
To verify the existence of the first variation, it suffices to show the existence of the difference quotients
\begin{equ}
    \frac{I_v(x_\bullet+hy_\bullet) - I_v(x_\bullet)}{h} =
    \int_0^1 \Delta_t^h \dif t,
\end{equ}
where
\begin{equs}
    \Delta^h_t
    & = \frac{\lVert \dot{x}_t + h\dot{y}_t - v_t(x_t+hy_t) \rVert^2 - \lVert \dot{x}_t - v_t(x_t) \rVert^2}{2h}
    \\ & = \frac{h}{2} \lVert \dot{y}_t\rVert^2 + (\dot{x}_t - v_t(x_t+h y_t))\cdot \dot{y}_t \\ & \qquad + \frac1{2h} \left( \lVert \dot{x}_t - v_t(x_t+hy_t) \rVert^2 - \lVert \dot{x}_t - v_t(x_t) \rVert^2 \right).
\end{equs}
Pointwise, the limit is given by
\begin{equ}
    \lim_{h\to0} \Delta_t^h = (\dot{x}_t - v_t(x_t))\cdot(\dot{y}_t - Dv_t(x_t)y_t).
\end{equ}
To construct a dominating function, we use that $H^1$-paths are continuous, which in particular gives
\begin{equ}
    M_x\coloneqq \sup_{t \in [0,1]} \|x_t\|, M_y\coloneqq \sup_{t \in [0,1]} \|y_t\| < \infty
\end{equ}
and we denote the spatial Lipschitz constant of $v_t$ by $\beta_t$.
For $\lvert h \rvert \le1$ we can estimate
\begin{equs}
    |\Delta^h_t| \leq \|\dot{x}_t\|^2 + 2\|\dot{y}_t\|^2 + 2\|v_t(0)\|^2 + 2(M_x^2 + M_y^2)\beta_t^2,
\end{equs}
which is integrable.
Hence, by dominated convergence, we obtain
\begin{equ}
    \delta I_v(x_\bullet)[y_\bullet] = \int_0^1 (\dot{x}_t - v_t(x_t))\cdot (\dot{y}_t - Dv_t(x_t) y_t) \dif t.
\end{equ}
For the existence of the second variation, we assume $v_\bullet\in L^2_T C^2_{\textup{b}}$.
We need to guarantee the existence of the limit of the difference quotients
\begin{equ}
    \frac{\delta I_v(x_\bullet + hy_\bullet)[z_\bullet] - \delta I_v(x_\bullet)[z_\bullet]}{h}
    =
    \int_0^1 \Gamma_t^h \dif t,
\end{equ}
where
\begin{equs}
    \Gamma_t^h & = \left(\dot{y}_t - \frac{v_t(x_t+hy_t)-v_t(x_t)}{h}\right)\cdot(\dot{z}_t-Dv_t(x_t+hy_t)z_t)
    \\
    &\qquad - (\dot{x}_t-v_t(x_t))\cdot\left(\frac{Dv_t(x_t+hy_t)-Dv_t(x_t)}{h}\right)z_t.
\end{equs}
Pointwise, the limit is given by
\begin{equ}
    \lim_{h\to0} \Gamma_t^h = (\dot{y}_t-Dv_t(x_t)y_t)\cdot(\dot{z}_t-Dv_t(x_t)z_t)-(\dot{x}_t-v_t(x_t))\cdot D^2v_t(x_t)[y_t,z_t].
\end{equ}
We set $\gamma_t \coloneqq \lVert D^2 v_t \rVert_{\infty}$.
By multiple applications of the mean value theorem applied to $v_t$ and $Dv_t$, together with Young's inequality, each term in $\Gamma_t^h$ can be bounded by quadratic expressions in $\dot{x}_t$, $\dot{y}_t$, $\dot{z}_t$, $v_t(0)$, $\beta_t$, and $\gamma_t$.
Hence there exists a constant $C=C(M_x,M_y,M_z)>0$, independent of $t$ and $h$, such that for $|h|\le1$,
\begin{equ}
    |\Gamma_t^h| \leq C\left(\|\dot{x}_t\|^2+\|\dot{y}_t\|^2+\|\dot{z}_t\|^2+\|v_t(0)\|^2+\beta_t^2+\gamma_t^2\right).
\end{equ}
Since $\beta_\bullet,\gamma_\bullet\in L^2(0,1)$, the right-hand side is integrable, and hence by dominated convergence, we obtain
\begin{equs}
    \delta^2 I_v(x_\bullet)[y_\bullet,z_\bullet] & = \int_0^1 (\dot{y}_t-Dv_t(x_t)y_t)\cdot(\dot{z}_t-Dv_t(x_t)z_t) \dif t
    \\ & \quad - \int_0^1 (\dot{x}_t-v_t(x_t))\cdot D^2v_t(x_t)[y_t,z_t] \dif t.
\end{equs}
\end{proof}
\begin{proof}[Proof of \Cref{thm:Freidlin--Wentzell}]
Recall that we want to show
\begin{equ}
    \int_{\R^d}\delta^2 I_v(x_\bullet)[\psi_\bullet^u(x_0), \psi_\bullet^u(x_0)] \rho_0(\D x_0)
     =  g^{\AFR}_{v_\bullet}(u_\bullet, u_\bullet).
\end{equ}
First, note that we can apply Lemma \ref{lem:variations-freidlin-wentzell} as $v_\bullet\in L^2_T C^2_{\textup{b}}$ and $\psi_\bullet(x)\in H^1_T(\R^d)$.
With $x_\bullet = \varphi_\bullet^v(x_0)$, we have $\dot{x}_t = v_t(x_t)$ for almost every $t \in [0,1]$ and hence
\begin{equ}
    \delta^2 I_v(x_\bullet)[\psi_\bullet^u(x_0), \psi_\bullet^u(x_0)] = \int_0^1 \|\dot{y}_t - Dv_t(x_t) y_t\|^2 \dif t.
\end{equ}
By \Cref{app:lem:sensitivity} we obtain for $y_t = \psi_t^u(x_0)$ that $\dot{y}_t - Dv_t(x_t) y_t = u_t(x_t)$.
As $v_\bullet\in L^2_TC^2_\textup{b}$ this ensures $y_\bullet\in H^1_T(\R^d)$
and hence we can use the expression of the second variation and obtain
\begin{equ}
    \delta^2 I_v(x_\bullet)[\psi_\bullet^u(x_0), \psi_\bullet^u(x_0)]
    =  \int_0^1 \|u_t(x_t)\|^2 \dif t
    =  \int_0^1 \|u_t(\varphi_t^v(x_0))\|^2 \dif t.
\end{equ}
Taking the expectation with respect to the initial condition $x_0$, we obtain
\begin{equs}
    \int_{\R^d}\delta^2 I_v(x_\bullet)[\psi_\bullet^u(x_0), \psi_\bullet^u(x_0)] \rho_0(\D x_0)
    & =  \int_{\R^d}\int_0^1 \|u_t(\varphi^v_t(x_0))\|^2 \dif t \rho_0(\D x)
    \\ & =  \int_0^1 \int_{\R^d}\|u_t(\varphi^v_t(x_0))\|^2 \rho_0(\D x)  \dif t
    \\ & =  \int_0^1 \int_{\R^d}\|u_t(x)\|^2 \rho_t(\D x)  \dif t
    \\ & =  g^{\AFR}_{v_\bullet}(u_\bullet, u_\bullet).
\end{equs}
\end{proof}
This shows that the advective Fisher--Rao metric is the expected value of the second variation of the Freidlin--Wentzell rate functional.
Note, however, that the rate functional includes the base velocity field $v_\bullet$, which is not differentiated in the second variation.
In the introduction, we stated the theorem as
\begin{equ}
    g^{\AFR}_{v_\bullet}(u_\bullet,w_\bullet) = \int_{\R^d} \left.\frac{\D^2}{\D h^2}\right|_{h=0}I_{v}(\varphi^{v+hu}(x_0)) \rho_0(\D x_0) .
\end{equ}
To see that this is equivalent to the statement of \Cref{thm:Freidlin--Wentzell}, recall that $I_v$ is minimized at $\varphi_\bullet^v(x_0)$ with value $0$,  and hence we have $\delta I_v(\varphi^{v}_\bullet(x_0)) = 0$.
Now we obtain
\begin{equs}
    \left.\frac{\D^2}{\D h^2}\right|_{h=0}I_{v}(\varphi^{v+hu}(x_0))
    & = \delta^2 I_v(\varphi^{v}_\bullet(x_0))[\psi_\bullet^u(x_0), \psi_\bullet^u(x_0)].
\end{equs}
\subsection{Optimal transport}
The advective Fisher--Rao metric operates on velocity fields and is tightly connected to the continuity equation, which describes the transport of mass along a flow.
As such, it is a natural question how it is connected to the geometry of optimal transport.
In this subsection, we relate both the least-squares objective on velocities and the advective Fisher--Rao metric to the action functional appearing in the dynamic formulation of optimal transport.
Recall the Benamou--Brenier formulation of the Wasserstein distance
\begin{equ}
    W_2^2(\mu,\nu) = 2\inf_{p_\bullet, j_\bullet} \left\{\mathcal{A}(p_\bullet, j_\bullet) : p_0 = \mu, p_1 = \nu, \partial_t p_t + \nabla\cdot j_t = 0\right\},
\end{equ}
where the action functional in momentum or current form is given by
\begin{equ}
    \mathcal{A}(p_\bullet, j_\bullet) =
        \frac12\int_0^1 \int_{\R^d} \frac{\lVert j_t(x) \rVert_2^2}{p_t(x)} \dif x \dif t.
\end{equ}
In the context of the Benamou--Brenier formulation, the current is given by $j_t = v_t p_t$.
It is known that the action functional is convex, and so is its domain
\begin{equ}
    \operatorname{dom}(\mathcal A) = \left\{(p_\bullet, j_\bullet)  : p_\bullet \in \textup{AC}_T(\mathcal{P}_2(\R^d)), j_\bullet \textup{ measurable}, \mathcal{A}(p_\bullet, j_\bullet)<+\infty\right\}.
\end{equ}

\begin{theorem}\label{thm:hessian-action}
Assume that the initial distribution $\rho_0\in\mathcal{P}_2(\R^d)$ has a positive density $p_0$ with respect to the Lebesgue measure.
For any velocity field $v_\bullet\in L^2_T C^1_{\textup{Lip}}$ and $u_\bullet \in C_{\textup{c}}^1([0,1]\times\R^d; \R^d)$, it holds that
\begin{equ}\label{eq:difference-quotient}
    g^{\AFR}_{v_\bullet}(u_\bullet, u_\bullet) = 2\lim_{h\to0}\frac{\mathcal{A}(p_\bullet^{h}, j_\bullet^{h}) - \mathcal{A}(p_\bullet, j_\bullet) - \delta^+\mathcal{A}(p_\bullet, j_\bullet)[p_\bullet^{h}-p_\bullet, j_\bullet^{h} -j_\bullet]}{h^2},
\end{equ}
where $p_\bullet = p^{v}_\bullet$, $p^h_\bullet = p^{v+h u}_\bullet$, $j_\bullet = v_\bullet p^v_\bullet$, and $j^h_\bullet = (v_\bullet+h u_\bullet) p^{v+h u}_\bullet$ denote densities of the solution of the continuity equation and the corresponding currents. 
\end{theorem}
In the
proof, we will use the following lemma, which is of independent interest as it states that the least-squares objective on velocity fields is the Bregman divergence induced by the action functional.
\begin{lemma}\label{lem:first-variation-action}
Consider $(p_\bullet, j_\bullet), (p_\bullet', j_\bullet')\in\operatorname{dom}(\mathcal A)$ and assume that $p_\bullet>0$ and $p'_\bullet>0$, and that
\begin{equ}\label{eq:integrability-assumption}
    \int_0^1 \int_{\R^d} \lVert v_t(x) \rVert_2^2 p'_t(x) \dif x \dif t \quad  \text{and }
    \int_0^1 \int_{\R^d} \lVert v'_t(x) \rVert_2^2 p_t(x) \dif x \dif t < +\infty,
\end{equ}
where $v_\bullet=\frac{j_\bullet}{p_\bullet}$ and $v_\bullet'=\frac{j_\bullet'}{p_\bullet'}$.
Then, the right-sided first variation
$\delta^+\mathcal A(p_\bullet, j_\bullet)[p'_\bullet-p_\bullet, j'_\bullet-j_\bullet]$ exists
and is given by
\begin{equs}
    \delta^+ \mathcal{A}(p_\bullet, j_\bullet)[p'_\bullet-p_\bullet, j'_\bullet-j_\bullet]
    & = \int_0^1 \int_{\R^d} \left( v'_t\cdot v_t p'_t
    -
    \lVert v_t\rVert_2^2 (p_t+p'_t) \right)\dif x \dif t,
\end{equs}
 Further, the Bregman divergence induced by the action functional is given by
\begin{equ}
    D_{\mathcal{A}}((p_\bullet, j_\bullet), (p'_\bullet, j'_\bullet)) = \frac12 \int_0^1 \int_{\R^d} \lVert v_t(x) - v'_t(x) \rVert_2^2 p_t(x) \dif x \dif t.
\end{equ}
\end{lemma}
\begin{proof}
First, we rewrite the action functional as
\begin{equ}
    \mathcal{A}(p_\bullet, j_\bullet) = \int_0^1 \int_{\R^d} f(p_t(x), j_t(x)) \dif x \dif t,
\end{equ}
where
\begin{equ}
    f\colon \R_{>0}\times \R^d \to \R, \quad f(p,j) = \frac{\lVert j \rVert_2^2}{2p}
\end{equ}
is a convex function.
Then, we need to establish the existence of the limit of the difference quotients
\begin{equ}
    \frac{\mathcal A(p_\bullet+h(p'_\bullet-p_\bullet), j_\bullet + h(j'_\bullet - j_\bullet)) - \mathcal A(p_\bullet, j_\bullet)}{h} = \int_0^1\int_{\R^d} \Delta^h_t(x) \dif x \dif t,
\end{equ}
where
\begin{equ}
    \Delta^h_t = \frac{f(p_t+h(p'_t-p_t),j_t+h(j'_t-j_t))  - f(p_t, j_t)}{h}.
\end{equ}
To compute the pointwise limit of $\Delta_t(x)$ we use that for $p>0$ the first derivative of $f$ is given by
\begin{equs}
    \nabla f(p,j)
    &
    = \begin{pmatrix} -\dfrac{\lVert j \rVert_2^2}{2p^2} \\[10pt] \dfrac{j}{p} \end{pmatrix}
    = \begin{pmatrix} -\dfrac{\lVert v \rVert_2^2}{2} \\[10pt] v \end{pmatrix}
    ,
\end{equs}
where $v = \frac{j}{p}$.
Now the chain rule implies
\begin{equs}
    \lim_{h\to0} \Delta_t^h(x)
    & = - (p'_t(x)-p_t(x)) \cdot\frac{\lVert v_t(x) \rVert^2}{2} + v_t(x)\cdot(p'_t(x)v'_t(x)-p_t(x)v_t(x)),
    \\ & = - (p'_t(x)+p_t(x)) \cdot\frac{\lVert v_t(x) \rVert^2}{2} + v_t(x)\cdot v'_t(x) p'_t(x),
    \\ & \eqqcolon \Delta^0_t(x),
\end{equs}
which is an integrable function over time and space by  assumption~\eqref{eq:integrability-assumption}.
To use the dominated convergence theorem, we first note that by the convexity of $f$, the difference quotient $\Delta^h$ is pointwise monotonically decreasing and hence
\begin{equ}
    \Delta_t^h(x) \ge \Delta^0_t(x) \quad \text{for all } t\in[0,1], x\in\R^d.
\end{equ}
Further, for all $t\in[0,1], x\in\R^d$ we have
\begin{equ}
    \Delta^h_t(x) \le\Delta^1_t(x) = f(p'_t(x),j'_t(x)) - f(p_t(x), j_t(x)),
\end{equ}
which is integrable over $[0,1]\times\R^d$ as $(p_\bullet, j_\bullet), (p_\bullet', j_\bullet')\in\operatorname{dom}(\mathcal A)$.
As $\Delta^h$ is pointwise bounded from above and below by integrable functions
, we can apply the dominated convergence theorem and obtain
\begin{equs}
    \delta^+\mathcal{A}(p_\bullet, j_\bullet)[p_\bullet'-p_\bullet, j_\bullet'-j_\bullet]
    & = \lim_{h\searrow0} \int_0^1\int_{\R^d} \Delta_t^h(x) \dif x \dif t
    \\ & = \int_0^1\int_{\R^d} \lim_{h\searrow0} \Delta_t^h(x) \dif x \dif t
    \\ & = \int_0^1 \int_{\R^d} v'\cdot v p'\dif x \dif t - \frac12 \int_0^1 \int_{\R^d} \lVert v\rVert_2^2 (p+p')\dif x \dif t.
\end{equs}
Finally, we compute the Bregman divergence
\begin{equs}
    D_{\mathcal{A}}((p_\bullet, j_\bullet), (p_\bullet', j_\bullet'))
    & = \mathcal{A}(p_\bullet, j_\bullet) - \mathcal{A}(p_\bullet', j_\bullet') - \delta^+\mathcal{A}(p_\bullet', j_\bullet')[p_\bullet-p_\bullet', j_\bullet-j_\bullet']
    \\ & = \frac12 \int_0^1\int_{\R^d} \lVert v \rVert_2^2 p  \dif x \dif t
    - \frac12 \int_0^1\int_{\R^d} \lVert v' \rVert_2^2 p'  \dif x \dif t
    \\ & \quad - \int_0^1 \int_{\R^d} v\cdot v' p\dif x \dif t + \frac12 \int_0^1 \int_{\R^d} \lVert v'\rVert_2^2 (p'+p)\dif x \dif t
    \\ & = \frac12 \int_0^1\int_{\R^d} \lVert v - v' \rVert_2^2 p \dif x \dif t.
\end{equs}
\end{proof}
\begin{remark}[G\^{a}teaux differentiability of the action]
With the notation $(\xi_\bullet, n_\bullet) = (p_\bullet', j_\bullet')-(p_\bullet, j_\bullet)$, the expression of the right-sided first variation in Lemma \ref{lem:first-variation-action} becomes
\begin{equ}
    \delta^+\mathcal{A}(p_\bullet, j_\bullet)[\xi_\bullet, n_\bullet] = \int_0^1 \int_{\R^d} \left(v_t(x)\cdot n_t(x) - \frac12 \lVert v_t(x) \rVert_2^2 \xi_t(x) \right) \dif x \dif t.
\end{equ}
As this expression is linear in $(\xi_\bullet, n_\bullet)$, this implies that if $(p_\bullet, j_\bullet) - h (\xi_\bullet, n_\bullet)\in\operatorname{dom}(\mathcal A)$ for $h>0$ small, then also the left-sided difference quotients converge to the same limit.
Consequently, this shows the G\^{a}teaux differentiability in the case that the two-sided difference quotients can be formed.
\end{remark}
Further, we require a bound on the Wasserstein-2 distance in terms of the least-squares distance of the velocity fields.
An analogous bound under stronger regularity conditions on the velocity fields can be found in \cite{benton2024error}.

\begin{lemma}\label{lem:W2-stability}
Consider $v_\bullet, w_\bullet\in L^2_TC^1_{\textup{Lip}}$ and assume that the initial condition $\rho_0$ has finite second moments.
Then for all $t\in[0,1]$ it holds that 
\begin{equ}
    W_2(\rho^v_t, \rho^w_t)
    \le e^{\lVert v_\bullet\rVert_{\V}} \, \| v_\bullet - w_\bullet \|_{L^1_T L^2(\rho^w_\bullet)}
    \le e^{\lVert v_\bullet\rVert_{\V}} \, \| v_\bullet - w_\bullet \|_{L^2_T L^2(\rho^w_\bullet)}.
\end{equ}
\end{lemma}
\begin{proof}
Let $\varphi^v$ and $\varphi^w$ be the flow maps associated with the vector fields $v_\bullet$ and $w_\bullet$.
By the definition of the Wasserstein distance, we can bound $W_2$ using this specific coupling
\begin{equ}
    W_2(\rho_t^v, \rho_t^w) \le \lVert \varphi_t^v - \varphi_t^w \rVert_{L^2(\rho_0)}.
\end{equ}
We can express the difference between the two flow maps at time $t$ via
\begin{equ}
    \varphi_t^v(x) - \varphi_t^w(x) = \int_0^t \left( v_s(\varphi_s^v(x)) - w_s(\varphi_s^w(x)) \right) \dif s.
\end{equ}
Taking the $L^2(p_0)$ norm of both sides and applying the triangle inequality, we obtain
\begin{equs}
    \lVert \varphi_t^v - \varphi_t^w \rVert_{L^2(\rho_0)}
     & \le \int_0^t \lVert v_s \circ \varphi_s^v - w_s \circ \varphi_s^w \rVert_{L^2(\rho_0)} \dif s.
    \\ & \le \int_0^t \lVert v_s \circ \varphi_s^v - v_s \circ \varphi_s^w \rVert_{L^2(\rho_0)} \dif s \\ & \quad + \int_0^t \lVert v_s \circ \varphi_s^w - w_s \circ \varphi_s^w \rVert_{L^2(\rho_0)} \dif s
    \\ &
    \le \int_0^t \beta_s \lVert \varphi_s^v - \varphi_s^w \rVert_{L^2(\rho_0)} \dif s + \int_0^t \lVert v_s - w_s\rVert_{L^2(p^w_s)} \dif s,
\end{equs}
where $\beta_s$ denotes the spatial Lipschitz constant of $v_s$.
Applying Gr\"onwall's inequality yields
\begin{equs}
    \lVert \varphi_t^v - \varphi_t^w \rVert_{L^2(\rho_0)} &\le
    e^{\int_0^t \beta_r \dif r} \int_0^t \lVert v_s - w_s\rVert_{L^2(p^w_s)} \dif s
     \\ & \le
     e^{\int_0^t \beta_r \dif r} \sqrt{t} \left(\int_0^t \lVert v_s - w_s\rVert_{L^2(p^w_s)}^2 \dif s\right)^{\frac12}
     \\ & \le
     e^{\int_0^t \beta_r \dif r} \sqrt{t} \, \| v_\bullet - w_\bullet \|_{L^2_T L^2(p^w_\bullet)}.
\end{equs}
\end{proof}

\begin{proof}[Proof of \Cref{thm:hessian-action}]
First, we notice that for $v_\bullet\in L^2_TC^1_{\textup{Lip}}$, we have $(p^v_\bullet, v_\bullet p^v_\bullet)\in \operatorname{dom}(\mathcal A)$.
Indeed, $v_\bullet\in L^2_TC^1_{\textup{Lip}}$ implies that $p_\bullet^v$ has uniformly in time bounded second moments, see, for example, Lemma \ref{lem:evolution-W2}, and hence
\begin{equs}
    \mathcal{A}(p^v_\bullet, v_\bullet p^v_\bullet) & = \frac12 \int_0^1 \int_{\R^d} \lVert v_t(x) \rVert_2^2 p_t^v(x) \dif x \dif t
     < +\infty.
\end{equs}
Consider now $v^h_\bullet=v_\bullet+h u_\bullet$, $p^h_\bullet=p_\bullet^{v+hu}$ and $j^h_\bullet=j^{v+hu}_\bullet$, then we obtain
\begin{equs}
    \frac{\mathcal{A}(p_\bullet^h,j_\bullet^h)
    -\mathcal{A}(p_\bullet,j_\bullet)
    -\delta^+\mathcal{A}(p_\bullet,j_\bullet)[\xi_\bullet^h, n_\bullet^h]}{h^2}
    & = h^{-2}D_{\mathcal{A}}((p_\bullet^h,j_\bullet^h), (p_\bullet,j_\bullet))
    \\ & = \frac{1}{2h^2} \int_0^1 \int_{\R^d} \lVert v^h_t(x) - v_t(x) \rVert_2^2 p_t^h(x) \dif x \dif t
    \\ & = \frac{1}{2} \int_0^1 \int_{\R^d} \lVert u_t(x) \rVert_2^2 p_t^h(x) \dif x \dif t.
\end{equs}
Note that $p^h_t\to p_t$ as $h\to 0$ in $\mathcal{P}_2(\R^d)$ at a rate of $O(h)$ uniformly in time, see for example Lemma \ref{lem:W2-stability}.
Together with the linear growth of $u_t$, this yields that
\begin{equ}
    \lim_{h\to0}\frac12\int_0^1 \int_{\R^d} \lVert u_t(x) \rVert_2^2 p_t^h(x) \dif x \dif t =
    \frac12\int_0^1 \int_{\R^d} \lVert u_t(x) \rVert_2^2 p_t(x) \dif x \dif t.
\end{equ}
\end{proof}

\section{The Advective Fisher--Rao Metric on Paths of Probability Measures}\label{sec:convex-geometry-BB}
So far, we have defined the advective Fisher--Rao metric on the space of velocity fields.
However, it is the goal to solve the control problem~\eqref{eq:matching-paths} on the paths of probability densities.
Note that there are multiple velocity fields that can generate the same path of probability densities via the continuity equation.
Hence, it is not directly possible to pushforward the geometry onto the space of paths of probability densities.
In this section, we will define a version of the advective Fisher--Rao metric on the space of paths of probability densities that is compatible with the advective Fisher--Rao metric on the space of gradient velocity fields.
Additionally, we show that it preserves the energy compatibility property and that it arises as the Hessian metric of the contracted Benamou--Brenier action functional.
We will construct a Riemannian metric on the space
\begin{equ}
    \textup{AC}_T^2(\mathcal{P}_2(\R^d)) = \left\{ \rho_\bullet\in\textup{AC}_T(\mathcal{P}_2(\R^d)) : \int_0^1 \lvert \dot{\rho}_t \rvert_{W_2}^2 \dif t <+\infty \right\},
\end{equ}
where $\lvert \dot{\rho}_t \rvert_{W_2}= \lim_{h\to0} W_2(\rho_{t+h}, \rho_t)h^{-1}$ is the metric derivative, see \cite{ambrosio2005gradient}.
To construct a tangent space for this space of paths, recall the definition of the Wasserstein--Otto tangent space
\begin{equs}
    T_{\rho}\mathcal{P}_{2}(\R^{d})&\coloneqq \{ - \nabla \cdot (\rho v) : v\in L^2(\rho) \} \subseteq \mathcal{D}'(\R^d). 
\end{equs}
which dates back to \cite{otto2001geometry}.
Note that $v\mapsto -\nabla\cdot(\rho v)$ is not injective. 
To remove this redundancy, it is common to consider 
\begin{equ}
    \mathcal{T}_\rho\mathcal{P}_2(\R^d)\coloneqq \overline{\left\{\nabla \phi : \phi\in C_\textup{c}^\infty(\R^d)\right\} }^{L^2(\rho)}
\end{equ}
and to identify the spaces $\mathcal{T}_\rho\mathcal{P}_2(\R^d) \cong {T}_\rho\mathcal{P}_2(\R^d)$ according to $v\mapsto -\nabla\cdot(\rho v) $ and to refer to $\mathcal{T}_\rho\mathcal{P}_2(\R^d)$ as the tangent space.
However, for us, it is convenient to keep the distinction between those two spaces.
The \emph{Wasserstein--Otto} metric is given by 
\begin{equ}
    g_\rho^{\textup{WO}}(\xi, \zeta) = \int_{\R^d} u\cdot w \rho(\D x), 
\end{equ}
where $\xi = -\nabla\cdot(\rho v)$ and $\zeta = -\nabla\cdot(\rho w)$. 

We define the time-dependent version of the tangent spaces by
\begin{equs}
    T_{\rho_{\bullet}}L^2_T(\mathcal P_2(\mathbb R^d))
    & \coloneqq
    \left\{
    \xi_\bullet
    :
    \begin{gathered}
        \xi_t \in T_{\rho_t} \mathcal{P}_2(\R^d) \text{ for every } t\in[0,1] 
    \end{gathered}
    \right\}, 
\end{equs}
and similarly 
\begin{equ}
    \mathcal{T}_{\rho_{\bullet}}L^2_T(\mathcal P_2(\mathbb R^d))
    \coloneqq
    \left\{
    v_\bullet\in L^2_TL^2(\rho_\bullet)
    :
    \begin{gathered}
        v_t \in \mathcal{T}_{\rho_t} \mathcal{P}_2(\R^d) \textup{ for every } t\in[0,1]
    \end{gathered}
    \right\}.
\end{equ}
The corresponding time-dependent metric is given by 
\begin{equ}
    g_{\rho_{\bullet}}^{\textup{WO}}(\xi_{\bullet},\xi_{\bullet})=\int_{0}^{1}g_{\rho_{t}}^{\textup{WO}}(\xi_{t}, \xi_{t})\dif t. 
\end{equ}

For a curve $\rho_\bullet\in \textup{AC}_T^2\mathcal{P}_2(\R^d)$, it is well known that one can find a velocity field $v_\bullet^\rho\in \mathcal{T}_{\rho_\bullet} L^2_T(\mathcal{P}_2(\R^d))$ such that
\begin{equ}
    \partial_t \rho_t + \nabla \cdot(\rho_t v_t^\rho) = 0,
\end{equ}
see \cite{ambrosio2005gradient}, which we refer to as the velocity of $\rho_\bullet$.
Now, we can proceed with our construction of a metric on the space $\textup{AC}_T^2(\mathcal{P}_2(\R^d))$, which respects the advective nature of the continuity equation.

\begin{definition}[Advective Fisher--Rao metric on paths of measures]\label{def:AFR-measures}
Consider a path $\rho_\bullet \in \textup{AC}_T^2(\mathcal{P}_{2}(\R^{d}))$ and its velocity  $v_\bullet  \in T_{\rho_\bullet} L_T^2(\mathcal{P}_{2}(\R^{d}))$.
Then we define the tangent space of $\textup{AC}_T^2(\mathcal{P}_2(\R^d))$ as
\begin{equs}
    T_{\rho_{\bullet}}\mathrm{AC}^2_T(\mathcal P_2(\mathbb R^d))
    \coloneqq
    \left\{
    \xi_\bullet \in 
    T_{\rho_\bullet}L^2_T(\mathcal{P}_2(\R^d))
    :
    \begin{gathered} 
    \xi_t v_t \in \mathcal{D}'(\R^d) \textup{ for all } t\in[0,1]\\ 
    D_t \xi_\bullet \in T_{\rho_\bullet}L^2_T(\mathcal{P}_2(\R^d))
    \end{gathered}
    \right\},
\end{equs}
where we refer to 
\begin{equ}\label{eq:continuity-operator}
    D_t \xi_t \coloneqq \partial_t\xi_t + \nabla\cdot(\xi_t v_t)
\end{equ}
as the \emph{transport derivative} of $\xi_\bullet$ along $\rho_\bullet$, which is to be understood in the distributional sense. 
For two tangent vectors $\xi_\bullet, \zeta_\bullet\in T_{\rho_\bullet} \textup{AC}_T^2(\mathcal{P}_2(\R^d))$, we refer to
\begin{equ}
    g_{\rho_{\bullet}}^{\AFR}(\xi_{\bullet},\xi_{\bullet})=\int_{0}^{1}g_{\rho_{t}}^{\textup{WO}}(D_{t}\xi_{t}, D_{t}\xi_{t})\dif t 
\end{equ}
as the \emph{advective Fisher--Rao metric}. 
\end{definition}

It can be shown that the transport derivative $D_t$ is the covariant derivative with respect to a flat connection, which is not the Levi-Civita connection with respect to the Wasserstein--Otto metric~\cite{ay2025information}.

\begin{remark}[Explicit expression of the advective Fisher--Rao metric]\label{rem:explicit-expression-AFR}
For tangent vectors $\xi_\bullet, \zeta_\bullet\in T_{\rho_\bullet} \textup{AC}_T^2(\mathcal{P}_2(\R^d))$, the advective Fisher--Rao metric is given by 
\begin{equ}
    g_{\rho_{\bullet}}^{\textup{AFR}}(\xi_{\bullet},\zeta_{\bullet})\coloneqq\int_{0}^{1}\int_{\R^d}u^{\xi}_{t}(x)\cdot w^{\zeta}_{t}(x)\rho_{t}(\D x)\dif t,
\end{equ}
where
\begin{equs}\label{eq:phi-psi}
    \partial_{t}\xi_{t}+\nabla\cdot(\xi_{t} v_t) & =-\nabla\cdot(\rho_{t}u^{\xi}_{t}) \quad \text{and}\\
    \partial_{t}\zeta_{t}+\nabla\cdot(\zeta_{t}v_t) & =-\nabla\cdot(\rho_{t}w^{\zeta}_{t}).
\end{equs}    
\end{remark}

The following result shows the compatibility of the Fisher--Rao metric on the space of paths of densities with the Fisher--Rao metric on the space of gradient velocity fields.
In particular, it shows that the solution operator of the continuity equation is a Riemannian isometry between the two spaces.
Further, it gives a class of non-trivial examples in which the advective Fisher--Rao metric on the space of paths of densities is well-defined.
\begin{lemma}[Compatibility]\label{thm:FR-isometrc}
Consider gradient velocity fields 
\begin{equ}
    v_\bullet=\nabla\phi_\bullet \in L^2_TC^1_{\textup{Lip}}, \quad
    u_\bullet=\nabla\psi_\bullet\in L^2_T C^1_{\textup{Lip}}(\R^d), \quad
    w_\bullet=\nabla\psi'_\bullet\in L^2_T C^1_{\textup{Lip}}(\R^d)
\end{equ}
and denote the solutions of the continuity equation with velocity fields $v_\bullet+h u_\bullet$ and $v_\bullet+h w_\bullet$ by $\rho^{v+hu}_\bullet, \rho^{v+hw}_\bullet\in \textup{AC}_T^2(\mathcal{P}_2(\R^d))$, respectively,
and assume that $\rho_0\in \mathcal{P}_2(\R^d)$.
Then
\begin{equ}
    \xi_\bullet\coloneqq \ddh \rho^{v+hu}_\bullet \quad \text{and }\zeta_\bullet\coloneqq \ddh \rho^{v+hw}_\bullet
\end{equ}
are elements in $T_{\rho_\bullet} \textup{AC}_T^2(\mathcal{P}_2(\R^d))$ and it holds that
\begin{equ}\label{eq:compatibility}
    g^{\AFR}_{\rho^v_\bullet}(\xi_\bullet, \zeta_\bullet) = g^{\AFR}_{v_\bullet}(u_\bullet, w_\bullet).
\end{equ}
\end{lemma}
\begin{proof}
By Lemma \ref{thm:derivative-solution}, $\xi_t$ and $\zeta_t$ exist as weak-$\ast$ limits in $(C^1_{\textup{Lip}})^\ast$ for every $t\in[0,1]$, and satisfy
\begin{equs}\label{eq:proof:compatibility}
    \partial_{t}\xi_{t}+\nabla\cdot(\xi^u_{t}\nabla\phi_{t}) & =-\nabla\cdot(\rho_{t}u_{t}) \quad \text{and}\\
    \partial_{t}\zeta_{t}+\nabla\cdot(\zeta^w_{t}\nabla\phi_{t}) & =-\nabla\cdot(\rho_{t}w_{t}).
\end{equs}
Further, as the second moments of $\rho^v_\bullet$ are uniformly bounded in time by Lemma \ref{lem:evolution-W2} and $u_\bullet, w_\bullet\in L^2_TC^1_{\textup{Lip}}(\R^d)$, we obtain $u_\bullet, w_\bullet\in L^2_TL^2(\rho_\bullet^v)$.
To show that $\xi_\bullet\in T_{\rho_\bullet} \textup{AC}_T^2(\mathcal{P}_2(\R^d))$, it remains to verify $\xi_\bullet\in T_{\rho_\bullet}L^2_T\mathcal{P}_2(\R^d)$.
For a fixed time $t\in[0,1]$, Lemma \ref{thm:derivative-solution} gives
\begin{equ}
    \langle \xi_t, \phi \rangle = \int_{\R^d} \nabla\phi(x) \cdot \psi^{u}_t((\varphi^v_t)^{-1}(x)) \rho_t(\D x)
\end{equ}
for all $\phi\in C^1_{\textup{Lip}}(\R^d)$, where for any $x\in\R^d$ $\psi^u_\bullet(x)$ is the unique solution to
\begin{equ}
    \partial_t \psi^u_t(x) = Dv_t(\varphi^v_t(x)) \psi^u_t(x) + u_t(\varphi^v_t(x)), \quad \psi_0(x)=0 .
\end{equ}
Hence, in distribution it holds that
\begin{equ}
    \xi_t = -\nabla\cdot(\rho_t^v v^\xi_t),
\end{equ}
where $v^\xi_t(x) = \psi^{u}_t((\varphi^v_t)^{-1}(x))$. 
To show $\xi_t\in T_{\rho_t}(\mathcal{P}_2(\R^d))$, we first show $v^\xi_\bullet \in L^2_TL^2(\rho_\bullet^v)$. 
From Lemma \ref{app:lem:flow-grow} we obtain
\begin{equs}
    \lVert \psi_t(x) \rVert & \le e^{\lVert v_\bullet \rVert_{\V}} \int_0^t \lVert u_s(\varphi^v_s(x)) \rVert \dif s
    \\ & \le e^{\lVert v_\bullet \rVert_{\V}} \int_0^t \lVert u_s \rVert_{C^1_{\textup{Lip}}} (1+\lVert \varphi^v_s(x) \rVert) \dif s
    \\ & \le \lVert v_\bullet \rVert_{\V} e^{2\lVert v_\bullet \rVert_{\V}} \int_0^t  \lVert u_s \rVert_{C^1_{\textup{Lip}}}(1+  \lVert x \rVert) \dif s ,
    \\ & \le \lVert u_\bullet \rVert_{\V} \lVert v \rVert_{\V} e^{2\lVert v_\bullet \rVert_{\V}} (1+  \lVert x \rVert),
    \\ & \le C (1+  \lVert x \rVert).
\end{equs}
Further, by Lemma \ref{app:lem:flow-grow} $\varphi^v_t$ is Lipschitz continuous for every $t\in[0,1]$ with constant bounded by $e^{\lVert v_\bullet \rvert_{\V}}$. 
As $(\varphi_t^v)^{-1}$ can be interpreted as the flow of $-v_{t-s}$, the map $(\varphi_t^v)^{-1}$ is also Lipschitz continuous with constant $e^{\lVert v_\bullet \rVert_{\V}}$ for every $t\in[0,1]$. 
Hence, $v^\xi_t$ is Lipschitz continuous with a Lipschitz constant bounded uniformly in $t\in[0,1]$. 
As $\rho_0\in\mathcal{P}_2(\R^d)$, Lemma \ref{lem:evolution-W2} guarantees that the second moments of $\rho_\bullet^v$ stay uniformly bounded in $t\in[0,1]$. 
Together, this shows $v_\bullet^\xi \in L^2_T L^2(\rho_\bullet^v)$ and by symmetry also $\zeta_\bullet \in T_{\rho_\bullet^v} \textup{AC}_T^2(\mathcal{P}_2(\R^d))$. 

Having checked that $\xi_\bullet, \zeta_\bullet\in T_{\rho_\bullet^v} \textup{AC}_T^2(\mathcal{P}_2(\R^d))$, it remains to show \eqref{eq:compatibility}. 
Definition \ref{def:AFR-measures} of the advective Fisher--Rao metric on $\textup{AC}_T^2(\mathcal{P}_2(\R^d))$ and \eqref{eq:proof:compatibility} yield 
\begin{equ}
    g_{\rho^v_\bullet}^{\AFR}(\xi_\bullet, \zeta_\bullet) = \int_0^1\int_{\R^d}  u_t(x) \cdot w_t(x) \rho_t^v(\D x)\dif t. 
\end{equ}
On the other hand, Definition \ref{def:AFR-velocities} of the advective Fisher--Rao metric on velocity fields gives 
\begin{equ}
    g_{v_\bullet}^{\AFR}(u_\bullet,w_\bullet) = \int_0^1\int_{\R^d}  u_t(x) \cdot w_t(x) \rho_t^v(\D x)\dif t. 
\end{equ}
\end{proof}
Similar to before, we can relate the advective Fisher--Rao metric to the action functional from the dynamic formulation of optimal transport.
To this end, we write the Benamou--Brenier formulation of the Wasserstein distance as
\begin{equ}
    W_2^2(\mu,\nu) = \inf_{\rho_\bullet \in \textup{AC}_T^2(\mathcal{P}_2(\R^d))} \left\{\mathcal{A}(p_\bullet) : \rho_0 = \mu, \rho_1 = \nu\right\},
\end{equ}
where the action of a curve is given by
\begin{equ}
    \mathcal{A}(\rho_\bullet) = \frac12\inf_{v_\bullet} \left\{\int_0^1 \int_{\R^d} \lVert v_t(x) \rVert_2^2 \rho_t(\D x)  \dif t : \partial_t \rho_t + \nabla\cdot(\rho_{t}v_{t}) = 0\right\}.
\end{equ}
Alternatively, the action functional can be written as
\begin{equ}
    \mathcal{A}(\rho_\bullet) = \frac12 \int_0^1 \int_{\R^d} \lVert v^\rho_t(x) \rVert_2^2 \rho_t(\D x) \dif t,
\end{equ}
where $v_\bullet^\rho\in T_{\rho_0} L^2_T(\mathcal{P}_2(\R^d))$ is the velocity of $\rho_\bullet$ meaning that $\partial_t \rho_t + \nabla\cdot(\rho_{t}v_{t}^\rho) = 0$.
\begin{theorem}\label{thm:convex-geometry-BB}
    Consider two gradient velocity fields $\nabla\phi_\bullet,\nabla\psi_\bullet\in L^2_T C^1_{\textup{Lip}}(\R^d)$ and set $\nabla\phi^h_\bullet \coloneqq \nabla\phi_\bullet+h\nabla\psi_\bullet \in L^2_TC^1_{\textup{Lip}}(\R^d)$ and denote the corresponding solutions of the continuity equation by $\rho_\bullet^h\in\textup{AC}_T(\mathcal{P}_2(\R^d))$. 
    Assume that for $\lvert h \rvert$ small enough, we have $\rho^h_\bullet\ll\rho^0_\bullet$ with $\frac{\D\rho^h_\bullet}{\D\rho^0_\bullet} \in L^\infty_T L^\infty(\rho^0_\bullet)$,
    and that there is a solution $v^h_\bullet\in \mathcal{T}_{\rho_\bullet^0}\textup{AC}_T^2(\mathcal{P}_2(\R^d))$ to 
    \begin{equ}
        \partial_t(\rho^h_t-\rho_t^0)+\nabla\cdot(\rho_t^0 v^h_t)=0.
    \end{equ}
    Then, for $\xi_\bullet \coloneqq \ddh \rho^{h}_\bullet \in T_{\rho_\bullet^0} \textup{AC}_T^2(\mathcal{P}_2(\R^d))$ it holds that 
    \begin{equ}
        g_{\rho_{\bullet}^0}^{\AFR}(\xi_{\bullet}, \xi_{\bullet})  =
        2 \lim_{h \to 0} \frac{\mathcal{A}(\rho^h_{\bullet}) - \mathcal{A}(\rho_{\bullet}^0) - \delta^+ \mathcal{A}(\rho_{\bullet}^0)[\rho^h_{\bullet} - \rho_{\bullet}^0]}{h^2}.
    \end{equ}
\end{theorem}
We use the following auxiliary result.
\begin{lemma}[Dual formulation of the action]
Consider $\rho_\bullet \in \textup{AC}_T^2(\mathcal{P}_2(\R^d))$.
Then, it holds that
\begin{equ}
    \mathcal{A}(\rho_\bullet) = \sup_{\phi_\bullet \in C_{\textup{c}}^\infty((0,1)\times\R^d)} - \int_0^1 \int_{\R^d} \left( \partial_t\phi_t(x) + \frac12 \lVert \nabla \phi_t(x) \rVert_2^2 \right) \rho_t(\D x) \dif t .
\end{equ}
\end{lemma}
\begin{proof}
The action functional is given by
\begin{equ}
    \mathcal A(\rho_\bullet) = \inf \left\{ \mathcal{E}_\rho(v_\bullet): \partial_t \rho_t + \nabla\cdot(\rho_tv_t)=0 \right\},
\end{equ}
where
\begin{equ}
    \mathcal{E}_\rho(v_\bullet) = \frac12 \int_0^1 \int_{\R^d} \lVert v_t(x) \rVert^2 \rho_t(\D x) \dif t.
\end{equ}
For a fixed path $\rho_\bullet\in\textup{AC}_T(\mathcal{P}_2(\R^d))$ we  consider the Lagrangian of the constrained problem
\begin{equs}
    \mathcal{L}_\rho\colon L^2_T L^2(\rho_\bullet)\times C_{\textup{c}}^\infty((0,1)\times\R^d)\to\R,
\end{equs}
where
\begin{equs}
    \mathcal{L}_\rho(v_\bullet, \phi_\bullet)
    & =  \int_0^1 \int_{\R^d} \left(\frac12\lVert v_t \rVert^2_2 - \partial_t \phi_t - v_t\cdot \nabla \phi_t \right) \rho_t(\D x) \dif t .
\end{equs}
Then, it holds that
\begin{equs}
    \mathcal{A}(\rho_\bullet)
    & = \inf_{v_\bullet\in L^2_T L^2(p_\bullet)} \sup_{\phi\in C_{\textup{c}}^\infty((0,1)\times\R^d)} \mathcal{L}_\rho(v_\bullet, \phi_\bullet)
    \\& \ge \sup_{\phi\in C_{\textup{c}}^\infty((0,1)\times\R^d)} \inf_{v_\bullet\in  L^2_T L^2(p_\bullet)}  \mathcal{L}_\rho(v_\bullet, \phi_\bullet).
\end{equs}
Carrying out the inner infimum gives
\begin{equ}
    \mathcal{A}(\rho_\bullet) \ge \sup_{\phi\in C_{\textup{c}}^\infty((0,1)\times\R^d)} - \int_0^1 \int_{\R^d} \left( \partial_t\phi_t + \frac12 \lVert \nabla \phi_t \rVert_2^2 \right) \rho_t(\D x) \dif t .
\end{equ}
Let $v_\bullet^\rho\in T_{\rho_\bullet} L^2_T \mathcal{P}_2(\R^d)$ denote the velocity of $\rho_\bullet$ meaning that
\begin{equ}
    \partial_t \rho_t + \nabla \cdot(\rho_t v_t^\rho) = 0
\end{equ}
holds in the distributional sense.
Next, we seek an approximation $\nabla \phi^\varepsilon_\bullet\to v_\bullet^\rho$ in $L^2_TL^2(\rho_\bullet)$, where $\phi_\bullet^\varepsilon \in C_{\textup{c}}^\infty ((0,1)\times\R^d)$ for $\varepsilon\to0$.
To see that this exists, it suffices to show that for a velocity field $v_\bullet\in T_{\rho_\bullet}L^2_T\mathcal{P}_2(\R^d)$, the orthogonality 
\begin{equ}
    \int_0^1 \int_{\R^d} v_t \cdot \nabla\phi_t \rho_t(\D x) \dif t  =0 \quad \text{for all } \phi_\bullet\in C^{\infty}_{\textup{c}}((0,1)\times\R^d)
\end{equ}
implies that $v_\bullet=0$. 
To see this, we note that this implies 
\begin{equ}
    \int_0^1 \eta(t) \int_{\R^d} v_t \cdot \nabla\psi \rho_t(\D x) \dif t =0 \quad \text{for all } \psi\in C^{\infty}_{\textup{c}}(\R^d), \eta\in  C^{\infty}_{\textup{c}}((0,1)).
\end{equ}
In particular, this implies
\begin{equ}
    \int_{\R^d} v_t \cdot \nabla\psi \rho_t(\D x) \dif t \quad \text{for all } \psi\in C^{\infty}_{\textup{c}}(\R^d)
\end{equ}
for almost every $t\in[0,1]$.
As $v_t\in T_{\rho_t}(\mathcal{P}_2(\R^d))$ for almost every $t\in[0,1]$ this implies $v_t=0$ for almost every $t\in[0,1]$ and hence $v_\bullet=0$ in $L^2_TL^2(\rho_t)$.
Now, we proceed and estimate
\begin{equs}
    & \sup_{\phi\in C_{\textup{c}}^\infty((0,1)\times\R^d)} - \int_0^1 \int_{\R^d} \left( \partial_t\phi_t + \frac12 \lVert \nabla \phi_t \rVert_2^2 \right) \rho_t(\D x) \dif t
    \\ \qquad & \ge
    - \int_0^1 \int_{\R^d} \left( \partial_t\phi_t^\varepsilon + \frac12 \lVert \nabla \phi_t^\varepsilon \rVert_2^2 \right) \rho_t(\D x) \dif t
    \\ & =
    \int_0^1 \int_{\R^d} \left( v_t^\rho\cdot \nabla\phi_t^\varepsilon - \frac12 \lVert \nabla \phi_t^\varepsilon \rVert_2^2 \right) \rho_t(\D x) \dif t
    \\ & \to
    \frac12 \int_0^1 \int_{\R^d} \lVert v_t^\rho \rVert_2^2  \rho_t(\D x) \dif t
    = \mathcal{A}(\rho_\bullet)
\end{equs}
for $\varepsilon\to0$, showing the other inequality.
\end{proof}
\begin{lemma}[First variation of the action]\label{lem:first-variation}
    Consider $\rho_\bullet, \sigma_\bullet\in\textup{AC}_T^2(\mathcal{P}_2(\R^d))$, set $\xi_t\coloneqq \sigma_t-\rho_t$, and denote the velocity of $\rho_\bullet$ by $v_\bullet^\rho\in \mathcal{T}_{\rho_\bullet}\textup{AC}_T^2(\mathcal{P}_2(\R^d))$.
    Assume now that $\sigma_t\ll\rho_t$ and $\frac{\D \sigma_\bullet}{\D \rho_\bullet} \in L^\infty_T L^\infty(\rho_\bullet)$, as well as the existence of $v^\xi_\bullet \in L^2_TL^2(\rho_\bullet)$ such that
    \begin{equ}
        \partial_t \xi_t + \nabla \cdot (\rho_t v^\xi_t)=0
    \end{equ}
    holds in the distributional sense.
    Then it holds that
    \begin{equ}\label{eq:first-variation-action}
        \delta^+ \mathcal{A}(\rho_\bullet)[\xi_\bullet]
        = \int_0^1 \left( \int_{\R^d}
          v_t^\rho\cdot v_t^\xi \rho_t(\D x)
          - \int_{\R^d} \frac12 \lVert v_t^\rho \rVert^2 \xi_t(\D x) \right) \dif t.
    \end{equ}
\end{lemma}
\begin{proof}
The dual formulation of the action functional gives
\begin{equ}
    \mathcal{A}(\rho_\bullet) = \sup_{\phi_\bullet \in C_{\textup{c}}^\infty((0,1)\times\R^d)} L_\phi(\rho_\bullet) ,
\end{equ}
where
\begin{equs}
    L_\phi(\rho_\bullet)
    & = - \int_0^1 \int_{\R^d} \left(\partial_t\phi_t + \frac12 \lVert \nabla\phi_t \rVert_2^2 \right) \rho_t(\D x) \dif t ,
\end{equs}
where $L_\phi$ is a linear function of $\rho_\bullet$.
We now consider the following functionals on $\R$ given by
\begin{equ}
    a_\phi(h) \coloneqq L_\phi(\rho_\bullet + h \xi_\bullet),
\end{equ}
which is well defined for $h\in[0,1]$ and can be extended to globally linear functions on $\R$.
Further, we set
\begin{equ}
    f(h) \coloneqq \sup_{\phi\in C^\infty_{\textup{c}}((0,1)\times\R^d)} a_\phi(h).
\end{equ}
Note that $f(h)= \mathcal{A}(\rho_\bullet + h\xi_\bullet)$ for all $h\in[0,1]$.
It is the goal to compute the derivative $f'(0)$, for which we first compute its subgradient, which, as the subgradient of a support function, is given by
\begin{equ}
    \partial f(0) = \overline{\operatorname{conv}\left\{\lim_{n\to\infty} a_{\phi^n}(1)-a_{\phi^n}(0) : \lim_{n\to\infty}a_{\phi^n}(0)=f(0) \right\}},
\end{equ}
see for example \cite{rockafellar1997convex}.
Let us fix such a sequence $(\phi^{n}_\bullet)_{n\in\mathbb N}\subseteq C^\infty_{\textup{c}}((0,1)\times\R^d)$, then we have
\begin{equs}
    a_{\phi^n}(0) - f(0)
    & = L_{\phi^n}(\rho_\bullet) - \mathcal{A}(\rho_\bullet)
    \\ & =
    \int_0^1 \int_{\R^d} \left( \nabla \phi^n_t\cdot v^\rho_t - \frac12 \lVert \nabla \phi^n_t \rVert_2^2- \frac12 \lVert v^\rho_t \rVert_2^2 \right) \rho_t(\D x) \dif t
    \\ & =
    \int_0^1 \int_{\R^d}  \lVert \nabla \phi^n_t - v^\rho_t\rVert_2^2  \rho_t(\D x) \dif t \to0
\end{equs}
for $n\to\infty$, hence $\nabla \phi^n_\bullet\to v^\rho_\bullet$ for $n\to\infty$ in $L^2_TL^2(\rho_\bullet)$.
Now, we obtain
\begin{equs}
    a_{\phi^n}(1)-a_{\phi^n}(0) & =
    - \int_0^1 \left(\int_{\R^d} \partial_t\phi_t^n + \frac12 \lVert \nabla \phi_t^n \rVert^2 \right) \xi_t(\D x) \dif t
    \\ & = \int_0^1 \left(\int_{\R^d} \nabla\phi_t^n\cdot v^\xi_t  \rho_t(\D x) - \int_{\R^d}\frac12 \lVert \nabla \phi_t^n \rVert^2 \xi_t(\D x)\right) \dif t
\end{equs}
As $\sigma_\bullet$ is dominated by $\rho_\bullet$ with a globally bounded Radon--Nikodym derivative, we obtain
\begin{equ}
    \lim_{n\to\infty} a_{\phi^n}(1)-a_{\phi^n}(0) = \int_0^1 \left(\int_{\R^d} v_t^\rho\cdot v^\xi_t  \rho_t(\D x) - \int_{\R^d}\frac12 \lVert v_t^\rho \rVert^2 \xi_t(\D x)\right) \dif t
\end{equ}
as asserted.
\end{proof}
\begin{lemma}[Bregman divergence of the Benamou--Brenier action functional]
    Consider two curves $\rho_\bullet, \sigma_\bullet \in \textup{AC}_T^2(\mathcal{P}_2(\R^d))$
    with velocities $v^\rho_\bullet \in \mathcal{T}_{\rho_\bullet} \textup{AC}_T^2(\mathcal{P}_2(\R^d))$ and $v^\sigma_\bullet \in \mathcal{T}_{\sigma_\bullet} \textup{AC}_T^2(\mathcal{P}_2(\R^d))$, respectively.
    Assume that $\sigma_t\ll\rho_t$ and $\frac{\D \sigma_\bullet}{\D \rho_\bullet} \in L^\infty_T L^\infty(\rho_\bullet)$,
    that
    $v^{\sigma}_\bullet \in L^2_T L^2(\rho_\bullet)$,  $v^{\rho}_\bullet \in L^2_T L^2(\sigma_\bullet)$, and that there is a solution $v_\bullet\in \mathcal{T}_{\sigma_\bullet} \textup{AC}_T^2(\mathcal{P}_2(\R^d))$ to
    \begin{equ}
        \partial_t(\rho_t-\sigma_t)+\nabla\cdot(\sigma_t v_t)=0.
    \end{equ}
    Then the Bregman divergence of the action functional is given by
    \begin{equ}\label{eq:bregman-action}
        D_{\mathcal{A}}(\rho_{\bullet}, \sigma_{\bullet}) = \frac12\int_0^1 \int_{\R^d} \lVert v_t^\rho(x) - v_t^\sigma(x) \rVert_2^2 \rho_t(\D x) \dif t.
    \end{equ}
\end{lemma}
\begin{proof}
The Bregman divergence is given by
\begin{equ}
    D_{\mathcal A}(\rho_\bullet, \sigma_\bullet) = \mathcal{A}(\rho_\bullet) - \mathcal{A}(\sigma_\bullet) - \delta^+\mathcal{A}(\sigma_\bullet)[\rho_\bullet-\sigma_\bullet].
\end{equ}
To use the expression of the first variation from Lemma \ref{lem:first-variation}, we first note that
\begin{equ}
    -\nabla \cdot (\sigma_t v_t)=\partial_t (\rho_t-\sigma_t)= -\nabla \cdot (\rho_t v^\rho_t - \sigma_t v^\sigma_t),
\end{equ}
which holds in distribution meaning that for any $\psi\in C_c^{\infty}(\R^d)$ and almost every $t\in[0,1]$, we have
\begin{equs}
    \int_{\R^d} \nabla \psi(x) \cdot \nabla v_t(x) \sigma_t(\D x) & =
    \int_{\R^d} \nabla \psi(x) \cdot v_t^\rho(x) \rho_t(\D x)
    \\ & \quad -
    \int_{\R^d} \nabla \psi(x) \cdot v_t^\sigma(x) \sigma_t(\D x).
\end{equs}
By density and as $v^\sigma_t\in L^2(\rho_t)$ for almost every $t\in[0,1]$, we obtain
\begin{equs}
    \int_{\R^d} v^\sigma_t(x) \cdot v_t(x) \sigma_t(\D x) & =
    \int_{\R^d} v^\sigma_t(x) \cdot v_t^\rho(x) \rho_t(\D x)
    \\ & \quad -
    \int_{\R^d} v^\sigma_t(x) \cdot v_t^\sigma(x) \sigma_t(\D x)
\end{equs}
for almost every $t\in[0,1]$.
Now, we use this to compute
\begin{equs}
    D_{\mathcal A}(\rho_\bullet, \sigma_\bullet)
    & =
    \frac12\int_0^1 \int_{\R^d} \lVert \nabla \phi_t^\rho \rVert_2^2 \rho_t(\D x) \dif t
    - \frac12\int_0^1 \int_{\R^d} \lVert \nabla \phi^\sigma_t \rVert_2^2 \sigma_t(\D x) \dif t
    \\ &
    - \int_0^1 \left(\int_{\R^d} \nabla \phi^\sigma_t\cdot \nabla \phi_t \sigma_t(\D x) - \frac12 \int_{\R^d} \lVert \nabla \phi^\sigma_t \rVert_2^2 (\rho_t - \sigma_t)(\D x)\right) \dif t
    \\ & =
    \frac12\int_0^1 \int_{\R^d} \lVert \nabla \phi_t^\rho \rVert_2^2 \rho_t(\D x) \dif t
    - \frac12\int_0^1 \int_{\R^d} \lVert \nabla \phi^\sigma_t \rVert_2^2 \sigma_t \dif x \dif t
    \\ & \quad
    - \int_0^1 \left( \int_{\R^d} \nabla \phi^\sigma_t\cdot \nabla \phi_t^\rho \rho_t(\D x) - \int_{\R^d} \lVert \nabla \phi^\sigma_t \rVert_2^2 \sigma_t(\D x) \right) \dif t
    \\ & \quad + \frac12 \int_0^1 \int_{\R^d}\lVert \nabla \phi^\sigma_t  \rVert_2^2(\rho_t - \sigma_t)(\D x) \dif t
    \\ & =
    \frac12\int_0^1 \int_{\R^d} \left(\lVert \nabla \phi_t^\rho \rVert_2^2 - 2 \nabla\phi^\rho_t\cdot\nabla\phi^\sigma_t+\lVert \nabla \phi_t^\sigma \rVert_2^2\right) \rho_t(\D x) \dif t
    \\ & =
    \frac12\int_0^1 \int_{\R^d} \lVert \nabla \phi_t^\rho -\nabla\phi_t^\sigma \rVert_2^2 \rho_t(\D x) \dif t.
\end{equs} 
\end{proof}

\begin{proof}[Proof of \Cref{thm:convex-geometry-BB}]
By definition of the Bregman divergence, we have
\begin{equs}
    \frac{\mathcal{A}(\rho^h_{\bullet}) - \mathcal{A}(\rho_{\bullet}^0) - \delta \mathcal{A}(\rho_{\bullet}^0)[\rho^h_{\bullet} - \rho_{\bullet}^0]}{h^2}
    &= \frac{D_{\mathcal{A}}(\rho^h_{\bullet}, \rho_{\bullet}^0)}{h^2}.
\end{equs}
Substituting the explicit formula \eqref{eq:bregman-action}
for the Bregman divergence yields
\begin{equs}
    \frac{D_{\mathcal{A}}(\rho^h_{\bullet}, \rho_{\bullet}^0)}{h^2}
    & =
    \frac{1}{2h^2} \int_0^1 \int_{\R^d} \lVert \nabla\phi_t^h(x) - \nabla\phi_t(x) \rVert_2^2 \rho_t^h(\D x) \dif t \\
    &= \frac{1}{2h^2} \int_0^1 \int_{\R^d} \lVert h \nabla \psi_t(x) \rVert_2^2 \rho_t^h(\D x) \dif t \\
    &= \frac12\int_0^1 \int_{\R^d} \lVert \nabla \psi_t(x) \rVert_2^2 \rho_t^h(\D x) \dif t.
\end{equs}
By standard $W_2$ stability of the continuity equation, $\rho^h_t \to \rho_t^0$ in $W_2$ as $h \to 0$, see for example \cite{ambrosio2012user}.
Since $\nabla \psi_t\in L^2_TC^1_{\textup{Lip}}(\R^d)$ it grows at most linearly.
Since $W_2$ convergence implies convergence when testing with functions with at most quadratic growth, see \cite[Proposition 3.4]{ambrosio2012user}, we obtain pointwise in $t$ convergence of the integrals
\begin{equ}
    \lim_{h\to 0} \int_{\R^d} \lVert \nabla \psi_t(x) \rVert_2^2 \rho_t^h(\D x) = \int_{\R^d} \lVert \nabla \psi_t(x) \rVert_2^2 \rho_t^0(\D x). 
\end{equ}
Finally, since the second moments of $\rho_t^h$ are uniformly bounded in $ t$ and $ h$, see Lemma \ref{lem:evolution-W2}, we can apply the dominated convergence theorem to pass the limit inside the time integral and obtain 
\begin{equ}
    \lim_{h\to0} \frac{\mathcal{A}(\rho^h_{\bullet}) - \mathcal{A}(\rho_{\bullet}^0) - \delta \mathcal{A}(\rho_{\bullet}^0)[\rho^h_{\bullet} - \rho_{\bullet}^0]}{h^2}
    = \int_0^1 \int_{\R^d} \lVert \nabla \psi_t(x) \rVert_2^2 \rho_t^0(\D x) \dif t. 
\end{equ}
Further, from Definition \ref{def:AFR-velocities} and Lemma \ref{thm:FR-isometrc} it holds that 
\begin{equ}
    g_{\rho_{\bullet}^0}^{\AFR}(\xi_{\bullet}, \xi_{\bullet}) 
    = g_{\nabla \phi_{\bullet}}^{\AFR}(\nabla \psi_\bullet, \nabla \psi_\bullet)
    = \int_0^1 \int_{\R^d} \lVert \nabla \psi_t(x) \rVert_2^2 \rho_t^0(\D x) \dif t. 
\end{equ}
\end{proof}
Consider the case that $\rho_\bullet, \rho_\bullet^\star \in \textup{AC}_T^2(\mathcal{P}_2(\R^d))$.
A natural projection of the flow-matching energy to the space of paths of densities is given by
\begin{equ}
    \mathcal{E}(\rho_{\bullet}) = \frac12 D_{\mathcal{A}}(\rho_{\bullet}^\star, \rho_{\bullet}) = \frac12\int_0^1 \int_{\R^d} \lVert v^\rho_t(x) - v_t^\star(x) \rVert_2^2 \rho_t^\star(x) \dif x \dif t,
\end{equ}
where $v^\rho_\bullet \in T_{\rho_\bullet} L^2_T \mathcal{P}_2(\R^d)$ and $v^\star_\bullet \in T_{\rho_\bullet^\star} L^2_T \mathcal{P}_2(\R^d)$ are the velocities of $\rho_\bullet$ and $\rho_\bullet^\star$, respectively, meaning that
\begin{equs}\label{eq:phis}
    \partial_t \rho_t + \nabla\cdot(\rho_{t} v_{t}^{\rho}) & = 0 \quad \text{and} \\
    \partial_t \rho_t^\star + \nabla\cdot(\rho_{t}^\star v^\star_{t}) & = 0.
\end{equs}
Also in this case, we obtain a characterization of the gradient in terms of an optimal update direction.

\begin{theorem}[Optimality] \label{thm:measures-optimality}
Consider $\rho^\star_\bullet\in \textup{AC}_T^2(\mathcal{P}_2(\R^d))$ and gradient fields $v_\bullet=\nabla\phi_\bullet, u_\bullet=\nabla\psi_\bullet \in L^2_T C^1_{\textup{Lip}}(\R^d)$ and denote the solutions to the continuity equation with velocity field $v_\bullet+hu_\bullet$ by $\rho^{v+hu}_\bullet$  and $\rho^v_0=\rho^\star_0$ and consider the variation $\xi_\bullet = \ddh \rho^{v+hu}_\bullet\in T_{\rho_\bullet} \textup{AC}_T^2(\mathcal{P}_2(\R^d))$. 
Assume that $\rho_\bullet^v\ll\rho^\star_\bullet$ and that 
\begin{equ}
    (v_\bullet - v_\bullet^\star) \cdot \frac{\D \rho^\star_\bullet}{\D \rho^v_\bullet} \in L^2_T L^2(\rho^v_\bullet),
\end{equ}
where $v_\bullet^\star\in T_{\rho_\bullet^\star} L^2_T(\mathcal{P}_2(\R^d))$ is the velocity field for $\rho_\bullet^\star$. 
Then it holds that
\begin{equ}
    \delta \mathcal{E}(\rho_{\bullet}^v)[\xi_\bullet] = g_{\rho_\bullet^v}^{\AFR}(\rho_{\bullet}^v-\rho^\star_{\bullet}, \xi_\bullet).
\end{equ}
\end{theorem}
\begin{proof}
Let us again fix $\nabla\psi^\xi\in L^2_T C^1_{\textup{Lip}}(\R^d)$.
The first variation is given by 
\begin{equs}
    \delta \mathcal{E}(\rho_{\bullet})[\xi_{\bullet}] & = \ddh \mathcal{E}(\rho^{v+hu}_{\bullet}) \\
    & = \ddh \frac12\int_0^1 \int_{\R^d} \lVert (v_t+hu_t)(x) - v_t^\star(x) \rVert_2^2 \rho_t^\star(\D x)  \dif t \\
    & = \int_0^1 \int_{\R^d} (v_t(x) - v_t^\star(x))\cdot u_t(x) \rho_t^\star(\D x) \dif t,
\end{equs}
where again
\begin{equ}
    \partial_t \xi_t + \nabla\cdot(\xi_t v_t) = -\nabla\cdot(\rho_t u_t)
\end{equ}
by Lemma \ref{thm:derivative-solution}.
Note that for $\zeta_\bullet = \rho_\bullet^v -\rho^\star_\bullet$ we have
\begin{equ}
    \partial_t \zeta_t + \nabla\cdot(\zeta_t v_t) = -\nabla\cdot(\rho_t^\star (v_t - v_t^\star)) = -\nabla\cdot(\rho_t \tilde{w}_t),
\end{equ}
where $\tilde{w}_\bullet = (v_\bullet - v_\bullet^\star) \frac{\D \rho^\star_\bullet}{\D \rho^v_\bullet} \in L^2_TL^2(\rho_\bullet)$.
Let $w_\bullet \in \mathcal{T}_{\rho_\bullet^v} L^2_T(\mathcal{P}_2(\R^d))$ be the $L^2_TL^2(\rho_\bullet)$ projection of $\tilde{w}_\bullet$ to the gradient fields, then we obtain  \begin{equ}
    \partial_t \zeta_t + \nabla\cdot(\zeta_t v_t) = -\nabla\cdot(\rho_t w_t) 
\end{equ}
as well as 
\begin{equ}
     -\nabla\cdot(\rho_t w_t) = -\nabla\cdot(\rho_t^\star (v_t - v_t^\star)). 
\end{equ}
Definition \ref{def:AFR-measures} now yields 
\begin{equs}
    g_{\rho_{\bullet}^v}^{\AFR}(\zeta_\bullet, \xi_\bullet)
    & = \int_0^1 \int_{\R^d} w_t(x)\cdot u_t(x) \rho_t^v(\D x) \dif t
    \\ & = \int_0^1 \int_{\R^d} (v_t(x) - v_t^\star(x))\cdot u_t^\xi(x) \rho_t^\star(\D x) \dif t
    \\ & = \delta \mathcal{E}(\rho_{\bullet}^v)[\xi_{\bullet}],
\end{equs}
which shows the assertion.
\end{proof}

\section{The Flat $L^2(\rho^\star_\bullet)$-Geometry}

Next, we contrast the advective Fisher--Rao metric with the flat
$L^{2}$-metric on velocity fields, which is given by
\begin{equation}
    (u_{\bullet},w_{\bullet})_{L^{2}(p_{\bullet}^{\star})}=\int_{0}^{1}\int_{\R^d}u_{t}(x)^{\top}w_{t}(x) \rho_{t}^{\star}(\D x)\dif t.
\end{equation}
Recall that we are investigating the energy function
\begin{equ}
    E\colon L^2_T C^1_{\textup{Lip}} \to \R, \quad v_\bullet \mapsto \frac12 \int_0^1 \int_{\R^d} \lVert v_t(x) - v^\star_t(x) \rVert_2^2 \rho^\star_t(\D x) \dif t.
\end{equ}
The $L^{2}$-metric agrees with the Hessian of $E$, which is a quadratic functional on the velocity fields. 
Hence, it is well known that the gradient points directly into the direction of the optimal velocity field $v_{\bullet}^{\star}$.

\begin{lemma}\label{lem:l2-gradient}
    The $L^{2}(\rho_{\bullet}^{\star})$-gradient of the energy $E\colon L^{2}(\rho_{\bullet}^{\star})\to\R$ is given by $v_{\bullet}-v_{\bullet}^{\star}$.
\end{lemma}
\begin{proof}
Note that the energy is given by the squared norm $E(v_\bullet)=\frac12\lVert v_\bullet-v_\bullet^{\star}\rVert_{L^2(\rho_\bullet^{\star})}^2$.
A direct computation yields that for any $u_{\bullet}\in L^{2}(\rho_{\bullet}^{\star})$ it holds that
\begin{equation*}
\begin{split}
	\delta E(v_{\bullet})[u_{\bullet}] & =\int_{0}^{1}\int_{\R^d}(v_{t}(x)-v_{t}^{\star}(x))\cdot u_{t}(x)\rho_{t}^{\star}(\D x)\dif t
	=(v_{\bullet}-v_{\bullet}^{\star},u_{\bullet})_{L^{2}(\rho_{\bullet}^{\star})},
\end{split}
\end{equation*}
which yields the claim.
\end{proof}
The $L^2$ gradient will not lead to a mixture geodesic in the space of paths of probability densities, but we can show that it is close to the mixture geodesic if the energy is small.

\begin{theorem}[$L^{2}$-gradients and $m$-geodesics]\label{thm:mismatch}
Consider an initial distribution $\rho_0\in\mathcal{P}_2(\R^d)$, velocity fields $v_{\bullet},v^{\star}_{\bullet}\in L^2_TC^{1}_{\textup{Lip}}$, the variation $u_{\bullet}\coloneqq v_{\bullet}-v_{\bullet}^{\star}$, denote the solution of the continuity equation with the velocity fields $v_\bullet$, $v^\star_\bullet$, and $v_\bullet+hu_\bullet$ by $\rho^{v}_\bullet$, $\rho^{\star}_\bullet$, and $\rho^{v+hu}_\bullet$,
and denote the corresponding variation $\xi_{\bullet}\coloneqq\ddh \rho_{\bullet}^{v+hu}$.
Then there is a constant $c>0$ such that for all $t\in[0,1]$ it holds that  
\begin{equ}
	\lVert \xi_{t} - (\rho_t^v - \rho^\star_t ) \rVert_{(C^1_{\textup{Lip}})^\ast}
    \le
    c  E(v_\bullet)^\frac12 . 
\end{equ}
Further, the constant can be chosen to be 
\begin{equ}
    c = \sqrt{2} \left(2e^{\lVert v_\bullet \rVert_{\V}} + \left(\lVert v_\bullet \rVert_{\V} + \lVert v^\star_\bullet \rVert_{\V} \right) e^{2\lVert v_\bullet \rVert_{\V}}\right). 
\end{equ}
\end{theorem}
\begin{proof}
For any $\phi\in C^1_{\textup{Lip}}$ and $t\in[0,1]$, Lemma \ref{thm:derivative-solution} yields that 
\begin{equ}
    \langle \xi_t^{u}, \phi\rangle = \int_{\R^d} \nabla \phi(\varphi^v_t(x)) \cdot \psi_t^{u}(x) \rho_0(\D x),
\end{equ}
where $\varphi^v_\bullet$ denotes the flow induced by $v_\bullet$ and $\psi^{u}_\bullet$ is the unique solution to
\begin{equ}
    \partial_{t}\psi_{t}^{u}(x)=Dv_{t}(\varphi_{t}^{v}(x))\cdot\psi_{t}^{u}(x)+u_{t}(\varphi_{t}^{v}(x)), \quad \psi_0^v(x)=0
\end{equ}
for all $t\in[0,1]$ and $x\in\R^d$. 
It holds that 
\begin{equs}
    \left\lvert \langle \xi_t^{u} - (\rho^v_t-\rho^\star_t), \phi\rangle \right\rvert 
    & = \left\lvert \int_{\R^d} \nabla\phi(\varphi_t^v) \cdot \psi_t^u - (\phi(\varphi_t^v)-\phi(\varphi_t^\star))  \rho_0(\D x) \right\rvert
    \\ & \le \lVert \nabla\phi \rVert_{\infty} \int_{\R^d} \left(\lVert \psi_t^u \rVert + \lVert \varphi_t^v-\varphi_t^\star \rVert \right) \rho_0(\D x)
    \\ & \le \lVert \nabla\phi \rVert_{\infty} \left( \underbrace{\int_{\R^d} \lVert \psi_t^u \rVert \rho_0(\D x)}_{I_1} +  \underbrace{\int_{\R^d} \lVert  \varphi_t^v-\varphi_t^\star \rVert \rho_0(\D x)}_{I_2}\right). 
\end{equs}
The second term can be estimated analogously to Lemma \ref{lem:W2-stability}, which gives 
\begin{equs}\label{eq:proof:mismatch}
    I_2 & = \lVert \varphi^v_t - \varphi^\star_t \rVert_{L^1(\rho_0)}
    \le \lVert \varphi^v_t - \varphi^\star_t \rVert_{L^2(\rho_0)} 
    \le \sqrt{2} e^{\lVert v_\bullet\rVert_{\V}} E(v_\bullet)^\frac12.  
\end{equs}
The first term can be estimated according to 
\begin{equs}
    \lVert \psi_t(x) \rVert \le \int_0^t \lVert Dv_s( \varphi^v_t(x)) \cdot \psi_s^u(x) + u_s(\psi_s^v(x)) \rVert \dif s. 
\end{equs}
Gr\"{o}nwall's inequality yields 
\begin{equs}
    \lVert \psi_t(x) \rVert 
    & \le \exp\left( \int_{0}^t \lVert Dv_r \rVert_{\infty} \dif r \right) \int_0^t  \lVert u_s(\varphi^v_s(x)) \rVert \dif s 
    \\ & \le e^{\lVert v_\bullet \rVert_{\V}}\int_0^t  \lVert u_s(\varphi^v_s(x)) \rVert \dif s. 
\end{equs}
Integration with respect to $\rho_0$ yields 
\begin{equs}
    I_1
    & \le e^{\lVert v_\bullet \rVert_{\V}} \int_0^t \int_{\R^d} \lVert u_s(\varphi^v_s(x)) \rVert \rho_0(\D x) \dif s. 
    \\ & \le e^{\lVert v_\bullet \rVert_{\V}} \int_0^t \int_{\R^d} \lVert u_s(x) \rVert \rho_s^v(\D x) \dif s 
    \\ & \le e^{\lVert v_\bullet \rVert_{\V}} \int_0^t \int_{\R^d} \lVert u_s(x) \rVert \rho_s^\star(\D x) \dif s + e^{\lVert v_\bullet \rVert_{\V}} \int_0^t [u_s]_{\textup{Lip}} W_1(\rho^v_s, \rho^\star_s) \dif s 
\end{equs}
The first term can be estimated as follows 
\begin{equs}
    \int_0^t \int_{\R^d} \lVert u_s \rVert \rho_s^\star(\D x) \dif s 
    & \le 
    \int_0^t \left(\int_{\R^d} \lVert u_s \rVert^2 \rho_s^\star(\D x)\right)^\frac12 \dif s 
    \\ & \le 
    \sqrt{t}\left( \int_0^t \int_{\R^d} \lVert v_s-v^\star_s \rVert^2 \rho_s^\star(\D x) \dif s \right)^\frac12
    \\ & \le \sqrt{2} E(v_\bullet)^\frac12. 
\end{equs}
To estimate the second component, we bound the Wasserstein distance according to 
\begin{equs}
    W_1(\rho^v_s, \rho^\star_s) \le \lVert \varphi^v_s - \varphi^\star_s \rVert_{L^1(\rho_0)}
    \le \sqrt{2} e^{\lVert v_\bullet\rVert_{\V}} E(v_\bullet)^\frac12,
\end{equs}
where we used \eqref{eq:proof:mismatch}. 
This implies 
\begin{equs}
    \int_0^t [u_s]_{\textup{Lip}} W_1(\rho^v_s, \rho^\star_s) \dif s\cdot 
    & \le  \sqrt{2} \int_0^t [u_s]_{\textup{Lip}} \dif s e^{\lVert v_\bullet\rVert_{\V}} E(v_\bullet)^\frac12
    \\ & \le \sqrt{2} \lVert v_\bullet - v^\star_\bullet \rVert_{\V} e^{\lVert v_\bullet\rVert_{\V}} E(v_\bullet)^\frac12
    \\ & \le \sqrt{2} \left(\lVert v_\bullet \rVert_{\V} + \lVert v^\star_\bullet \rVert_{\V} \right) e^{\lVert v_\bullet\rVert_{\V}} E(v_\bullet)^\frac12
\end{equs}
Combining the bounds yields 
\begin{equs}
    I_1\le \sqrt{2} \left(e^{\lVert v_\bullet \rVert_{\V}} + \left(\lVert v_\bullet \rVert_{\V} + \lVert v^\star_\bullet \rVert_{\V} \right) e^{2\lVert v_\bullet \rVert_{\V}}\right) E(v_\bullet)^\frac12. 
\end{equs}
Overall, we obtain 
\begin{equs}
    \left\lvert \langle \xi_t^{u} - (\rho^v_t-\rho^\star_t, \phi\rangle \right\rvert 
    & \le \lVert \nabla \phi \rVert_{\infty} (I_1+I_2) 
    \le c \lVert \nabla \phi \rVert_{\infty}  E(v_\bullet)^\frac12,
\end{equs}
where the constant $c>0$ is given by 
\begin{equ}
    c = \sqrt{2} \left(2e^{\lVert v_\bullet \rVert_{\V}} + \left(\lVert v_\bullet \rVert_{\V} + \lVert v^\star_\bullet \rVert_{\V} \right) e^{2\lVert v_\bullet \rVert_{\V}}\right). 
\end{equ}
\end{proof}

\Cref{thm:mismatch} ensures that the update direction of the $L^2$-gradient is close to the mixture geodesic if the energy $E(v_\bullet)$ is small, which will typically be the case late in training.
Note that the norm $\lVert v_{\bullet}\rVert_{L^2_T C^1_{\textup{Lip}}}$ is not controlled by the energy $E(v_{\bullet})$. 
However, it can be controlled by the Lipschitz constant of the network function.
For example, it can be controlled by means of the magnitude of the weights of the network, which is commonly done in practice.

\section{Implications for Flow-Based Generative Models}\label{subsec:NG-method}
In generative modeling, it is the goal to design a sampling procedure that approximately samples from a target distribution $p_{\textup{data}}$, which is unknown but can be accessed via samples.
Most modern approaches first sample $X_{0}\sim p_{0}$ from an easy distribution $p_{0}$ like a Gaussian or uniform distribution and then apply a transformation $X_{1}=\Phi(X_{0})$.
Then, it is the task to choose $\Phi$ such that $X_{1}$ is approximately distributed according to $p_{\textup{data}}$, hence such that $\Phi_{\sharp}p_{0}\approx p_{\textup{data}}$.
Earlier approaches like variational auto-encoders (VAEs) \cite{Kingma2014}, generative adversarial networks (GANs) \cite{goodfellow2014generative} directly parameterize the transformation $\Phi^\theta$ and optimize the models parameters to minimize $D(\Phi^\theta_\sharp p_0, p_{\textup{data}})$, where $D$ is some distance measure.
Recently, a new paradigm has become very popular, where the samples $X_0$ from $p_0$ are transformed successively into synthetic data $X_1$.
Most notably, this paradigm shift has led to the development of many successful models, including normalizing flows \cite{Dinh2015}, diffusion models \cite{Sohl-Dickstein2015, Ho2020,rombach2022high}, score and flow matching \cite{lipman2023flow}, and Schr\"odinger bridges \cite{de2021diffusion}.
In flow matching, one considers the flow $\varphi_{t}$ associated to a velocity field $v_{t}$ given by
\begin{equation*}
	\partial_{t}\varphi_{t}(x)=v_{t}(\varphi_{t}(x)).
\end{equation*}
Then, synthetic data is produced by sampling $X_{0}\sim p_{0}$ and transforming the data along the flow until the time $t=1$.
Hence, one considers $X_{1}=\varphi_{1}(X_{0})$ and one hopes to approximate $(\varphi_{1})_{\sharp}p_{0}\approx p_{\textup{data}}$.
The velocity $v_{\bullet}$ is then parametrized by a neural network $v_{\bullet}^{\theta}\approx v_{\bullet}$ and the parameters $\theta$ are optimized to minimize the objective
\begin{equation}
	L_{\textup{FM}}(\theta)=\frac{1}{2}\int_{0}^{1}\int_{\R^d}\|v_{t}^{\theta}(x)-v_{t}^{\star}(x)\|^{2}p_{t}^{\star}(x)\dif x\dif t.
\end{equation}
Here, a reference velocity field $v_{\bullet}^{\star}$ is constructed to perfectly match the data, meaning that $(\varphi_{1}^{\star})_{\sharp}p_{0}=p_{\textup{data}}$, where $\varphi_{\bullet}^{\star}$ denotes the flow induced by $v_{\bullet}^{\star}$.
Various ways to construct such a reference velocity field have been proposed in the literature, such as via conditional interpolation and optimal transport \cite{lipman2023flow}, via the Schr\"odinger bridge problem \cite{de2021diffusion}.
For more details on such constructions, we refer to the general frameworks for stochastic interpolants \cite{albergo2025stochastic,wald2025flow} and the overview articles \cite{lipman2024flow,holderrieth2025introduction,pierret2026flow}.
Note that typically, the velocity field $v_{\bullet}^{\star}$ is not accessible directly, but rather the conditioned version on the initial data $v_{t}^{\star}(x|X_{0})$.
Similar to flow matching, diffusion models aim to approximate a reference velocity of a---possibly stochastic---differential equation, thereby approximately transporting $p_{0}$ to the data distribution $p_{\textup{data}}$.
They use a forward-backward process where first $p_{\textup{data}}$ is diffused by a forward dynamic
\begin{equation*}
	\D X_{t}^{\star}=f_{t}(X_{t}^{\star})\D t+\D W_{t}
\end{equation*}
leading to densities $p_{t}^{\star}=\mathcal{L}(X_{t}^{\star})$, where $f_{\bullet}$ is a drift term chosen by the practitioner.
At some time $T>0$, the dynamics are assumed to have approximately equilibrated to $p_{T}^{\star}\approx p_{0}$.
It is common to approximate the score function $s_{t}^{\star}=\nabla\ln p_{T-t}^{\star}$ with a neural network $s_{\bullet}^{\theta}$ by minimizing the objective
\begin{equ}
	L_{\textup{SM}}(\theta)=\frac{1}{2}\int_{0}^{1}\int_{\R^d}\|s_{t}^{\theta}(x)-s_{t}^{\star}(x)\|^{2}p_{t}^{\star}(x)\dif x\dif t.
\end{equ}
The data is then either generated by simulating the probability flow ODE
\begin{equ}
	\dot{x}_{t}=f_{t}(x_{t})-\frac{1}{2}s_{t}^{\theta}(x_{t})
\end{equ}
or the reverse-time SDE
\begin{equ}
	\D X_{t}=\left(f_{t}(X_{t})-\frac{1}{2}s_{t}^{\theta}(X_{t})\right)\D t+\D\tilde{W}_{t},
\end{equ}
which both induce the same marginal distributions as the forward SDE~\cite{lai2025principles}.
A second important paradigm shift, next to the parametrization of the instantaneous change of the data that led to the empirical success of flow-based generative models, lies in the choice of the loss function.
Earlier flow-based generative models like normalizing flows use $D((\varphi_1^\star)_\sharp p_0, p_{\textup{data}})$, which requires both the simulation of the forward dynamics as well as the differentiation through the numerical solver of the forward dynamics.
In contrast, flow matching and score matching objectives operate on the level of velocity fields.
Further, it measures the discrepancy of the velocity fields in terms of the reference densities $p_{t}^{\star}$, which avoids the costly simulation of the forward dynamics induced by $v^\theta_\bullet$.
For this reason, flow matching and score matching are often referred to as \emph{simulation-free}.
The choice of the loss function as a square-based objective has great computational advantages, which have arguably led to the success of flow-based generative models.
However, it comes at the cost of introducing a proxy for the eventual goal of matching the data distribution at the terminal time.
Indeed, the standard Flow Matching loss acts on an Eulerian level of velocity fields, where the generative performance depends on the densities induced by the trajectories.
At the level of velocities, the least-squares loss is simply a quadratic function. However, it is inherently connected to relative entropy at the level of path measures, a concept that is well known in stochastic optimal control theory~\cite{follmer1985entropy}.
This already suggests the existence of a non-trivial geometry underlying the optimization of flow-based generative models, which is not visible at the level of velocity fields.
In this work, we reveal the geometry induced by the least-squares objectives of flow-based generative models.
\subsection{Computational considerations}
When training a flow-matching model, we do not have access to the reference velocity field $v^\star_\bullet$; hence, the flow-matching loss
\begin{equ}
    L_{\textup{FM}}(\theta) = \frac12\int_0^1 \int_{\R^d} \lVert v_t^\theta(x) - v_t^\star(x) \rVert_2^2 p_t^\star(x) \dif x \dif t
\end{equ}
is not directly accessible.
Therefore, it is common to use a conditional flow-matching loss, which is given by
\begin{equ}
    L_{\textup{CFM}}(\theta) = \frac12\int_0^1 \int_{\R^d} \int_{\R^d} \left\lVert v_t^\theta(x) - v_t^\star(x|x_0) \right\rVert^2 p_t(x|x_0) \dif x \dif x_0 \dif t,
\end{equ}
where $v_t^\star(x|x_0)$ is a conditional reference velocity field.
One of the most common choices for the conditional reference velocity field is given by the velocity field obtained through sample-wise interpolation, which is given by
\begin{equ}
    v_t^\star(x|x_0) = \frac{x-x_0}{1-t}.
\end{equ}
We refer to this approach as \emph{conditional linear interpolation}, where a variety of other methods and frameworks, which use couplings of the source and target distributions, Markov kernels, and stochastic processes, have been suggested \cite{lipman2023flow,albergo2025stochastic,wald2025flow}.
It is well-known that the conditional flow-matching loss differs from the unconditional flow-matching loss only by a constant; hence, it can be used as a surrogate loss for training, as $\nabla L_{\textup{FM}}=\nabla L_{\textup{CFM}}$.
Note that samples from the reference density $p^\star_t$ are given by conditional interpolation; then we can sample from $p^\star_t$ by sampling $X_0$ and $X_1$ from $p_0$ and $p_1$, respectively, and setting
\begin{equ}
    X_t = (1-t)X_0 + tX_1.
\end{equ}
\paragraph{The natural gradient method}
In the context of flow-based generative models, the Fisher-information matrix is given by
\begin{equ}\label{eq:definition-FIM}
    F(\theta)_{ij}=\int_{0}^{1}\int_{\R^d}\partial_{\theta_{i}}v_{t}^{\theta}(x) \cdot \partial_{\theta_{j}}v_{t}^{\theta}(x)p_{t}^{\theta}(x)\dif x\dif t.
\end{equ}
Note again the difference from the Gauss--Newton matrix, where the reference density $p_{\bullet}^{\star}$ is used instead of the current density $p_{\bullet}^{\theta}$.
This means that, different from the Gauss--Newton method, the natural gradient method requires the simulation of the forward dynamics to compute the current density $p_{\bullet}^{\theta}$.
The natural gradient method is given by the iteration
\begin{equation}
    \theta_{k+1}=\theta_{k}-\eta\cdot F(\theta_{k})^{+}\nabla L(\theta_{k}).
\end{equation}
\paragraph{The Gauss--Newton method}
The $L^{2}$-geometry of velocity fields leads to the Gauss--Newton method for optimization. It is given by the iteration
\begin{equ}
	\theta_{k+1}=\theta_{k}-\eta\cdot G(\theta_{k})^{+}\nabla L(\theta_{k}),
\end{equ}
where the Gauss--Newton matrix is given by
\begin{equ}\label{eq:definition-GNM}
	G(\theta)_{ij}
	= (\partial_{\theta_i} v^\theta_\bullet, \partial_{\theta_j} v^\theta_\bullet)_{L^2(p_\bullet^\star)}
	=\int_{0}^{1}\int_{\R^d}\partial_{\theta_{i}}v_{t}^{\theta}(x)\cdot\partial_{\theta_{j}}v_{t}^{\theta}(x)p_{t}^{\star}(x)\dif x\dif t.
\end{equ}
\paragraph{Estimation of preconditioning matrices}
Recall the definition of the Gramian matrices $F(\theta)$ and $G(\theta)$ in \eqref{eq:definition-FIM} and \eqref{eq:definition-GNM}, respectively.
In practice, we need to estimate these matrices from samples.
To estimate the Gauss--Newton matrix $G(\theta)$, we can use the samples $(t_k^\star, x^\star_k)$ used for the discretization of the loss, where $t^\star_k$ is sampled uniformly from $[0,1]$ and $x^\star_k$ is sampled from $p^\star_{t_k^\star}$.
Then, we obtain the empirical Gauss--Newton matrix
\begin{equ}
    \hat{G}(\theta)_{ij} = \frac{1}{N}\sum_{k=1}^N \partial_{\theta_i} v_{t_k^\star}^\theta(x_k^\star)\cdot \partial_{\theta_j} v_{t_k^\star}^\theta(x_k^\star).
\end{equ}
As the Fisher-information matrix uses the current density $p_t^\theta$ instead of the reference density $p_t^\star$, we need to sample from $p_t^\theta$ to obtain an empirical estimate.
To this end, we simulate the forward dynamics of the model to obtain samples $(t_k, x_k)$, where $t_k$ is sampled uniformly in time and $x_k$ is sampled from $p_t^\theta$.
Then, we can obtain an empirical estimate of the Fisher information matrix by
\begin{equ}
    \hat{F}(\theta)_{ij} = \frac{1}{N}\sum_{k=1}^N \partial_{\theta_i} v_{t_k}^\theta(x_k)\cdot \partial_{\theta_j} v_{t_k}^\theta(x_k).
\end{equ}
Therefore, the natural gradient method is not simulation-free and has an additional overhead compared to the Gauss--Newton method, which is simulation-free.
\paragraph{Linear algebra}
For the actual computation, we use standard implementation procedures for preconditioned and natural gradient systems as laid out in \cite{martens2020new,guzman-cordero2025improving}.
Both $G(\theta)$ and $F(\theta)$ are positive semi-definite Gram matrices, so their pseudo-inverses can be computed via the Cholesky decomposition.
To ensure numerical stability, we apply damping, also known as Tikhonov regularization, which gives
\begin{equs}
    (G(\theta) + \lambda I)\,\delta\theta_{\text{GN}} &= \nabla L(\theta), \\
    (F(\theta) + \lambda I)\,\delta\theta_{\text{NG}} &= \nabla L(\theta),
\end{equs}
where $\lambda > 0$ is a small regularization parameter.
The Gauss--Newton Gramian $G(\theta)$ uses the reference density $p_t^\star$, whereas the Fisher information $F(\theta)$ uses the current density $p_t^\theta$; the gradient $\nabla L(\theta)$ is the same in both cases.
Empirically, with $N$ samples, $\hat G = \frac1N J_\star^\top J_\star$, $\hat F = \frac1N J_\theta^\top J_\theta$, and $\nabla L = \frac1N J_\star^\top r$, where $J_\star$ and $J_\theta$ are the Jacobians of $v_t^\theta$ at samples from
$p_t^\star$ and $p_t^\theta$, respectively, and $r$ is the residual $v_t^\theta - v_t^\star$.
For small parameter dimension $p$, we factorize the damped Gramian via Cholesky and solve directly at $O(p^3)$ cost.
For GN, where both the Gramian and the gradient involve the same Jacobian $J_\star$, the push-through identity
\begin{equ}
    (J^\top J + \mu I)^{-1} J^\top = J^\top (J J^\top + \mu I)^{-1}
\end{equ}
reduces the $p\times p$ system to an $(Nd)\times(Nd)$ system:
\begin{equ}
    \delta\theta_{\text{GN}} = J_\star^\top (J_\star J_\star^\top + N\lambda I)^{-1} r.
\end{equ}
For NG, the Gramian uses $J_\theta$ while the gradient uses $J_\star$, so the push-through identity does not apply.
In this case, we either solve the damped $p\times p$ system directly (Cholesky, $O(p^3)$) when $p$ is small, or apply the Woodbury identity to the Gramian factor alone
\begin{equ}
    \delta\theta_{\text{NG}} = \frac{1}{N\lambda}
    \bigl[ I - J_\theta^\top (N\lambda I + J_\theta J_\theta^\top)^{-1} J_\theta \bigr]
    \,J_\star^\top r.
\end{equ}
When both $p$ and $d$ are large, we avoid forming the Gramian or the Jacobian explicitly and use conjugate gradient to solve the damped system iteratively via the conjugate gradient method, which only requires Jacobian--vector products, see \cite{martens2020new}.
\subsection{Visualization of update directions} \label{subsec:visualization}
Note that the $L^{2}$-gradient points directly into the direction of the optimal velocity field.
Hence, by the general projection property \eqref{eq:projection} of natural gradients, the Gauss--Newton method leads to an optimal fitting of the velocity fields.
\begin{figure}[!htbp]
    \centering
    \includegraphics{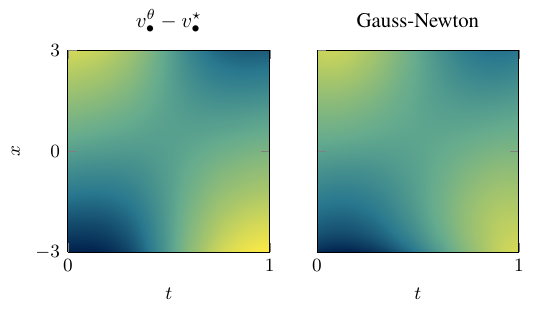}
    \includegraphics{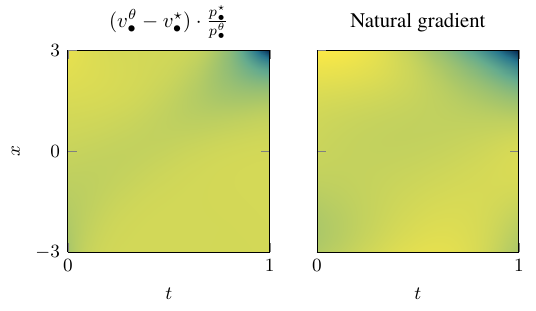}
    \includegraphics{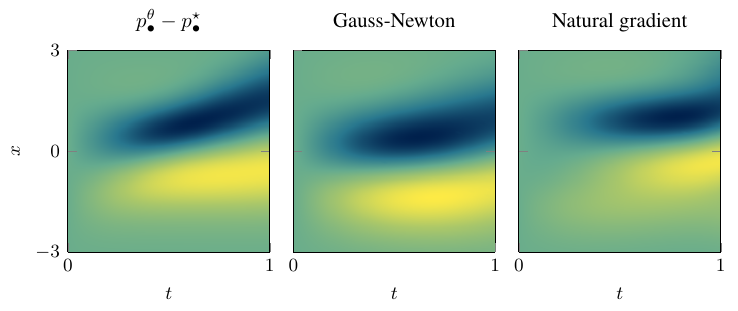}
    \caption{
      Shown are the gradient with respect to the $L^2$ metric $v_\bullet-v^\star_\bullet$ (top left) and and continuity Fisher--Rao metric $(v_\bullet-v^\star_\bullet)\cdot\frac{p^\star_\bullet}{p^v_\bullet}$ as well as the pushforward corresponding to the update directions of the Gauss--Newton method (top right) and natural gradient method (bottom right); from \eqref{eq:projection} we expect the pushforwards to agree with the gradients up to a projection; in the bottom row, the residual $p_\bullet-p^\star_\bullet$ in the space of paths of probability densities (left panel) as well as the pushforwards corresponding to the update directions of the Gauss--Newton and the natural gradient method (middle and right panels) are shown; from \Cref{eq:projection} we expect the pushforward of the natural gradient to agree with the residual up to a projection.
    }
    \label{fig:velocity-GN}
\end{figure}
We can visualize this in an easy case of matching two Gaussians in one dimension.
As a source distribution $p_0$, we consider a standard Gaussian $\mathcal{N}(0,1)$ and as a target distribution $p_{\textup{data}}$ we consider $\mathcal{N}(1,1)$.
One great benefit of this simple setting is that we have access to a closed-form expression of the reference velocity field $v^\star$ obtained through conditional interpolation.
First, the reference density $p_{t}^{\star}$ obtained by sample-wise interpolation is the Gaussian $\mathcal{N}(\mu_{t}^{\star},\sigma_{t}^{\star})$, where
\begin{equs}
	\mu_{t}^{\star} & =(1-t)\mu_{0}^{\star}+t\mu_{1}^{\star},\\
	(\sigma_{t}^{\star})^{2} & =(1-t)^{2}\sigma_{0}^{2}+t^{2}\sigma_{1}^{2}.
\end{equs}
Further, the marginal vector field is given by
\begin{equ}
	v_{t}^{\star}(x)=(\mu_{1}-\mu_{0})+\frac{t\sigma_{1}^{2}-(1-t)\sigma_{0}^{2}}{(1-t)^{2}\sigma_{0}^{2}+t^{2}\sigma_{1}^{2}}\cdot(x-\mu_{t}^{\star}),
\end{equ}
see \cite{bertrand2025closed}.
To visualize the update directions of the two optimizers, we consider a feedforward, fully connected neural network $v^\theta_\bullet$ with $2$ hidden layers of width $32$ and hyperbolic tangent as an activation function.
Here, we provide the time as an explicit input to the network, which is a common choice in the literature.
We initialize the weights of the network according to a normal distribution and the biases to zero and visualize the $L^2$-gradient, which agrees with the difference to the reference velocity field $v^\theta_\bullet-v^\star_\bullet$, see the left plot in \Cref{fig:velocity-GN}.
We compare this to the pushforward of the Gauss--Newton update direction, which is given by $DP(\theta)G(\theta)^{+}\nabla L(\theta)$, see the right plot in \Cref{fig:velocity-GN}.
The Gauss--Newton method leads to an optimal fitting of the velocity fields, but it is a method that is agnostic of the actual goal of fitting the probability densities induced by the corresponding generative model.
This motivates the construction of an optimizer based on the geometry of probability densities.
This is contrasted with the natural gradient method.
Here, the gradient in the space of velocities is weighted by the density ratio.
We compare the push forward $DP(\theta)\delta\theta_{\textup{NG}}$ of the natural gradient direction $\delta\theta_{\textup{NG}}$ to this weighted difference, see \Cref{fig:velocity-GN}.
\subsection{Computational Experiments}
Here, we conduct a series of computational experiments to compare the performance of the natural gradient and the Gauss--Newton method and the standard first-order optimizers like stochastic gradient descent (SGD) and Adam.
The code for these experiments will be made available at the time of publication.
We use a standard Gaussian as a source distribution and compare the optimizers on the following target distributions, which are:
\begin{enumerate}
    \item a $2D$ Gaussian mixture with two standard uniform modes
    \item a $2D$ Gaussian mixture with a standard uniform and an ill-conditioned mode
    \item a half-moon distribution
    \item a Gaussian mixture model with eight components 
    \item and a checkerboard distribution.
\end{enumerate}
For every experiment, we report the learning curves of the different optimizers in terms of the conditional flow-matching loss $L_{\textup{CFM}}$ as well as the $W_2$ distance to the target distribution, which we estimate from samples.
Further, we provide tables of the final accuracy of the optimizers when given the same computational budget.
\paragraph{Discussion}
Across all five target distributions, the second-order methods Gauss--Newton and natural gradient consistently achieve lower $W_2$ distances than the first-order baselines.
For all experiments, the training curves for a fixed number of iterations as well as samples produced after training are shown in \Cref{fig:balanced-gmm,fig:ill-gmm,fig:half-moons,fig:8gmm,fig:checkerboard}.
Here, multiple initializations are used, and the median and 10\% and 90\% percentiles are reported in the learning curves.
Finally, smoothed medians at the end of training are reported for all experiments in \Cref{tab:final-results}.
We consistently observe that the second-order optimizers in the form of the Gauss--Newton and the natural gradient method provide accuracy surpassing that of first-order methods even when given the same computational budget.
\begin{figure}
\centering
    \includegraphics{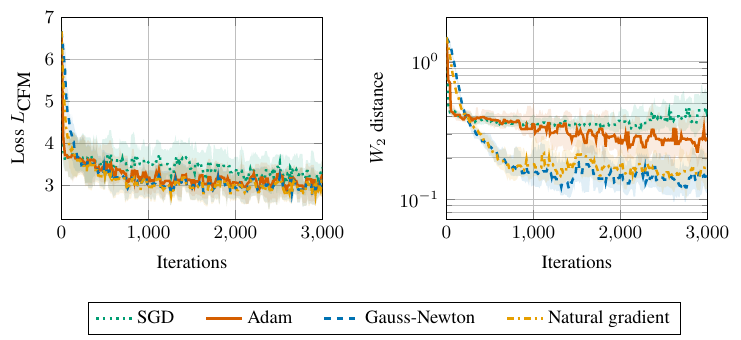}
    \includegraphics{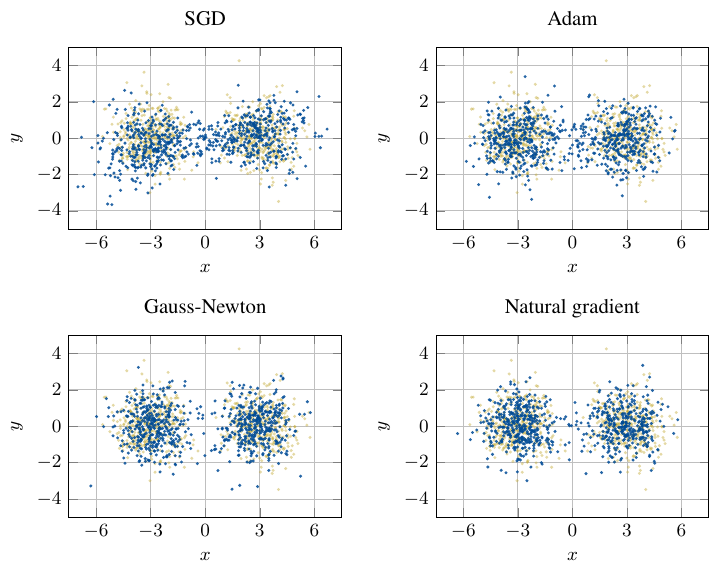}
    \caption{
      The top row shows the loss $L_{\textup{CFM}}$ (left) and the $W_2$ distance to the target distribution (right) for the different optimization methods;
      the curves show the median over 7 runs, while the shaded areas show the region between the 10\% and 90\% percentiles;
      note that the conditional flow matching loss differs from the pure regression loss by a constant; hence, it cannot be expected to vanish;
      the second and third rows show scatter plots of generated samples for each optimizer;
      each panel overlays generated samples (blue) with the target distribution (gold) for visual comparison.
    }
    \label{fig:balanced-gmm}
\end{figure}
\begin{figure}
\centering
    \includegraphics{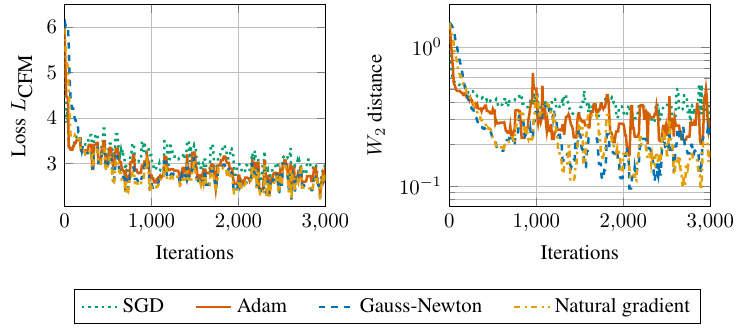}
    \includegraphics{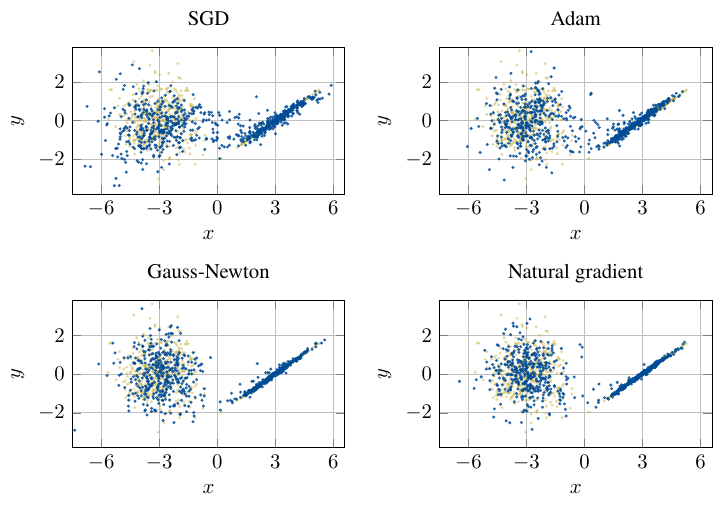}
    \caption{
      Ill-conditioned GMM: The top row shows the loss $L_{\textup{CFM}}$ (left) and the $W_2$ distance to the target distribution (right) for the different optimization methods;
      the curves show the median over 7 runs, while the shaded areas show the region between the 10\% and 90\% percentiles;
      note that the conditional flow matching loss differs from the pure regression loss by a constant; hence, it cannot be expected to vanish;
      the second and third rows show scatter plots of generated samples for each optimizer;
      each panel overlays generated samples (blue) with the target distribution (gold) for visual comparison.
    }
    \label{fig:ill-gmm}
\end{figure}
\begin{figure}
    \centering
    \includegraphics{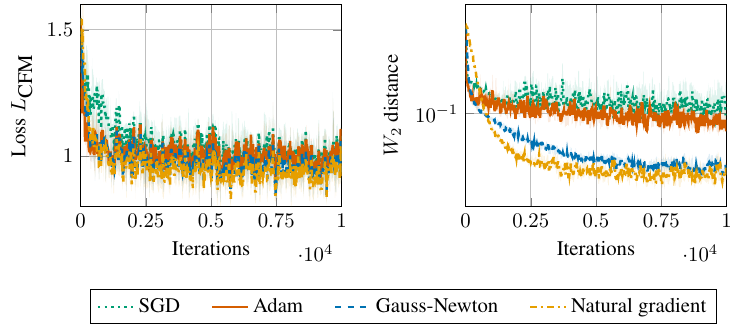}
    \includegraphics{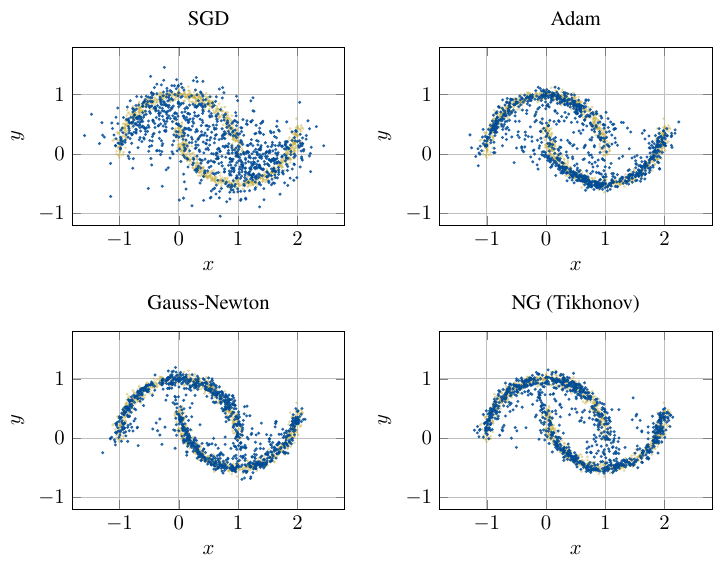}
    \caption{
      Half moons:
        The top row shows the loss $L_{\textup{CFM}}$ (left) and the $W_2$ distance to the target distribution (right) for the different optimization methods;
        the curves show the median over 7 runs, while the shaded areas show the region between the 10\% and 90\% percentiles;
        note that the conditional flow matching loss differs from the pure regression loss by a constant; hence, it cannot be expected to vanish;
        the second and third rows show scatter plots of generated samples for each optimizer;
        each panel overlays generated samples (blue) with the target distribution (gold) for visual comparison.
    }
    \label{fig:half-moons}
\end{figure}
\begin{figure}
    \centering
    \includegraphics{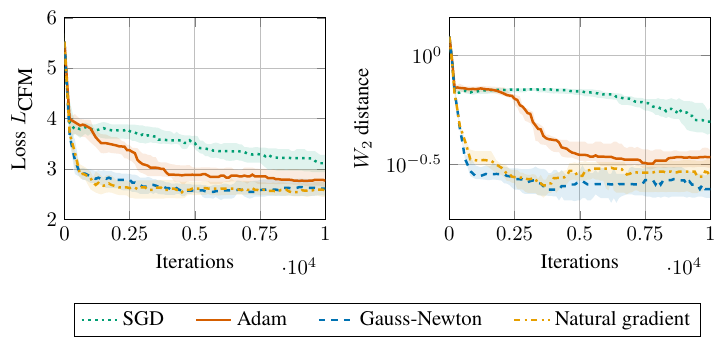}
    \includegraphics{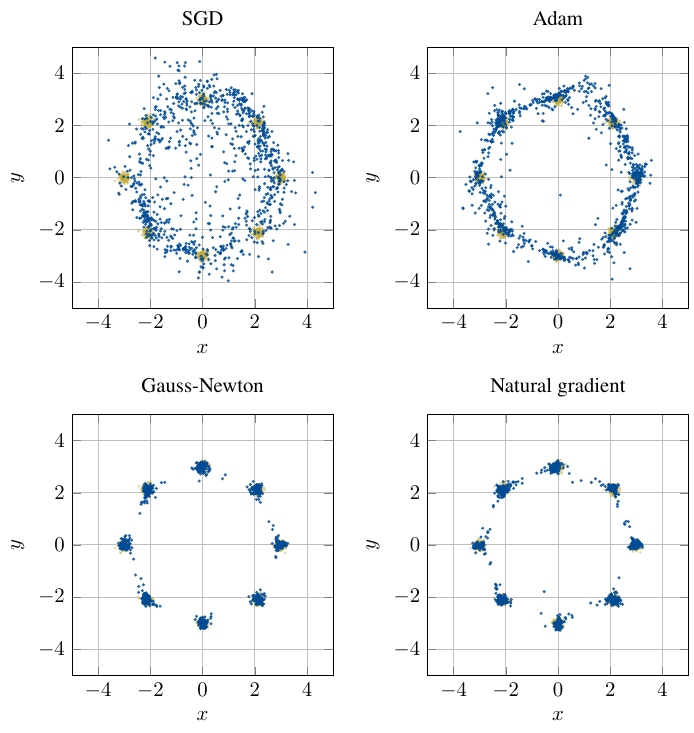}
    \centering
    \caption{
      Eight-GMM experiment: The top row shows the loss $L_{\textup{CFM}}$ (left) and the $W_2$ distance to the target distribution (right) for the different optimization methods;
      the curves show the median over 7 runs, while the shaded areas show the region between the 10\% and 90\% percentiles;
      note that the conditional flow matching loss differs from the pure regression loss by a constant; hence, it cannot be expected to vanish;
      the second and third rows show scatter plots of generated samples for each optimizer;
      each panel overlays generated samples (blue) with the target distribution (gold) for visual comparison.
    }
    \label{fig:8gmm}
\end{figure}
\begin{figure}
    \centering
    \includegraphics{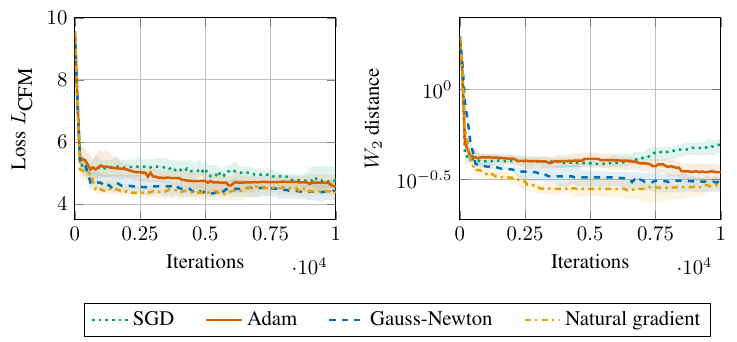}
    \includegraphics{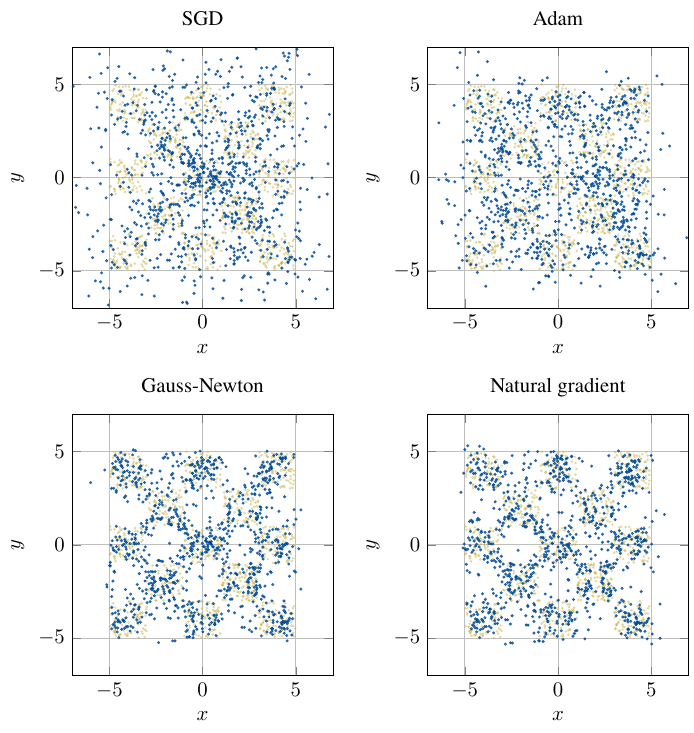}
    \caption{
      Checkerboard experiment: Top row shows the loss $L_{\textup{CFM}}$ (left) and the $W_2$ distance to the target distribution (right) for the different optimization methods;
      the curves show the median over 7 runs, while the shaded areas show the region between the 10\% and 90\% percentiles;
      note that the conditional flow matching loss differs from the pure regression loss by a constant; hence, it cannot be expected to vanish;
      the second and third rows show scatter plots of generated samples for each optimizer;
      each panel overlays generated samples (blue) with the target distribution (gold) for visual comparison.
    }
    \label{fig:checkerboard}
\end{figure}
\renewcommand{\arraystretch}{1.2}
\begin{table}
\centering
\small
\renewcommand{\arraystretch}{1.3}
\begin{tabular}{llcccccc}
\toprule
 &  & \multicolumn{6}{c}{Optimizer} \\
\cmidrule(lr){3-8}
Problem & Metric & SGD & Adam & GN & GN-LS & NG & NG-LS \\
\midrule
\multirow{2}{*}{Isotropic GMM} & $W_2$ & 0.481 & 0.482 & 0.273 & 0.351 & \textbf{0.167} & 0.403 \\
 & TV & 0.346 & 0.340 & 0.197 & 0.289 & \textbf{0.191} & 0.322 \\
\multirow{2}{*}{Anisotropic GMM} & $W_2$ & 0.647 & 0.507 & 0.466 & 0.391 & \textbf{0.135} & 0.502 \\
 & TV & 0.589 & 0.465 & \textbf{0.182} & 0.443 & 0.207 & 0.542 \\
\multirow{2}{*}{8-GMM} & $W_2$ & 0.597 & 0.261 & 0.309 & \textbf{0.202} & 0.268 & 0.240 \\
 & TV & 0.891 & 0.476 & \textbf{0.212} & 0.463 & 0.240 & 0.536 \\
\multirow{2}{*}{Checkerboard} & $W_2$ & 0.352 & 0.272 & 0.321 & \textbf{0.206} & 0.305 & 0.271 \\
 & TV & 0.493 & 0.422 & \textbf{0.374} & 0.388 & 0.387 & 0.443 \\
\multirow{2}{*}{Moons 2D} & $W_2$ & 0.178 & 0.141 & \textbf{0.078} & 0.155 & 0.083 & 0.165 \\
 & TV & 0.690 & 0.618 & \textbf{0.532} & 0.681 & 0.606 & 0.684 \\
\bottomrule
\end{tabular}
\caption{Final $W_2$ and TV distances for each optimizer across all problems after a fixed wall-clock budget; best values per metric per problem are shown in bold; LS indicates the use of a linesearch.}
\label{tab:final-results}
\end{table}
\FloatBarrier

\appendix
\section{Carath\'{e}odory Theory and Method of Characteristics}
We begin by recalling a classic well-posedness theorem.
\begin{theorem}
[Carath\'{e}odory existence theorem] \label{app:them:caratheodory}
Consider a velocity field $v_\bullet$ such that there are functions
$\alpha,\beta\in L^1_T$ satisfying
\begin{equ}
    \lVert v_t(0) \rVert \le \alpha_t, \quad \lVert v_t(x)-v_t(y) \rVert \le \beta_t \lVert x -y\rVert \quad \text{for all } t\in[0,1], x, y\in\R^d.
\end{equ}
Then the flow $\varphi_{t}^{v}$, given by
\begin{equation}
	\partial_{t}\varphi_{t}^{v}(x) =v_{t}(\varphi_{t}^{v}(x)) \quad \text{and } \varphi_{0}^{v}(x) =x,
\end{equation}
for all $x\in\R^{d}$, $t\in[0,1]$ exists and is unique.
Further, we have
\begin{equation}
	\lVert\varphi_{t}(x)-\varphi_{t}(y)\rVert\le\lVert x-y\rVert e^{\int_{0}^{t}\beta_{s}\dif s}
\end{equation}
for all $x,y\in\R^{d}$, $t\in[0,1]$.
\end{theorem}
The proof can be found in \cite[Lemma 8.1.4]{ambrosio2008gradient}.
Additionally, the following bounds will be useful.
\begin{lemma}\label{app:lem:flow-grow}
Consider the setting of the previous theorem. Then, the following bounds hold:
\begin{enumerate}
    \item For all $x\in\R^{d}$, $t\in[0,1]$ it holds that
    \begin{equation}
        \lVert\varphi_{t}(x)-x\rVert
        \le \lVert \alpha \rVert_{L^1([0,t])} e^{\lVert \beta \rVert_{L^1([0,t])}} + \lVert x \rVert \left(e^{\lVert \beta \rVert_{L^1([0,t])}} -1  \right).
    \end{equation}
    \item For all $x\in\R^{d}$, $t\in[0,1]$ it holds that
    \begin{equation}
        \lVert\varphi_{t}(x)\rVert
        \le \lVert \alpha \rVert_{L^1([0,t])} e^{\lVert \beta \rVert_{L^1([0,t])}} + \lVert x \rVert \cdot e^{\lVert \beta \rVert_{L^1([0,t])}}.
    \end{equation}
\end{enumerate}
\end{lemma}
\begin{proof}
First, observe that for any $t\in[0,1]$ and $x\in\R^d$ it holds that
\begin{equation}
    \lVert v_t(x) \rVert \le \lVert v_t(y) - v_t(0) \rVert + \lVert v_t(0) \rVert \le \alpha_t + \beta_t \lVert y \rVert.
\end{equation}
By the definition of the flow, we have $\varphi_t(x) = x + \int_0^t v_s(\varphi_s(x)) \dif s$.
\begin{enumerate}
    \item[2.] We start by showing the second statement, for which we take the norm of the integral equation and obtain
    \begin{equs}
        \lVert \varphi_t(x) \rVert &\le \lVert x \rVert + \int_0^t \lVert v_s(\varphi_s(x)) \rVert \dif s \\
        &\le \lVert x \rVert + \int_0^t \left( \alpha_s + \beta_s \lVert \varphi_s(x) \rVert \right) \dif s \\
        &= \left( \lVert x \rVert + \lVert \alpha \rVert_{L^1([0,t])} \right) + \int_0^t \beta_s \lVert \varphi_s(x) \rVert \dif s.
    \end{equs}
    Now Gr\"onwall's inequality yields
    \begin{equs}\label{app:lem:proof-help}
        \lVert \varphi_t(x) \rVert &\le \left( \lVert x \rVert + \int_0^t \alpha_s ds \right) \exp\left( \int_0^t \beta_s \dif s \right) \\
        &= \left( \lVert x \rVert + \lVert \alpha \rVert_{L^1([0,t])} \right) e^{\lVert \beta \rVert_{L^1([0,t])}} \\
        &= \lVert \alpha \rVert_{L^1([0,t])} e^{\lVert \beta \rVert_{L^1([0,t])}} + \lVert x \rVert \cdot e^{\lVert \beta \rVert_{L^1([0,t])}}.
    \end{equs}
    \item To prove the first statement, we estimate
    \begin{equs}
        \lVert \varphi_t(x) - x \rVert &= \left\lVert \int_0^t v_s(\varphi_s(x)) \dif s \right\rVert \\
        &\le \int_0^t \lVert v_s(\varphi_s(x)) \rVert \dif s \\
        &\le \int_0^t \left( \alpha_s + \beta_s \lVert \varphi_s(x) \rVert \right) \dif s.
    \end{equs}
    Substituting the bound \eqref{app:lem:proof-help} for $\lVert \varphi_s(x) \rVert$ we get
    \begin{equs}
        \lVert \varphi_t(x) - x \rVert &\le \int_0^t \alpha_s \dif s + \int_0^t \beta_s \left( \lVert \alpha \rVert_{L^1([0,s])} e^{\lVert \beta \rVert_{L^1([0,s])}} + \lVert x \rVert e^{\lVert \beta \rVert_{L^1([0,s])}} \right) \dif s.
    \end{equs}
    Since $\lVert \alpha \rVert_{L^1([0,s])} \le \lVert \alpha \rVert_{L^1([0,t])}$ for all $s \in [0,t]$, we can upper bound the integral by pulling this term out
    \begin{equs}
        \lVert \varphi_t(x) - x \rVert &\le \lVert \alpha \rVert_{L^1([0,t])} + \left( \lVert \alpha \rVert_{L^1([0,t])} + \lVert x \rVert \right) \int_0^t \beta_s e^{\lVert \beta \rVert_{L^1([0,s])}} \dif s.
    \end{equs}
    For almost every $s\in[0,1]$ the chain rule for absolutely continuous functions gives
    $\beta_s e^{\lVert \beta \rVert_{L^1([0,s])}} = \frac{d}{ds} e^{\lVert \beta \rVert_{L^1([0,s])}}$.
    Hence, we obtain $e^{\lVert \beta \rVert_{L^1([0,t])}} - 1$ and consequently
    \begin{equs}
        \lVert \varphi_t(x) - x \rVert &\le \lVert \alpha \rVert_{L^1([0,t])} + \left( \lVert \alpha \rVert_{L^1([0,t])} + \lVert x \rVert \right) \left( e^{\lVert \beta \rVert_{L^1([0,t])}} - 1 \right) \\
        &= \lVert \alpha \rVert_{L^1([0,t])} e^{\lVert \beta \rVert_{L^1([0,t])}} + \lVert x \rVert \left( e^{\lVert \beta \rVert_{L^1([0,t])}} - 1 \right).
    \end{equs}
    This concludes the proof of the first statement.
\end{enumerate}
\end{proof}
\begin{lemma}[Method of characteristics]
Consider a velocity field $v\in \V$ and the corresponding flow $\varphi_{\bullet}^{v}$ as well as an initial probability measure $\rho_{0}\in\mathcal{P}_{2}(\R^{d})$ with finite second moments.
Then $\rho_{t}\coloneqq(\varphi_{t}^{v})_{\sharp}\rho_{0}$ is a distributional solution to the continuity equation
\begin{equation}
    \partial_{t}\rho_{t}+\nabla\cdot(v_{t}\rho_{t})=0
\end{equation}
meaning that for any test function $\phi\in C_{\textup{c}}^{\infty}((0,1)\times\R^{d})$ it holds that
\begin{equation}
    \int_{0}^{1}\int_{\R^{d}}\left(\partial_{t}\phi_{t}(x)+v_{t}(x)\cdot\nabla\phi_{t}(x)\right)\dif\rho_{t}(x)\dif t=0.
\end{equation}
\end{lemma}
The proof can be found in \cite[Lemma 8.1.6.]{ambrosio2008gradient} or \cite[Lemma 4.1.1.]{figalli2021invitation}.
\begin{lemma}[Stability of flows]\label{app:lem:stability}
Consider two velocity fields $v_\bullet$ and $u_\bullet$ satisfying
\begin{equs}
    \lVert v_{t}(0)\rVert & \le\alpha_{t}\quad\text{and }\lVert v_{t}(x)-v_{t}(y)\rVert\le\beta_{t} \lVert x -y \rVert\\
    \lVert u_{t}(0)\rVert & \le\alpha_{t}^{\prime}\quad\text{and }\lVert u_{t}(x)-u_{t}(y)\rVert\le\beta_{t}^{\prime} \lVert x -y \rVert
\end{equs}
 and by $\varphi^{v+hu}$ we denote the flow induced by the velocity
$v+hu$ for $h\in\R$. Then for any $\delta>0$ there are constants
$c_{1},c_{2}\ge0$ such that
\begin{equ}
    \lVert\varphi_{t}^{v+hu}(x)-\varphi_{t}^{v}(x)\rVert\le
    \lvert h \rvert (c_{1}+c_{2}\lVert x\rVert)\quad\text{for all }h\in(-\delta,\delta),x\in\R^{d},t\in[0,1],
\end{equ}
 where we can choose
\begin{equs}
    c_{1} & = \lVert \beta'\rVert_{L^1_T}\left(\lVert\alpha\rVert_{L^1_T}+\delta\lVert\alpha'\rVert_{L^1_T}\right) e^{2\lVert\beta\rVert_{L^1_T}+\delta\lVert\beta'\rVert_{L^1_T}}
    \\ & \qquad  + \lVert \alpha' \rVert_{L^1_T}e^{\lVert\beta\rVert_{L^1_T}},
    \quad\text{and }
    \\
    c_{2} & =\lVert\beta^{\prime}\rVert_{L^1_T} e^{2\lVert\beta\rVert_{L^1_T}+\delta\lVert\beta^{\prime}\rVert_{L^1_T}}.
\end{equs}
\end{lemma}
\begin{proof}
Assume without loss of generality that $h>0$, as the other case is then obtained by symmetry.
We fix $x\in\R^{d}$ and set $x_{t}^{h}\coloneqq\varphi_{t}^{v+hu}(x)$.
Then, we have
\begin{equs}
    \lVert x_{t}^{h}-x_{t}^{0}\rVert & =\left\lVert \int_{0}^{t}(v_{s}(x_{s}^{h})-v_{s}(x_{s}^{0}))\dif s-h\int_{0}^{t}u_{s}(x_{s}^{h})\dif s\right\rVert \\
    & \le\int_{0}^{t}\lVert v_{s}(x_{s}^{h})-v_{s}(x_{s}^{0})\rVert\dif s+h\int_{0}^{t}\lVert u_{s}(x_{s}^{h})\rVert\dif s\\
    & \le\int_{0}^{t}\beta_{s}\lVert x_{s}^{h}-x_{s}^{0}\rVert\dif s+h\int_{0}^{t}\lVert u_{s}(x_{s}^{h})\rVert\dif s.
\end{equs}
To bound $\lVert u_{s}\rVert$, we use the notation $v_{t}^{h}\coloneqq v+hu$,
which satisfies the boundedness and Lipschitz conditions with $\tilde{\alpha}=\alpha+h\alpha'$
and $\tilde{\beta}=\beta+h\beta'$ and apply Gr\"{o}nwall's inequality
to
\begin{equs}
    \lVert x_{t}^{h}\rVert & =\left\lVert x+\int_{0}^{t}v_{s}^{h}(x_{s}^{h})\dif s\right\rVert
    \\ & \le\lVert x \rVert + \int_{0}^{t}\lVert v_{s}^{h}(x_{s}^{h})\rVert\dif s
    \\ & \le\lVert x \rVert + \int_{0}^{t}\left(\tilde{\alpha}_{s}+\tilde{\beta}_{s}\lVert x_{s}^{h}\rVert\right)\dif s,
\end{equs}
which yields
\begin{equs}
\lVert x_{t}^{h}\rVert & \le
\left(\lVert x\rVert+\int_{0}^{t}\tilde{\alpha}_{s}\dif s \right)\cdot\exp\left(\int_{0}^{t}\tilde{\beta}_{s}\dif s\right)\\
 & \le
 \left(\lVert x\rVert+\lVert\alpha\rVert_{L^1_T}+h\lVert\alpha'\rVert_{L^1_T}\right)\cdot\exp\left(\lVert\beta\rVert_{L^1_T}+h\lVert\beta'\rVert_{L^1_T}\right).
\end{equs}
This gives us
\begin{equation*}
\begin{split}
    \lVert u_{s}(x_{s}^{h})\rVert & \le\alpha'_{s}+\beta'_{s}\lVert x_{s}^{h}\rVert\\
    & \le\alpha'_{s}+\beta'_{s}\left(\lVert x\rVert+\lVert\alpha\rVert_{L^1_T}+h\lVert\alpha'\rVert_{L^1_T}\right)
    \\
    & \qquad \times \exp\left(\lVert\beta\rVert_{L^1_T}+h\lVert\beta'\rVert_{L^1_T}\right)\\
    & \eqqcolon a_{s}.
\end{split}
\end{equation*}
Plugging back in yields
\begin{equation*}
\lVert x_{t}^{h}-x_{t}^{0}\rVert\le\int_{0}^{t}\beta_{s}\lVert x_{s}^{h}-x_{s}^{0}\rVert\dif s+h\int_{0}^{t}a_{s}\dif s.
\end{equation*}
An application of Gr\"{o}nwall's inequality yields the claim.
\end{proof}
\begin{lemma}[Sensitivity of ODEs]\label{app:lem:sensitivity}
Consider two velocity fields $v,u\in\V$
and denote the flow induced by the velocity $v+hu$  by $\varphi^{v+hu}$ for $h\in\R$.
Then, the limit $\psi_{t}^{u}(x)\coloneqq\ddh \varphi_{t}^{v+hu}(x)$ exists for all $x\in\R^{d}$, $t\in[0,1]$ and is the unique solution to the equation
\begin{equation}
	\partial_{t}\psi_{t}^{u}(x)=Dv_{t}(\varphi_{t}^{v}(x))\cdot\psi_{t}^{u}(x)+u_{t}(\varphi_{t}^{v}(x)), \quad \psi_0(x)=0.
\end{equation}
\end{lemma}
\begin{proof}
Let us denote the spatial Lipschitz constants of $v_\bullet$ and $u_\bullet$ by $\beta, \beta'\in L^1_T$, respectively.
We fix the initial condition $x\in\R^{d}$ and consider the difference
quotients $\Delta_{t}^{h}\coloneq\frac{\varphi_{t}^{v+hu}(x)-\varphi_{t}^{v}(x)}{h}$.
Further, let $\psi$ be the unique solution of
\begin{equation*}
\partial_{t}\psi_{t}=Dv_{t}(\varphi_{t}^{v}(x))\cdot\psi_{t}+u_{t}(\varphi_{t}^{v}(x))\quad\text{and }\psi_{0}=0.
\end{equation*}
Note that there is a unique solution, as this is a linear ordinary differential equation and $Dv_{t}(\varphi_{t}^{v}(x))$ is bounded by $\beta_{t}$, where $\beta\in L^1_T$.
We denote the difference between the difference quotients and the limiting candidate by $e_{t}^{h}\coloneqq\Delta_{t}^{h}-\psi_{t}$.
By using the differential equations that both $\Delta_{t}^{h}$ as well as $\psi_{t}$ satisfy, we obtain
\begin{equs}
    \lVert e_{t}^{h}\rVert & =\left\lVert \int_{0}^{t}\left(\frac{v_{s}(x_{s}^{h})-v_{s}(x_{s}^{0})}{h}-Dv_{s}(x_{s}^{0})\psi_{s}\right)\dif s+\int_{0}^{t}(u_{s}(x_{s}^{h})-u_{s}(x_{s}^{0}))\dif s\right\rVert \\
    & \le\int_{0}^{t}\left\lVert \frac{v_{s}(x_{s}^{h})-v_{s}(x_{s}^{0})}{h}-Dv_{s}(x_{s}^{0})\psi_{s}\right\rVert \dif s+\int_{0}^{t}\lVert u_{s}(x_{s}^{h})-u_{s}(x_{s}^{0})\rVert\dif s\\
    & \le\int_{0}^{t}\left\lVert \frac{v_{s}(x_{s}^{h})-v_{s}(x_{s}^{0})}{h}-Dv_{s}(x_{s}^{0})\psi_{s}\right\rVert \dif s+\int_{0}^{t}\beta'_{s}\lVert x_{s}^{h}-x_{s}^{0}\rVert\dif s\\
    & \le\int_{0}^{t}\left\lVert \frac{v_{s}(x_{s}^{h})-v_{s}(x_{s}^{0})}{h}-Dv_{s}(x_{s}^{0})\psi_{s}\right\rVert \dif s+h(c_{1}+c_{2}\lVert x\rVert)\int_{0}^{t}\beta'_{s}\dif s
\end{equs}
By \Cref{app:lem:flow-grow} we can find a compact set $K$ such that $x_{s}^{h}\in K$
for all $h\in(-\delta,\delta)$, $s\in[0,1]$.
Over $K$ we obtain
\begin{equation*}
    \lVert v_{s}(y)-v_{s}(x)-Dv_{s}(x)(y-x)\rVert\le\omega_{Dv_s}^{K}(\lVert x-y\rVert)\lVert x-y\rVert
\end{equation*}
for $x,y\in K$, where $\omega_{Dv_s}^{K}$ denotes the modulus of continuity.
Now, we estimate
\begin{equs}
    \left\lVert \frac{v_{s}(x_{s}^{h})-v_{s}(x_{s}^{0})}{h}-Dv_{s}(x_{s}^{0})\psi_{s}\right\rVert  &
    \le\left\lVert \frac{Dv_{s}(x_{s}^{0})(x_{s}^{h}-x_{s}^{0})}{h}-Dv_{s}(x_{s}^{0})\psi_{s}\right\rVert
    \\ & \qquad+\frac{\omega_{Dv_{s}}^{K}(\lVert x_{s}^{0}-x_{s}^{h}\rVert)\lVert x_{s}^{0}-x_{s}^{h}\rVert}{h}\\
    & \le\lVert Dv_{s}(x_{s}^{0})\rVert\left\lVert \Delta_{s}^{h}-\psi_{s}\right\rVert
    \\ & \qquad+\frac{\omega_{Dv_{s}}^{K}((c_{1}+c_{2}\lVert x\rVert)h)(c_{1}+c_{2}\lVert x\rVert)h}{h}\\
    & =\lVert Dv_{s}(x_{s}^{0})\rVert\left\lVert e_{s}^{h}\right\rVert
    \\ & \qquad+\omega_{Dv_{s}}^{K}((c_{1}+c_{2}\lVert x\rVert)h)(c_{1}+c_{2}\lVert x\rVert) \\
    & \le \lVert Dv_{s}(x_{s}^{0})\rVert\left\lVert e_{s}^{h}\right\rVert+c_3\omega_{Dv_{s}}^{K}(c_4h).
\end{equs}
This gives
\begin{equ}
    \lVert e_{t}^{h}\rVert\le\int_{0}^{t}\beta_{s}\lVert e_{s}^{h}\rVert\dif s + f(h),
\end{equ}
where
\begin{equ}
    f(h)=c_3\int_0^t \omega_{Dv_s}^K(c_4 h) \dif s + c_5 h \int_0^t \beta'_s \dif s.
\end{equ}
Gr\"{o}nwall's inequality yields
\begin{equation*}
    \lVert e_{t}^{h}\rVert\le f(h)e^{\int_{0}^{t}\beta_{s}\dif s},
\end{equation*}
hence it remains to show that $\lim_{h\to0} f(h)=0$.
The second term of $f$ clearly vanishes as $h\to0$.
To see that the second term vanishes, we first notice that for any $r>0$ we obtain
\begin{equ}
    \omega_{Dv_s}^K(r) = \sup\{ \lVert Dv_s(x) - Dv_s(y) \rVert : x,y\in K, \lVert x-y\rVert \le r \} \le 2 \lVert Dv_s \rVert_{\infty} \le 2\beta_s,
\end{equ}
which is integrable in time.
As $Dv_s$ is continuous and $K$ is compact, pointwise in time, we have $\lim_{h\to0} \omega_{Dv_s}^K(c_4 h)=0$ and by the dominated convergence theorem, it holds that
\begin{equ}
    \lim_{h\to0} \int_0^t \omega_{Dv_s}^K(c_4 h) \dif s = 0,
\end{equ}
which finishes the proof.
\end{proof}
\section{Fisher--Rao Gradient Flows and Geodesics}
Here, we derive the expressions of the Fisher--Rao gradients of the
KL-divergence and reverse KL-divergence. We show that they point into
the directions of the $m$- and $e$-geodesics, respectively.
\begin{lemma}
[$m$- and $e$-logarithmic maps]
For any two probability measures $\P_0$ and $\P_1$ we have
\begin{equation*}
    \left.\frac{\D}{\D t}\right|_{t=0}\mathbb{P}_{t}^{(m)}=\mathbb{P}_{1}-\mathbb{P}_{0}.
\end{equation*}
If further $\P_0\ll\P_1$ and $D_{\KL}(\P_0, \P_1)<+\infty$, then
\begin{equation*}
    \left.\frac{\D}{\D t}\right|_{t=0}\mathbb{P}_{t}^{(e)}=\left(D_{\textup{KL}}(\mathbb{P}_{0},\mathbb{P}_{1})-\ln\left(\frac{\D\mathbb{P}_{0}}{\D\mathbb{P}_{1}}\right)\right)\mathbb{P}_{0}.
\end{equation*}
\end{lemma}
\begin{proof}
The first statement is obtained by differentiating the linear interpolation.
For the second statement, we fix a reference measure $\mathbb{Q}$,
set $\rho_{t}=\frac{\D\mathbb{P}_{t}^{(e)}}{\D\mathbb{Q}}\propto\rho_{0}^{1-t}\rho_{1}^{t},$
denote the normalization constant by $\psi_{t}=\int\rho_{0}^{1-t}\rho_{1}^{t}\dif x$.
We note that $\psi_{0}=1$ and
\begin{equation*}
    \frac{\D}{\D t}\rho_{0}^{1-t}\rho_{1}^{t}=\frac{\D}{\D t}\rho_{0}\exp\left(t\ln\left(\frac{\rho_{1}}{\rho_{0}}\right)\right)=\rho_{0}^{1-t}\rho_{1}^{t}\ln\left(\frac{\rho_{1}}{\rho_{0}}\right).
\end{equation*}
 Further, we have
\begin{equs}
    \left.\frac{\D}{\D t}\right|_{t=0}\psi_{t}
    =\int\left.\frac{\D}{\D t}\right|_{t=0}\rho_{0}^{1-t}\rho_{1}^{t}\dif x
    =\int\rho_{0}^{1-t}\rho_{1}^{t}\ln\left(\frac{\rho_{1}}{\rho_{0}}\right)\dif x|_{t=0}
    =-D_{\textup{KL}}(\mathbb{P}_{0},\mathbb{P}_{1}).
\end{equs}
Overall, we compute
\begin{equs}
    \frac{\D}{\D t}\rho_{t}|_{t=0} & =\frac{\D}{\D t}\left(\frac{\rho_{0}^{1-t}\rho_{1}^{t}}{\psi_{t}}\right)|_{t=0}\\
    & =\frac{\psi_{0}\frac{\D}{\D t}\rho_{0}^{1-t}\rho_{1}^{t}|_{t=0}-\rho_{0}\frac{\D}{\D t}\psi_{t}|_{t=0}}{\psi_{0}^{2}}\\
    & =\rho_{0}\ln\left(\frac{\rho_{1}}{\rho_{0}}\right)+\rho_{0}D_{\textup{KL}}(\mathbb{P}_{0},\mathbb{P}_{1})\\
    & =\left(D_{\textup{KL}}(\mathbb{P}_{0},\mathbb{P}_{1})-\ln\left(\frac{\rho_{0}}{\rho_{1}}\right)\right)\rho_{0}.
\end{equs}
\end{proof}
\begin{theorem}[Fisher--Rao gradient flows and geodesics]
Consider two probability measures $\P$ and $\P^\star$.
If $D_{\KL}(\P^\star, \P)<+\infty$ it holds that
\begin{equation}
\begin{split}
	\nabla_{\mathbb{P}}^{\textup{FR}}D_{\textup{KL}}(\mathbb{P}^{\star},\mathbb{P}) & =\mathbb{P}-\mathbb{P}^{\star},
\end{split}
\end{equation}
in the sense that
\begin{equ}
    \ddh D_{\KL}(\P^\star, \P+h\Xi) = g^{\FR}_{\P}(\P-\P^\star, \Xi)
\end{equ}
for all $\Xi$ such that $D_{\KL}(\P^\star, \P+h\Xi)<+\infty$ for $\lvert h \rvert$ small enough.
Further, if $D_{\KL}(\P, \P^\star)<+\infty$ it holds that
\begin{equation}
    \nabla_{\mathbb{P}}^{\textup{FR}}D_{\textup{KL}}(\mathbb{P},\mathbb{P}^{\star})=\left(\ln\left(\frac{\D\mathbb{P}}{\D\mathbb{P}^{\star}}\right)-D_{\textup{KL}}(\mathbb{P},\mathbb{P}^{\star})\right)\mathbb{P}
\end{equation}
in the sense that
\begin{equ}
    \ddh D_{\KL}(\P+h\Xi, \P^\star) = g^{\FR}_{\P}\left(\left(\ln\left(\frac{\D\mathbb{P}}{\D\mathbb{P}^{\star}}\right)-D_{\textup{KL}}(\mathbb{P},\mathbb{P}^{\star})\right)\mathbb{P}, \Xi\right)
\end{equ}
for all $\Xi\ll\P$ such that $D_{\KL}(\P+h\Xi, \P^\star)<+\infty$ for $\lvert h \rvert$ small enough. 
\end{theorem}
\begin{proof}
For the finite-dimensional case, see, for example, \cite{datar2025convergence}.
Here, we provide the proof for general Polish spaces.
To this end, we compute the first variation of $D_{\textup{KL}}(\P^{\star},\P)$ into a direction $\Xi$ such that $D_{\KL}(\P^{\star},\P+h\Xi)$ is finite for $\lvert h \rvert$ sufficiently small.
Note that the finiteness of the Kullback--Leibler divergence together with the concavity of the logarithm allows us to use the dominated convergence theorem, and we obtain
\begin{equs}
    \frac{\D}{\D h}D_{\textup{KL}}(\mathbb{P}^{\star},\P+h\Xi)|_{h=0}
    & =\frac{\D}{\D h}\mathbb{E}_{\mathbb{P}^{\star}}\left[\ln\left(\frac{\D\P^{\star}}{\D(\P+h\Xi)}\right)\right]\Big|_{h=0}\\
    & =\frac{\D}{\D h}\left(-\mathbb{E}_{\mathbb{P}^{\star}}\left[\ln\left(\frac{\D(\P+h\Xi)}{\D\P^{\star}}\right)\right]\right)\Big|_{h=0}\\
    & =-\mathbb{E}_{\mathbb{P}^{\star}}\left[\frac{\D}{\D h}\ln\left(\frac{\D(\P+h\Xi)}{\D\P^{\star}}\right)\Big|_{h=0}\right]\\
    & =-\mathbb{E}_{\mathbb{P}^{\star}}\left[\frac{\frac{\D}{\D h}\frac{\D(\P+h\Xi)}{\D\P^{\star}}}{\frac{\D(\P+h\Xi)}{\D\P^{\star}}}\Big|_{h=0}\right]\\
    & =-\mathbb{E}_{\mathbb{P}^{\star}}\left[\frac{\frac{\D\Xi}{\D\P^{\star}}}{\frac{\D\P}{\D\P^{\star}}}\right]\\
    & =-\mathbb{E}_{\mathbb{P}^{\star}}\left[\frac{\D\Xi}{\D\P}\right]
\end{equs}
Note that $Z=\P-\P^{\star}$ is contained in the tangent space and
we have
\begin{equs}
    g_{\P}^{\textup{FR}}(\Xi,Z) & =\mathbb{E}_{\mathbb{P}}\left[\frac{\mathrm{d}\Xi}{\mathrm{d}\P}\cdot\frac{\mathrm{d}Z}{\mathrm{d}\P}\right]=\mathbb{E}_{\mathbb{P}}\left[\frac{\mathrm{d}\Xi}{\mathrm{d}\P}\right]-\mathbb{E}_{\mathbb{P}}\left[\frac{\mathrm{d}\Xi}{\mathrm{d}\P}\cdot\frac{\mathrm{d}\mathbb{P}^{\star}}{\mathrm{d}\P}\right]=-\mathbb{E}_{\mathbb{P}^{\star}}\left[\frac{\D\Xi}{\D\P}\right].
\end{equs}
For the second part, we fix $\Xi$ with an $L^1(\P)$ density such that $D_{\KL}(\P+h\Xi, \P^\star)$ is finite for $\lvert h \rvert$ small enough.
As the integrand $x\ln x$ is  to compute the first variation
\begin{equs}
    \frac{\D}{\D h}D_{\textup{KL}}(\P+h\Xi,\mathbb{P}^{\star})|_{h=0} & =\frac{\D}{\D h}\mathbb{E}_{\P+h\Xi}\left[\ln\left(\frac{\D(\P+h\Xi)}{\D\P^{\star}}\right)\right]\Big|_{h=0}\\
    & =\mathbb{E}_{\Xi}\left[\ln\left(\frac{\D\P}{\D\P^{\star}}\right)\right]+\mathbb{E}_{\P}\left[\frac{\D}{\D h}\ln\left(\frac{\D(\P+h\Xi)}{\D\P^{\star}}\right)\Big|_{h=0}\right]\\
    & =\mathbb{E}_{\Xi}\left[\ln\left(\frac{\D\P}{\D\P^{\star}}\right)\right]+\mathbb{E}_{\P}\left[\frac{\frac{\D}{\D h}\frac{\D(\P+h\Xi)}{\D\P^{\star}}}{\frac{\D(\P+h\Xi)}{\D\P^{\star}}}\Big|_{h=0}\right]\\
    & =\mathbb{E}_{\Xi}\left[\ln\left(\frac{\D\P}{\D\P^{\star}}\right)\right]+\mathbb{E}_{\P}\left[\frac{\D\Xi}{\D\mathbb{P}}\right]\\
    & =\mathbb{E}_{\Xi}\left[\ln\left(\frac{\D\P}{\D\P^{\star}}\right)\right].
\end{equs}
Note that $Z=(\ln\left(\frac{\D\mathbb{P}}{\D\mathbb{P}^{\star}}\right)-D_{\textup{KL}}(\mathbb{P},\mathbb{P}^{\star}))\mathbb{P}$
is contained in the tangent space, and we have
\begin{equs}
    g_{\mathbb{P}}^{\textup{FR}}(\Xi,Z) & =\mathbb{E}_{\mathbb{P}}\left[\frac{\D\Xi}{\D\mathbb{P}}\cdot\left(\ln\left(\frac{\D\mathbb{P}}{\D\mathbb{P}^{\star}}\right)-D_{\textup{KL}}(\mathbb{P},\mathbb{P}^{\star})\right)\right]\\
    & =\mathbb{E}_{\mathbb{P}}\left[\frac{\D\Xi}{\D\mathbb{P}}\cdot\ln\left(\frac{\D\mathbb{P}}{\D\mathbb{P}^{\star}}\right)\right]-D_{\textup{KL}}(\mathbb{P},\mathbb{P}^{\star})\cdot\mathbb{E}_{\mathbb{P}}\left[\frac{\D\Xi}{\D\mathbb{P}}\right]\\
    & =\mathbb{E}_{\mathbb{P}}\left[\frac{\D\Xi}{\D\mathbb{P}}\cdot\ln\left(\frac{\D\mathbb{P}}{\D\mathbb{P}^{\star}}\right)\right].
\end{equs}
\end{proof}

\endappendix

\section*{Acknowledgements}
We thank Guido Mont\'{u}far for helpful comments on this manuscript. 
The first author acknowledges support from the Max Planck Society through the Research Group ``Stochastic Analysis in the Sciences (SAiS)". This work was supported in part by ERC grant FluCo (Grant Agreement No. 101088488). Funded by the European Union. Views and opinions expressed are, however, those of the author(s) only and do not necessarily reflect those of the European Union or the European Research Council Executive Agency. Neither the European Union nor the granting authority can be held responsible for them.

\small

\bibliographystyle{Martin}
\bibliography{./refs}{}

\end{document}